\documentclass[a4paper,headlines=2.1]{scrartcl}

\usepackage[utf8]{inputenc}
\usepackage{amsmath,amsthm,amssymb}
\usepackage{amsrefs}
\usepackage{mathtools}
\usepackage{hyperref}
\usepackage{bbm}
\usepackage{xcolor}
\usepackage{tabto}
\usepackage{bm}
\usepackage{enumitem}
\usepackage{authblk}
\usepackage{upgreek}
\usepackage{theoremref}
\usepackage{ifthen}

\usepackage{tikz}
\usetikzlibrary{decorations, decorations.markings, shapes.geometric, positioning, arrows} 

\allowdisplaybreaks
\newcommand{\C}{\mathbb{C}}

\newcommand{\R}{\mathbb{R}}

\newcommand{\Z}{\mathbb{Z}}
\newcommand{\N}{\mathbb{N}}
\newcommand{\E}{\mathbb{E}}

\newcommand{\bP}{\mathbb{P}}
\newcommand{\bS}{\mathbb{S}}

\newcommand{\diag}{\operatorname{diag}}

\newcommand{\Cov}{\mathrm{Cov}}

\newcommand{\dist}{\mathrm{dist}}

\usepackage[OT1]{fontenc}
\DeclareFontFamily{OT1}{pzc}{}
\DeclareFontShape{OT1}{pzc}{m}{it}{<-> s * [1.0] pzcmi7t}{}
\DeclareMathAlphabet{\mathpzc}{OT1}{pzc}{m}{it}

\renewcommand{\Re}{\operatorname{Re}}
\renewcommand{\Im}{\operatorname{Im}}
\renewcommand{\mod}{\, \operatorname{mod} \,}

\newcommand{\cs}{\bm{m}}
\newcommand{\cK}{\mathcal{K}}

\newcommand{\ol}{\overline}
\newcommand{\ul}{\underline}
\newcommand{\tr}{\mathrm{tr}}

\renewcommand{\bullet}{{\bm{\cdot}}}

\usepackage{lmodern}

\numberwithin{equation}{section}

\tikzset{
	-<-/.style args={#1 #2 #3}{
		decoration={
			markings,
			mark= at position 0.5 with
			{
				\ifthenelse{#3 = 1}
				{
					\fill[#2] (#1/6.0,0pt) -- (0.5*#1, #1/3.0) -- (-0.5*#1,0pt) -- (0.5*#1, #1/-3.0);   % stealth type
				}
				{
					\ifthenelse{#3 = 2}
					{
						\fill[#2] (#1/-2.0,0pt) -- (0.5*#1, #1/3.0) -- (0.5*#1, #1/-3.0);   % latex type
					}
					{
						\ifthenelse{#3 = 3}
						{
							\draw[semithick, #2]  (0.533*#1,#1/2) -- (-0.433*#1, 0) -- (0.533*#1,-#1/2);  % 40 degree arrow
						}{}
					}
				}
			},
		},
		postaction={decorate}
	},
	-<-/.default={6pt black 1}
}

\newcommand\s{1}
\renewcommand\d{\s cm}
\tikzset{
	blacknode/.style={
		circle,
		thick,
		draw=black,
		fill=gray!50,
		minimum size=0.05*\d,
		scale = \s,
	},
	whitenode/.style={
		circle,
		thick,
		draw=black,
		fill=white!50,
		minimum size=0.05*\d,
		scale = \s,
	},
	noborder/.style={
		circle,
		thick,
		minimum size=0.05*\d,
		scale = \s,
	},
	el/.style = {inner sep=2pt, align=left, sloped, font=\tiny},
	every label/.append style = {font=\tiny},
	position/.style args={#1:#2 from #3}{
		at=(#3.#1), anchor=#1+180, shift=(#1:#2)
	}
}

\newsavebox\thesmashminipage

\theoremstyle{plain}
\newtheorem{theorem}{Theorem}[section]
\newtheorem{corollary}[theorem]{Corollary}
\newtheorem{lemma}[theorem]{Lemma}
\newtheorem{proposition}[theorem]{Proposition}

\theoremstyle{definition}
\newtheorem{definition}[theorem]{Definition}
\newtheorem{remark}[theorem]{Remark}
\newtheorem{assumption}[theorem]{Assumption}

\numberwithin{theorem}{section}

\title{Marchenko-Pastur laws for Daniell smoothed periodograms}
\date{}
\author{Ben Deitmar}
\affil{\small \textit{Department of Mathematical Stochastics, ALU Freiburg \protect\\ Ernst-Zermelo-Str. 1, 79104 Freiburg, Germany \protect\\ E-mail: ben.deitmar@stochastik.uni-freiburg.de}}
\begin{document}
	\thispagestyle{empty}
	\maketitle
	\vspace{-1.2cm}
	
	\begin{abstract}
		Given a sample $X_0,...,X_{n-1}$ from a $d$-dimensional stationary time series $(X_t)_{t \in \Z}$, the most commonly used estimator for the spectral density matrix $F(\theta)$ at a given frequency $\theta \in [0,2\pi)$ is the Daniell smoothed periodogram
		$$S(\theta) = \frac{1}{2m+1} \sum\limits_{j=-m}^m I\Big( \theta + \frac{2\pi j}{n} \Big) \ ,$$
		which is an average over $2m+1$ many periodograms at slightly perturbed frequencies. We prove that the Marchenko-Pastur law holds for the eigenvalues of $S(\theta)$ uniformly in $\theta \in [0,2\pi)$, when $d$ and $m$ grow with $n$ such that $\frac{d}{m} \rightarrow c>0$ and $d\asymp n^{\alpha}$ for some $\alpha \in (0,1)$. This demonstrates that high-dimensional effects can cause $S(\theta)$ to become inconsistent, even when the dimension $d$ is much smaller than the sample size $n$.\\
		\\
		Notably, we do not assume independence of the $d$ components of the time series. The Marchenko-Pastur law thus holds for Daniell smoothed periodograms, even when it does not necessarily hold for sample auto-covariance matrices of the same processes.
	\end{abstract}
	
\section{Introduction}
The spectral density matrix (SDM) of a multivariate stationary time series encodes information about its periodicity and characterizes the covariance structure. It is a central object in the analysis of time series (\cite{Brillinger1981,Brillinger2001,taniguchi2012Book,BrockwellITSF,BrockwellTSTM,hannan2009Book,lutkepohl2005Book,lutkepohl2013Book,wei2019Book}) and its uses include change point detection (\cite{zhang2023spectral,mosaferi2025optimal,casini2024change,guan2019nonparametric}), independence testing (\cite{chang2025statistical,besserve2013statistical,LoubatonRosuelNew,hong1999hypothesis,yajima2009nonparametric}) as well as discriminant analysis (\cite{kakizawa1996discriminant,krafty2016discriminant,sakiyama2004discriminant}). It is well-known (see \cite{Brillinger1981}) that the na\"{i}ve canonical estimator for the spectral density, the periodogram, is not consistent due to its high variance. Instead, one commonly uses the Daniell smoothed periodogram (\cite{Brillinger1981,Brillinger2001,BrockwellITSF,BrockwellTSTM,krafty2016discriminant, FiecasPeriodogram, bohm2008structural, das2021spectral}), which is an average of several periodograms at slightly perturbed frequencies.\\
\\
When the dimension of the time series is comparable to the sample size, the Daniell smoothed periodogram can also become inconsistent, necessitating more complex and situational methods of inference of the SDM in a high-dimensional context (\cite{chang2025statistical, FiecasPeriodogram,li2021adaptive,FiecasSachs,sun2018large}). For high-dimensional independent identically distributed (i.i.d.) data, which does not have time series structure, there is a substantial amount of literature on the estimation of the underlying population covariance matrix $\Sigma$ from finite samples (see for instance \cite{DingFan_SpikedShrinkage, DonohoJohnstone_SpikedShrinkage, LedoitWolf_OptShrink}). Of fundamental importance is the Marchenko-Pastur law (first shown in \cite{MPOriginal}), which describes the relationship between the eigenvalues of $\Sigma$ and the eigenvalues of the sample covariance matrix. It can be used to infer the distribution of the population eigenvalues (\cite{el2008spectrum, FreeDeconvolution, LedoitWolf1, LedoitWolf2}).
\\
\\
Under assumptions equivalent to the $d$ components of the time series being independent, the Marchenko-Pastur law was shown to hold for (symmetrized) sample auto-covariance matrices (\cite{Aue, StrongMP}) and the Daniell smoothed periodogram as well as the corresponding spectral coherency matrix (\cite{Loubaton}). In \cite{bhattacharjee2016large} it was shown that sample auto-covariance matrices do not necessarily satisfy the Marchenko-Pastur law, if said assumptions are relaxed. The contribution of this paper is that in a quite general setting and without assuming independence of components, we show that the Marchenko-Pastur law holds for the Daniell smoothed periodogram.
\\
\\
While we do not pursue this direction here, our result allows for the application of inference methods developed for sample covariance matrices (\cite{el2008spectrum, FreeDeconvolution, LedoitWolf1, LedoitWolf2}) to the inference of the eigenvalues of the spectral density matrix. From results of \cite{bhattacharjee2016large} it follows that these methods are not applicable to the estimation of auto-covariance matrices, when the components are not independent.

\subsection{Model and notation}
Let $(X_t)_{t \in \Z}$ be a stationary time series with values in $\R^d$. To a given frequency $\theta \in [0,2\pi)$ the $(d \times d)$ spectral density matrix $F(\theta)$ is defined as the Fourier transform of the auto-covariance matrices
\begin{align}\label{Eq_Def_SpectralDensityMatrix}
	& F(\theta) = \sum\limits_{t \in \Z} e^{it\theta} \Cov\big[ X_0,X_t \big] \ .
\end{align}
For a sample $X_0,...,X_{n-1}$ of the time series and a fixed bandwidth $m \in \N$ the Daniell smoothed periodogram is defined as the average
\begin{align}\label{Eq_Def_Daniell}
	& S(\theta) \coloneq \frac{1}{2m+1} \sum\limits_{j=-m}^m I\Big( [\theta]_n + \frac{2\pi j}{n} \Big)
\end{align}
of periodograms
\begin{align}\label{Eq_Def_Periodogram}
	& I(\theta) \coloneq \frac{1}{n} \sum\limits_{t \in \Z} e^{it\theta} \sum\limits_{\substack{k,k' \in \{0,...,n-1\} \\ k'-k=t}} X_k X_k^\top \ ,
\end{align}
where $[\theta]_n \equiv \frac{2\pi}{n} \lfloor\frac{n \theta}{2\pi}+\frac{1}{2}\rfloor$ is the closest point to $\theta$ on the grid $\frac{2\pi}{n} \Z$.
\\
\\
Using the $(n \times n)$ fast Fourier transform matrix
\begin{align}\label{Eq_Def_Vn}
	& V \coloneq \frac{1}{\sqrt{n}} \big( e^{-2\pi i \frac{(k-1)(l-1)}{n}} \big)_{k,l \in \{1,...,n\}}
\end{align}
and the notation $r = r(\theta) = \lfloor\frac{n \theta}{2\pi}+\frac{1}{2}\rfloor$ such that $[\theta]_n \equiv \frac{2\pi r}{n}$,
we can write the Daniell smoothed periodogram as
\begin{align}\label{Eq_DaniellAsProduct}
	& S\Big( \frac{2\pi r}{n} \Big) = \frac{1}{2m+1} \bm{X} V D_r V^* \bm{X}^\top \ ,
\end{align}
where $\bm{X} = [X_0,...,X_{n-1}]$ is the $(d \times n)$ data-matrix with the samples $X_0,...,X_{n-1}$ as columns and $D_r$ is the $(n \times n)$ diagonal matrix of the form
\begin{align}\label{Eq_Def_Dr_imprecise}
	& D_r = \operatorname{diag}\big( 0,...,0,\underbrace{1,....}_{m \times},\hspace{-0.8cm}\overbrace{1}^{(r+1)\text{-th position}}\hspace{-0.8cm},\underbrace{....,1}_{m \times},0,....,0 \big) \in \R^{n \times n} \ .
\end{align}
More precisely, we define $(D_r)_{s,s} = \mathbbm{1}_{\rho_n(r,s-1) \leq m}$ for the modulo-$n$ distance
\begin{align}
	& \rho_n(r,r') \coloneq \dist(r,r'+n\Z) \ .
\end{align}
The following assumption appears in varying forms in \cite{Aue,StrongMP,bhattacharjee2016large,MikoschHeinyLinear}, since it generalizes popular auto regressive (AR) and multivariate auto regressive moving average (VARMA) models.
\begin{assumption}[Linear process]\thlabel{Assumption0_LinearProcess}\
\\
For matrices $(\Psi_k)_{k \in \N_0} \subset \R^{d \times d}$ and a sequence of independent $d$-dimensional innovations $(\eta_t)_{t \in \Z}$ such that each innovation vector $\eta_t$ has independent components with variance one, we assume the process $(X_t)_{t \in \Z}$ to have the form
\begin{align}
	& X_t = \sum\limits_{k=0}^\infty \Psi_k \eta_{t-k} \ .
\end{align}
In order to ensure the convergence of the right hand side as an $L^2$-limit and almost sure limit, we also assume
\begin{align*}
	& \sum\limits_{k=0}^\infty ||\Psi_k||^2 < \infty \ ,
\end{align*}
where $||\cdot||$ denotes the operator norm.
\end{assumption}\
\\
It is easily calculated that the spectral density matrix of $X_t$ as above will have the form
\begin{align}\label{Eq_SpectralDensity_LinearProcess}
	& F(\theta) = \bigg( \sum\limits_{k=0}^\infty e^{-ik\theta} \Psi_k \bigg) \bigg( \sum\limits_{k=0}^\infty e^{-ik\theta} \Psi_k \bigg)^* \ .
\end{align}
For notational convenience, we define
\begin{align}\label{Eq_Def_G}
	& G(\theta) \coloneq \sum\limits_{k=0}^\infty e^{-ik\theta} \Psi_k \in \C^{d \times d} \ .
\end{align}

\begin{remark}[Asymptotics in $n$]\
\\
In random matrix theory, high-dimensional behavior is studied by asymptotic results when the dimension $d$ grows with $n$. Similarly, in analysis of time series in the frequency domain, the bandwidth $m$ of the Daniell smoothed periodogram is assumed to grow with $n$, but still satisfying $m = o(n)$, in order to achieve consistency.
\\
We will thus write $d^{(n)}$ and $m^{(n)}$ to highlight the nature of the dimension and bandwidth as sequences in $n$ when it aids comprehension, while otherwise suppressing this relationship in our notation. Likewise, the process $(X_t)_{t \in \Z}$ and all the objects $F$, $S$, $I$, $V$, $\bm{X}$, $D_r$, $\Psi_k$, $\eta_t$ defined above also depend on $n$, sometimes only through their dependence on $d^{(n)}$, and we will highlight this when necessary by writing $X_t^{(n)}$, $F^{(n)}$, $S^{(n)}$, $I^{(n)}$, $V^{(n)}$, $\bm{X}^{(n)}$, $D_r^{(n)}$, $\Psi_k^{(n)}$ or $\eta_t^{(n)}$.
\end{remark}\
\\
Some further notation we will use is $||\cdot||$ for the operator norm of matrices, and
\begin{align}\label{Eq_Def_ESD}
	& \hat{\mu}_A \coloneq \frac{1}{d} \sum\limits_{j=1}^d \delta_{\lambda_j(A)}
\end{align}
for the empirical spectral distribution (ESD) of a Hermitian $(d \times d)$ matrix $A$. The complex upper half plane $\{z \in \C \mid \Im(z) > 0\}$ will be denoted as $\C^+$. The Stieltjes transform $\cs_\mu$ of a measure $\mu$ on $\R$ is defined as the map
\begin{align}\label{Eq_Def_SteiltjesTransform}
	& \cs_{\mu} : \C^+ \rightarrow \C^+ \ \ ; \ \ z \mapsto \int_\R \frac{1}{\lambda - z} \, d\mu(\lambda) \ .
\end{align}
The symbol $\bm{\eta}$ will stand for the $(d \times n)$ matrix $[\eta_0,...,\eta_{n-1}]$ containing a subset of the innovations $(\eta_t)_{t\in \Z}$ as its columns. For any odd $N \in \N$, the double factorial $N!!$ is defined as $N(N-2)\cdots1$. Unit vectors in $\R^d$ will be written as
\begin{align*}
	& \mathrm{u}_i^{(d)} \coloneq (\underbrace{0,...,0}_{\times(i-1)},1,0,...,0)^\top \in \R^d \ .
\end{align*}

\begin{definition}[Bounded Lipschitz metric]\thlabel{Def_BL_metric}\
\\
For any two probability measures $\mu_1,\mu_2$ on $(\R,\mathcal{B}(\R))$ the Bounded Lipschitz metric $d_{\operatorname{BL}}(\mu_1,\mu_2)$ is defined as
\begin{align}\label{Eq_Def_BL_distance}
	& d_{\operatorname{BL}}(\mu_1,\mu_2) \coloneq \sup\limits_{f \in \operatorname{Lip}_1} \bigg| \int_\R f \, d(\mu_1-\mu_2) \bigg| \ ,
\end{align}
where
\begin{align}\label{Eq_Def_Lip}
	& \operatorname{Lip}_a \coloneq \Big\{ f \in C(\R) \ \Big| \ \sup\limits_{x \in \R} |f(x)| \leq a , \ \sup\limits_{\substack{x,y \in \R \\ x \neq y}} \frac{|f(x)-f(y)|}{|x-y|} \leq a \Big\} \ .
\end{align}
\end{definition}\
\\
It is well known (Theorem 1.12.4 of \cite{wellner2013weak}) that $d_{\operatorname{BL}}$ metricizes weak convergence of probability measures.

\subsection{Marchenko-Pastur law}
The well-known Marchenko-Pastur law (\cite{MPOriginal,silverstein1995strong,StrongMP,silverstein1995analysis}) is formulated for i.i.d. samples as follows (see (1.4) of \cite{silverstein1995strong} or p.556 of \cite{BaiCLT}). Let $d=d(n)$ be a sequence in $\N$ that goes to infinity such that $\frac{d}{n} \xrightarrow{n \rightarrow \infty} c > 0$. For a sequence $(\Sigma_n)_{n \in \N}$ of $(d \times d)$ population covariance matrices that satisfy
\begin{align}\label{Eq_BasicMP_PopulationConvergence}
	& H_n \coloneq \hat{\mu}_{\Sigma_n} \xRightarrow{n \rightarrow \infty} H_\infty
\end{align}
for a probability measure $H_\infty$ on $[0,\infty)$ with compact support, let $\hat{\Sigma}_n = \frac{1}{n} \sum\limits_{j=1}^n Y_{j,n} Y_{j,n}^\top$ be sample covariance matrices, each constructed from i.i.d. centered $d$-dimensional samples $Y_{1,n},...,Y_{n,n}$ with $\Cov[Y_{1,n},Y_{1,n}] = \Sigma_n$.
The almost sure convergence of measures
\begin{align}\label{Eq_BasicMP_PopulationResult}
	& 1 = \bP\Big( \hat{\mu}_{\hat{\Sigma}_n} \xRightarrow{n \rightarrow \infty} \nu_\infty \Big)
\end{align}
holds, where the probability measure $\nu_\infty$ is uniquely defined by $H_\infty$ and $c$ through the following lemma.

\begin{lemma}[Marchenko-Pastur equation]\thlabel{Lemma_MPEquation}\
\\
For every $c>0$ and probability measure $H$ on $[0,\infty)$ with compact support there exists a probability measure $\nu$ on $[0,\infty)$ with compact support that is uniquely defined through the property
\begin{align}\label{Eq_MPEquation}
	& \forall z \in \C^+ : \ \cs_{\nu}(z) = \int_\R \frac{1}{\lambda(1-c-cz\cs_{\nu}(z))-z} \, dH(\lambda)
\end{align}
of its Stieltjes transform $\cs_{\nu}$.

\end{lemma}

\subsection{Independent components and simultaneous diagonalizability}
Previous papers on Marchenko-Pastur laws and spectral CLTs of smoothed periodograms, such as \cite{Loubaton} and \cite{LoubatonRosuelNew}, work under the assumption that the $d$-dimensional process $(X_t)_{t \in \N}$ is Gaussian and that the $d$-many components $\big\{((X_t)_j)_{t \in \Z}\big\}_{j \leq d}$ are independent. The papers \cite{Aue} and \cite{StrongMP} on Marchenko-Pastur laws for sample auto-covariance matrices assume $(X_t)_{t \in \N}$ to be a linear process, where the matrices $(\Psi_k)_{k_0 \in \N}$ are Hermitian and simultaneously diagonalizable, i.e. there exists a unitary matrix $U \in U(d)$ such that $U\Psi_kU^*$ is diagonal for all $k \in \N$. In the Gaussian case, this assumption of simultaneous diagonalizability is equivalent to the process $(U X_t)_{t \in \Z}$ having independent components. For the spectral density matrix $F(\theta)$ these assumptions translate to $F(\theta)$ being diagonal or simultaneously diagonalizable over all $\theta \in [0,2\pi)$.
\\
\\
For sample auto-covariance matrices, the paper \cite{bhattacharjee2016large} derives limiting spectral distributions in the high-dimensional setting even without the assumption of simultaneous diagonalizability. The limiting spectral distributions are described on a theoretical level by their integrals over certain polynomials.
\\
\\
In this paper, the Marchenko-Pastur law for Daniell smoothed periodogram is shown to hold even without assuming that $(\Psi_k)_{k_0 \in \N}$ are simultaneously diagonalizable (or even Hermitian). Since the Marchenko-Pastur law was in \cite{bhattacharjee2016large} shown to be unstable under generalized non simultaneously diagonalizable $(\Psi_k)_{k_0 \in \N}$ for sample auto-covariance matrices, it may seem surprising that it is stable for Daniell smoothed periodograms.

\section{Main results}

\begin{assumption}\thlabel{Assumption_Main}\
\begin{itemize}
	\item[A1)] \textit{Asymptotics of $d$ and $m$:}\\
	Given an $\alpha \in (0,1)$ let $(d^{(n)})_{n \in \N}$ and $(m^{(n)})_{n \in \N}$ be sequences in $\N$ satisfying $d \asymp n^{\alpha} \asymp m$, which means there exists a constant $\cK_1 > 0$ such that
	\begin{align}\label{Eq_Assumption_Asymptotics}
		& \frac{1}{\cK_1} \leq \frac{d}{n^{\alpha}} \leq \cK_1 \ \ \text{ and } \ \ \frac{1}{\cK_1} \leq \frac{m}{n^{\alpha}} \leq \cK_1
	\end{align}
	for all $n \in \N$. Furthermore, assume the existence of a constant $c>0$ such that
	\begin{align}\label{Eq_Assumption_dmConvergence}
		& \frac{d}{m} \xrightarrow{n \rightarrow \infty} c \ .
	\end{align}
	
	\item[A2)] \textit{Limiting long-range dependence:}\\
	For $(\Psi_k)_{k \in \N_0}$ and $(\eta_t)_{ t\in \Z}$ as in Assumption \ref{Assumption0_LinearProcess}, suppose there is a constant $\cK_2>0$ independent of $n$ such that
	\begin{align}\label{Eq_Assumption_LongRandeDependence_cK2}
		& ||\Psi_0^{(n)}|| + \sum\limits_{k=1}^\infty k ||\Psi_k^{(n)}|| \leq \cK_2
	\end{align}
	and a constant $\gamma > 1$ with
	\begin{align}\label{Eq_Assumption_LongRandeDependence_gamma}
		& \forall K \in \N : \ \sum\limits_{k=K}^\infty ||\Psi_k^{(n)}|| \leq \cK_2 \, K^{-\gamma} \ .
	\end{align}
	
	\item[A3)] \textit{Convergence of the population spectral distribution:}\\
	Suppose that for every $\theta \in [0,2\pi)$ there exists a probability distribution $H_\infty(\theta) \neq \delta_0$ with compact support on $[0,\infty)$ such that the weak convergence
	\begin{align}\label{Eq_Assumption_ESD_Convergence}
		& H_n(\theta) \coloneq \hat{\mu}_{F^{(n)}(\theta)} \xRightarrow{n \rightarrow \infty} H_\infty(\theta) 
	\end{align}
	holds for all $\theta \in [0,2\pi)$.
	
	\item[A4)] \textit{Universality requirements:}\\
	Suppose that all moments of the innovations exist and are uniformly bounded in the sense
	\begin{align}\label{Eq_Assumption_InnovationMoments}
		& \forall p \in \N \, \exists C_p>0 : \ \sup\limits_{n \in \N} \sup\limits_{t \in \Z} \max\limits_{i \leq d} \E\big[|(\eta_{t,n})_i|^p\big] \leq C_p \ .
	\end{align}
	And for the constants $\alpha \in (0,1)$ from (\ref{Eq_Assumption_Asymptotics}) and $\gamma>1$ from (\ref{Eq_Assumption_LongRandeDependence_gamma}) we further require
	\begin{align}\label{Eq_Assumption_UniversalityRequirement}
		& \alpha > \frac{1}{2} \ \ \text{ and } \ \ \gamma > \frac{1}{3 \min(\frac{1-\alpha}{2},\alpha-\frac{1}{2})} \ .
	\end{align}
\end{itemize}
\end{assumption}

\begin{remark}[Discussion of assumptions]\
\\
The Marchenko-Pastur law for sample covariance matrices is usually formulated under the assumption that $d$ and $n$ go to infinity simultaneously such that $\frac{d}{n} \rightarrow c$ for some $c>0$. In our case, the bandwidth $m$ takes the role of $n$. In time series analysis, the bandwidth $m$ is usually chosen with $m \asymp n^{\alpha}$ for some $\alpha \in (0,1)$, which leads us directly to (A1).\\
\\
The papers \cite{Aue,StrongMP,bhattacharjee2016large} examine a linear process model for $(X_t)_{t \in \N}$ and must also limit the long-range dependence by postulations on the behavior of $\Psi_k$ for large $k$. Our (A2) is most similar to
\begin{align*}
	& \sum\limits_{k=1}^\infty k \, \sup\limits_{n \in \N}||\Psi_k^{(n)}|| < \infty \ \ \ \ \text{(see Assumption 2.1 (a) of \cite{Aue})} \ .
\end{align*}
Corresponding assumptions in the other two papers range from much more restrictive (see (A3) of \cite{bhattacharjee2016large}) to somewhat weaker (see (3.10) of \cite{StrongMP}).\\
\\
Assumption (A3) is analogous to (\ref{Eq_BasicMP_PopulationConvergence}), where instead of the convergence of the population ESD, the convergence of the ESD of the spectral density matrix $F(\theta)$ is required at every frequency $\theta \in [0,2\pi)$.\\
\\
If the innovations $(\eta_t)_{t \in \Z}$ are Gaussian, then (A4) is not needed. This rather restrictive assumption is a consequence of universality methods for the Marchenko-Pastur law (in particular Sections 9 and 10 of \cite{Aue} or Section A.3 of \cite{StrongMP}) not being suited to the dimensions of the matrices ($d$ and $m$ in our case) growing slower than the index $n$ over which the asymptotics are taken. This, together with the fact that our Marchenko-Pastur law (Theorem \ref{Thm_MarchenkoPastur_Daniell}) is uniform in $\theta \in [0,2\pi)$, makes the strong assumption $\alpha > \frac{1}{2}$ necessary for our methods.
For weaker versions of the Marchenko-Pastur law, which do not give almost sure convergence, it may be possible to weaken the requirement (\ref{Eq_Assumption_UniversalityRequirement}). The moment assumption (\ref{Eq_Assumption_InnovationMoments}) is likely not optimal and might be adaptable to only require moments up to an exponent depending on the choice of $\alpha$ similar to the observations made in \cite{MikoschHeinyLinear}. We believe these generalizations would exceed the scope of this paper.
\end{remark}\
\\
We are now ready to formulate the main result of this paper.

\begin{theorem}[Marchenko-Pastur law for the Daniell smoothed periodogram]\thlabel{Thm_MarchenkoPastur_Daniell}\
\\
Suppose Assumptions \ref{Assumption0_LinearProcess} and \ref{Assumption_Main} hold, where (A4) may be dropped, if the innovations $(\eta_t)_{t \in \Z}$ are Gaussian. For each frequency $\theta \in [0,2\pi)$ let $\nu_\infty(\theta)$ be the probability measure defined as in Lemma \ref{Lemma_MPEquation} through $H_\infty(\theta)$ and $c$.
\\
\\
With probability one, for all $\theta \in [0,2\pi)$ the weak convergence of measures $$\hat{\mu}_{S(\theta)} \xRightarrow{n \rightarrow \infty} \nu_\infty(\theta)$$ holds.
\end{theorem}
\
\\
The following result plays a role in the proof of Theorem \ref{Thm_MarchenkoPastur_Daniell} and is likely to have broader applications in the field of high-dimensional time series.

\begin{theorem}[Trace moment bound for Gaussian Gram matrices]\thlabel{Thm_TraceMomentBound}\
\\
Let $\bm{Y}$ be a $(d \times M)$ centered real-valued Gaussian matrix. Define the $(d \times d)$ auto-covariance matrices
\begin{align}\label{Eq_Def_AutoCov}
	& A_{s,s'} = \E\big[ \bm{Y}_{\bullet,s} \bm{Y}_{\bullet,s'}^\top \big] \ .
\end{align}
Suppose the symmetric $(M \times M)$-matrix $B = \big( ||A_{s,s'}|| \big)_{s,s' \leq M}$ satisfies
\begin{align}\label{Eq_BNormCondition}
	& ||B|| \leq \kappa
\end{align}
for some $\kappa > 0$, then for all $L\in \N$ the bound
\begin{align}\label{Eq_TraceMomentBound}
	& \E\Big[ \tr\Big( \big( \bm{Y} \bm{Y}^\top \big)^L \Big) \Big] \leq \kappa^L (2L-1)!! \, (d+M)^{L+1}
\end{align}
holds.
\end{theorem}

\begin{remark}[Application of the trace moment bound]\
\\
Theorem \ref{Thm_TraceMomentBound} may be used to bound the effect of errors that are localized in neither columns nor rows. It is the reason why we do not require assumptions on the simultaneous diagonalizability of $(\Psi_k)_{k \in \N_0}$. A standard application would use the Markov bound
\begin{align*}
	& \bP\big( ||\bm{Y}\bm{Y}^\top|| \geq \varepsilon \big) \leq \frac{1}{\varepsilon^L} \E\big[ ||\bm{Y}\bm{Y}^\top||^L \big] \leq \frac{1}{\varepsilon^L} \E\big[ \tr\big( (\bm{Y}\bm{Y}^\top)^L \big) \big]
\end{align*}
and then (\ref{Eq_TraceMomentBound}) to bound the effect of some error $\bm{Y}$. The bound clearly becomes sharper with larger $L \in \N$. By Gershgorin's circle theorem, the assumption (\ref{Eq_BNormCondition}) is satisfied, if
\begin{align*}
	& \forall s \leq M : \ \sum\limits_{s'=1}^M ||A_{s,s'}|| \leq \kappa \ .
\end{align*}
Since $A_{s,s'}$ are the auto-covariance matrices of the columns, this bound may be interpreted as a mixing condition for $d$-dimensional (Gaussian) time series.
\end{remark}

\section{Proof of Theorem \ref{Thm_TraceMomentBound}}
The proof relies heavily on the assumption of Gaussianity and the well-known Wick's formula (\cite{isserlis1918formula,munthe2025short,alberts2018calculus,WienerChaos}) in the formulation of the following lemma.
\begin{lemma}[Wick's formula]\thlabel{Lemma_WicksFormula}\
	\\
	Let $\Pi(2L)$ denote the set of partitions $\bm{\pi}$ of the set $\{1,...,2L\}$, i.e.
	\begin{align}\label{Eq_DefPartitions}
		& \Pi(2L) \coloneq \bigg\{ \bm{\pi} \subset \mathcal{P}(\{1,...,2L\}) \ \bigg| \ \emptyset \notin \bm{\pi} \ , \ \ \bigcup\limits_{A \in \bm{\pi}} A = \{1,...,2L\} \ , \nonumber\\
		& \hspace{6cm} \forall A\neq A' \in \bm{\pi} : \, A \cap A' = \emptyset \bigg\} \ .
	\end{align}
	Further, we call a partition $\bm{\pi} \in \Pi(2L)$ a pairing, if all sets in $\bm{\pi}$ have cardinality $2$. Let
	\begin{align}\label{Eq_DefPairing}
		& \Pi_2(2L) \coloneq \{ \bm{\pi} \in \Pi(2L) \mid \forall A \in \bm{\pi} : \, \#A=2 \}
	\end{align}
	denote the set of pairings of $\{1,...,2L\}$.\\
	\\
	For any random variables $Y_1,...,Y_{2L}$, which jointly follow a centered (possibly degenerate) complex Gaussian distribution, the formula
	\begin{align}\label{Eq_MeanByKumulantPairings}
		& \E\big[ Y_1 \cdots Y_{2L} \big] = \sum\limits_{\bm{\pi} \in \Pi_2(2L)} \prod\limits_{\{a,b\} \in \bm{\pi}} \E[Y_a Y_b]
	\end{align}
	holds.
\end{lemma}\
\\
The simple bound
\begin{align}\label{Eq_TraceBound_Neumann}
	& |\tr(A)| \leq \mathrm{rank}(A) \, ||A|| \ ,
\end{align}
that holds for any square matrix $A$ and follows directly from von Neumann's trace inequality (\cite{mirsky1975trace}) will also be of use.
\\
\\
We now begin the proof of Theorem \ref{Thm_TraceMomentBound}.
\\
\\
\textit{Application of Wick's formula:}\\
We apply Wick's formula to the right hand mean in the expression
\begin{align}\label{Eq_TraceExpansion}
	& \E\Big[\tr\Big( \big( \bm{Y} \bm{Y}^\top \big)^L \Big) \Big] = \sum\limits_{j_1,...,j_L=1}^d \sum\limits_{\substack{s_1,...,s_L=1}}^M \E\Big[ \bm{Y}_{j_1,s_1} \bm{Y}_{j_2,s_1} \cdots \bm{Y}_{j_L,s_L} \bm{Y}_{j_1,s_L} \Big]
\end{align}
that arises from expanding all sums in the matrix products. For notational convenience, we re-order the product from the right hand mean into \begin{align}\label{Eq_ProdNiceOrder_New}
	& \E\Big[ \bm{Y}_{j_1,s_1} \bm{Y}_{j_2,s_2} \cdots \bm{Y}_{j_L,s_L} \times \bm{Y}_{j_2,s_1} \bm{Y}_{j_3,s_2} \cdots \bm{Y}_{j_1,s_L} \Big]
\end{align}
and interpret the indexes
\begin{align}\label{Eq_IndexToEdge}
	& (j_1,s_1),(j_2,s_2),...,(j_L,s_L) \ , \ (j_2,s_1),(j_3,s_2),...,(j_1,s_L)\\
	=: & \hspace{0.4cm} e_1, \hspace{0.8cm} e_2, \hspace{0.4cm} ... \hspace{0.7cm} e_L, \hspace{0.9cm} e_{L+1}, \hspace{0.35cm} e_{L+2}, \hspace{0.1cm} ... \hspace{0.4cm} e_{2L} \nonumber
\end{align}
as edges $e_1,...,e_{2L}$ of a $(2L)$-polygon.
\begin{align}\label{Pic_Polygon1}
	\renewcommand{\s}{1}
	\renewcommand{\d}{1cm}
	& \begin{array}{l} \begin{tikzpicture}
			\node (pol1) [draw, minimum size=5*\d, regular polygon, regular polygon sides=8,
			]  at (0,0) {};
			\foreach \x/\y/\i/\j in {1/2/1/1,2/3/{8},3/4/7/4,4/5/6/7,5/6/5/3,6/7/4/6,7/8/3/2,8/1/2/5}
			\path[auto=left, -]
			(pol1.corner \x)--(pol1.corner \y)
			node[midway, circle, inner sep=0cm, draw=none](e1_\i){$e_ {\j}$};
			\foreach \x/\y/\i/\j in {1/2/1/1,3/4/7/4,5/6/5/3,7/8/3/2}
			\path[auto=left, -]
			(pol1.corner \x)--(pol1.corner \y)
			node [whitenode] (s_\i) at (pol1.corner \x) {$s_\j$}
			node [blacknode] (j_\i) at (pol1.corner \y) {$j_\j$};
	\end{tikzpicture} \end{array}
\end{align}
Wick's formula then for (\ref{Eq_ProdNiceOrder_New}) yields
\begin{align*}
	& \E\Big[ \bm{Y}_{j_1,s_1} \bm{Y}_{j_2,s_2} \cdots \bm{Y}_{j_L,s_L} \times \bm{Y}_{j_2,s_1} \bm{Y}_{j_3,s_2} \cdots \bm{Y}_{j_1,s_L} \Big]\\
	& = \E\Big[ \bm{Y}_{e_1} \bm{Y}_{e_2} \cdots \bm{Y}_{e_{2L}} \Big] = \sum\limits_{\bm{\pi} \in \Pi_2(2L)} \prod\limits_{\{a,b\} \in \bm{\pi}} \E[\bm{Y}_{e_a} \bm{Y}_{e_b}]
\end{align*}
and equality (\ref{Eq_TraceExpansion}) becomes
\begin{align}\label{Eq_TraceBoundStep1}
	\E\Big[\tr\Big( \big( \bm{Y} \bm{Y}^\top \big)^L \Big) \Big] & = \sum\limits_{j_1,...,j_L=1}^d \sum\limits_{\substack{s_1,...,s_L=1}}^M \sum\limits_{\bm{\pi} \in \Pi_2(2L)} \prod\limits_{\{a,b\} \in \bm{\pi}} \E[\bm{Y}_{e_a} \bm{Y}_{e_b}] \nonumber\\
	& = \sum\limits_{\bm{\pi} \in \Pi_2(2L)} \sum\limits_{\substack{s_1,...,s_L=1}}^M \sum\limits_{j_1,...,j_L=1}^d \prod\limits_{\{a,b\} \in \bm{\pi}} \E[\bm{Y}_{e_a} \bm{Y}_{e_b}] \ .
\end{align}
Changing the order of summation to sum over pairings $\bm{\pi} \in \Pi_2(2L)$ first, allows us to examine the expression
\begin{align}\label{Eq_TraceBoundExpression1}
	& \sum\limits_{\substack{s_1,...,s_L=1}}^M \sum\limits_{j_1,...,j_L=1}^d \prod\limits_{\{a,b\} \in \bm{\pi}} \E[\bm{Y}_{e_a} \bm{Y}_{e_b}]
\end{align}
for a fixed pairing $\bm{\pi}$.
\\
\\
\textit{Definition of row- and column-cycles:}\\
In addition to the summation-pairing $\bm{\pi}$, two pairings $\bm{\tau}_{r},\bm{\tau}_{c} \in \Pi_2(2L)$ are inherent to our construction of $e_1,...,e_{2L}$. Define the row-pairing $\bm{\tau}_r \in \Pi_2(2L)$ such that $\{a,b\} \in \bm{\tau}_r$ iff the edges $e_a=(j_{x},s_y)$ and $e_b=(j_{x'},s_{y'})$ have the same $j$-index, i.e. $x=x'$. From (\ref{Eq_IndexToEdge}) it is clear that this is indeed a pairing and it may be formalized by
\begin{align*}
	& \{q,\ol{q}\} \in \bm{\tau}_c \ \Leftrightarrow \ \big(q \leq L < \ol{q} \text{ and } \ol{q} = (q-1 \mod L)+L\big) \ \text{ or }\\
	& \hspace{4cm} \big(\ol{q} \leq L < q \text{ and } q = (\ol{q}-1 \mod L)+L\big) \ .
\end{align*}
Similarly, define the column-pairing $\bm{\tau}_c \in \Pi_2(2L)$ such that $\{a,b\} \in \bm{\tau}_c$ iff the edges $e_a=(j_{x},s_y)$ and $e_b=(j_{x'},s_{y'})$ have the same $s$-index, i.e. $y=y'$. This can be formalized more easily by
\begin{align*}
	& q,\tilde{q} \in \bm{\tau}_r \ \Leftrightarrow \ |q-\tilde{q}| = L \ .
\end{align*}
In the polygon-interpretation (\ref{Pic_Polygon1}), we have $\{a,b\} \in \bm{\tau}_r$, iff $e_a$ and $e_b$ share the same $j$-vertex (gray) and $\{a,b\} \in \bm{\tau}_c$, iff $e_a$ and $e_b$ share the same $s$-vertex (white).
\\
\\
With the two pairings $\bm{\pi}$ and $\bm{\tau}_r$ we define a row-cycle $\xi$ as a subset of $\{1,...,2L\}$ that is minimal with the properties
\begin{align*}
	\xi & \neq \emptyset\\
	\forall a \in \xi \, \forall b \in \{1,...,2L\} & : \ \big( \{a,b\} \in \bm{\pi} \Rightarrow b \in \xi \big)\\
	\forall a \in \xi \, \forall b \in \{1,...,2L\} & : \ \big( \{a,b\} \in \bm{\tau}_r \Rightarrow b \in \xi \big) \ .
\end{align*}
Let $C_r(\bm{\pi})$ denote the set of row-cycles to a given $\pi \in \Pi_2(2L)$. Analogously, we define a column-cycle $\zeta$ as a subset of $\{1,...,2L\}$ that is minimal with the properties
\begin{align*}
	\zeta & \neq \emptyset\\
	\forall a \in \zeta \, \forall b \in \{1,...,2L\} & : \ \big( \{a,b\} \in \bm{\pi} \Rightarrow b \in \zeta \big)\\
	\forall a \in \zeta \, \forall b \in \{1,...,2L\} & : \ \big( \{a,b\} \in \bm{\tau}_c \Rightarrow b \in \zeta \big)
\end{align*}
and write $C_c(\bm{\pi})$ for the set of column-cycles to given $\pi \in \Pi_2(2L)$.\\
\\
\textit{Polygon-interpretation of row- and column-cycles:}\\
In the polygon-interpretation, we imagine $\bm{\pi}$ as a pairing of the edges of the polygon. If we fuse the vertices of paired edges (white to white and gray to gray), each remaining vertex will represent a row-cycle (gray) or a column-cycle (white). As an example, we for $L=4$ draw the pairing
\begin{align*}
	& \bm{\pi}_0 = \big\{ \{1,3\}, \{5,6\}, \{2,7\}, \{4,8\} \big\} \in \Pi_2(8)
\end{align*}
into the polygon with blue connections between the paired edges.
\begin{align}\label{Pic_Polygon2}
	\renewcommand{\s}{1}
	\renewcommand{\d}{1cm}
	& \begin{array}{l} \begin{tikzpicture}
			\node (pol1) [draw, minimum size=7*\d, regular polygon, regular polygon sides=8,
			]  at (0,0) {};
			\foreach \x/\y/\i/\j in {1/2/1/1,2/3/{8},3/4/7/4,4/5/6/7,5/6/5/3,6/7/4/6,7/8/3/2,8/1/2/5}
			\path[auto=left, -]
			(pol1.corner \x)--(pol1.corner \y)
			node[midway, circle, inner sep=0cm, draw=none](e1_\i){$e_ {\j}$};
			\foreach \x/\y/\i/\j in {1/2/1/1,3/4/7/4,5/6/5/3,7/8/3/2}
			\path[auto=left, -]
			(pol1.corner \x)--(pol1.corner \y)
			node [whitenode] (s_\j) at (pol1.corner \x) {$s_\j$}
			node [blacknode] (j_\j) at (pol1.corner \y) {$j_\j$};
			\path[-]
			(e1_2) edge [out=-45-45*2,in=-45-45*3, color=blue]  (e1_3)
			(e1_4) edge [out=-45-45*4,in=-45-45*5, color=blue] (e1_5)
			(e1_6) edge [out=-45-45*6,in=-45-45*8, color=blue] (e1_8)
			(e1_7) edge [out=-45-45*7,in=-45-45*1, color=blue] (e1_1)
			;
			\path[-, dashed]
			(s_1) edge [bend right=0, color=red]  (s_2)
			(s_2) edge [bend right=0, color=red]  (s_3)
			(s_3) edge [bend right=0, color=red]  (s_4)
			(s_4) edge [bend right=0, color=red]  (s_1)
			(j_1) edge [bend right=80, color=red]  (j_4)
			(j_3) edge [bend right=80, color=white]  (j_2)
			%The white edge is just to center the picture
			;
	\end{tikzpicture} \end{array}
\end{align}
The red dashed lines show which vertices must be fused due to the pairing of their connecting edges. Fusion of the vertices thus leads to the ribbon-graph:
\renewcommand{\s}{1}
\renewcommand{\d}{1cm}
\begin{align}\label{Pic_Graph1}
	& \begin{array}{l} \begin{tikzpicture}[node distance=\d and \d,>=stealth',auto, every place/.style={draw}]
			\node [whitenode] (s1) {$\zeta_1$};
			\node [blacknode] (j1) [left=1.4*\d of s1] {$\xi_1$};
			\node [blacknode] (j2) [above right=of s1] {$\xi_2$};
			\node [blacknode] (j3) [below right=of s1] {$\xi_3$};
			\path[-]
			(s1) edge [bend right=10] node[el,below] (e5) {$e_5$} (j2)
			(s1) edge [bend right=-10] node[el,above] (e2) {$e_2$} (j2)
			(s1) edge [bend right=10] node[el,below] (e6) {$e_6$} (j3)
			(s1) edge [bend right=-10] node[el,above] (e3) {$e_3$} (j3)
			(j1) edge [bend right=-50] node[el,above] (e1) {$e_1$} (s1)
			(j1) edge [bend right=-30] node[el,below] (e4) {$e_4$} (s1)
			(j1) edge [bend right=30] node[el,above] (e8) {$e_8$} (s1)
			(j1) edge [bend right=50] node[el,below] (e7) {$e_7$} (s1)
			;
			\path[-]
			%(e4) edge [bend right=0, color=blue]  (e8)
			%(e1) edge [bend right=0, color=blue]  (e7)
			%(e5) edge [bend right=0, color=blue]  (e6)
			%(e4) edge [bend right=0, color=blue]  (e8)
			;
	\end{tikzpicture} \end{array}
\end{align}
The cycles can be recovered by listing all the edges that connect to a remaining vertex, so the only column-cycle in $C_c(\bm{\pi}_0)$ is
\begin{align*}
	& \zeta_1 = \{1,...,8\}
\end{align*}
and the row-cycles in $C_r(\bm{\pi}_0)$ are
\begin{align*}
	& \xi_1 = \{1,4,7,8\} \ , \ \xi_2 = \{2,5\} \ \text{ and } \ \xi_3 = \{3,6\} \ .
\end{align*}
An advantage of this interpretation is that it is immediately obvious that the number of total cycles cannot exceed $L+1$, i.e.
\begin{align}\label{Eq_CycleBound}
	& \#C_r(\bm{\pi}) + \#C_c(\bm{\pi}) \leq L+1 \ ,
\end{align}
since that is the maximum number of vertices a (connected) ribbon-graph with $2L$ edges can have.
\\
\\
\textit{Splitting the sums according to cycles:}\\
By construction of the row-cycles, each $j_x$ can only occur in edges of a single row-cycle $\xi \in C_r(\bm{\pi})$ (compare (\ref{Pic_Polygon2})). It follows that we can split the sum (\ref{Eq_TraceBoundExpression1}) up by row-cycles to see
\begin{align}\label{Eq_TraceBoundExpression2}
	& \text{(\ref{Eq_TraceBoundExpression1})} = \sum\limits_{\substack{s_1,...,s_L=1}}^M \prod\limits_{\xi \in C_r(\bm{\pi})} \sum\limits_{\substack{j_a=1 \\ \forall a \leq L, a \in \xi}}^d \prod\limits_{\substack{\{a,b\} \in \bm{\pi} \\ a,b \in \xi}}  \E[\bm{Y}_{e_a} \bm{Y}_{e_b}] \ .
\end{align}
The row-cycles $\xi \in C_r(\bm{\pi})$ by construction have a cycle-structure in the sense that there exists an enumeration $q^{\xi}_0,...,q^{\xi}_{\#\xi-1}$ of $\xi$ such that
\begin{align*}
	& \forall \text{ even } i<\#\xi : \ \ \{q^{\xi}_i,q^{\xi}_{i+1}\} \in \bm{\pi} \ \ \text{ and } \ \ \{q^{\xi}_{i-1 \mod \#\xi},q^{\xi}_{i}\} \in \bm{\tau}_r \ .
\end{align*}
This cycle structure turns the second sum in (\ref{Eq_TraceBoundExpression2}) into a trace of the form
\begin{align}\label{Eq_TraceBoundStep2}
	& \sum\limits_{\substack{j_a=1 \\ \forall a \leq L, a \in \xi}}^d \prod\limits_{\substack{\{a,b\} \in \bm{\pi} \\ a,b \in \xi}}  \E[\bm{Y}_{e_a} \bm{Y}_{e_b}] = \tr\big( (A_{\bm{s}(q^\xi_0),\bm{s}(q^\xi_1)}) \cdots (A_{\bm{s}(q^\xi_{\#\xi-2}),\bm{s}(q^\xi_{\#\xi-1})}) \big) \ ,
\end{align}
where $\bm{s}(a)$ denotes the second entry $s_y$ of the edge $e_a = (j_x,s_y)$. We can thus bound
\begin{align}\label{Eq_TraceBoundStep3}
	& \E\Big[\tr\Big( \big( \bm{Y} \bm{Y}^\top \big)^L \Big) \Big] \overset{\text{(\ref{Eq_TraceBoundStep1})}}{=} \sum\limits_{\bm{\pi} \in \Pi_2(2L)} \sum\limits_{\substack{s_1,...,s_L=1}}^M \sum\limits_{j_1,...,j_L=1}^d \prod\limits_{\{a,b\} \in \bm{\pi}} \E[\bm{Y}_{e_a} \bm{Y}_{e_b}] \nonumber\\
	& \overset{\text{(\ref{Eq_TraceBoundExpression2})}}{=} \sum\limits_{\bm{\pi} \in \Pi_2(2L)} \sum\limits_{\substack{s_1,...,s_L=1}}^M \prod\limits_{\xi \in C_r(\bm{\pi})} \sum\limits_{\substack{j_a=1 \\ \forall a \leq L, a \in \xi}}^d \prod\limits_{\substack{\{a,b\} \in \bm{\pi} \\ a,b \in \xi}}  \E[\bm{Y}_{e_a} \bm{Y}_{e_b}] \nonumber\\
	& \overset{\text{(\ref{Eq_TraceBoundStep2})}}{=} \sum\limits_{\bm{\pi} \in \Pi_2(2L)} \sum\limits_{\substack{s_1,...,s_L=1}}^M \prod\limits_{\xi \in C_r(\bm{\pi})} \tr\big( (A_{\bm{s}(q^\xi_0),\bm{s}(q^\xi_1)}) \cdots (A_{\bm{s}(q^\xi_{\#\xi-2}),\bm{s}(q^\xi_{\#\xi-1})}) \big) \nonumber\\
	& \overset{\text{(\ref{Eq_TraceBound_Neumann})}}{\leq} \sum\limits_{\bm{\pi} \in \Pi_2(2L)} \sum\limits_{\substack{s_1,...,s_L=1}}^M \prod\limits_{\xi \in C_r(\bm{\pi})} d \, ||A_{\bm{s}(q^\xi_0),\bm{s}(q^\xi_1)}|| \cdots ||A_{\bm{s}(q^\xi_{\#\xi-2}),\bm{s}(q^\xi_{\#\xi-1})}|| \nonumber\\
	& = \sum\limits_{\bm{\pi} \in \Pi_2(2L)} d^{\#C_r(\bm{\pi})} \sum\limits_{\substack{s_1,...,s_L=1}}^M \prod\limits_{\xi \in C_r(\bm{\pi})}  ||A_{\bm{s}(q^\xi_0),\bm{s}(q^\xi_1)}|| \cdots ||A_{\bm{s}(q^\xi_{\#\xi-2}),\bm{s}(q^\xi_{\#\xi-1})}|| \nonumber\\
	& = \sum\limits_{\bm{\pi} \in \Pi_2(2L)} d^{\#C_r(\bm{\pi})} \sum\limits_{\substack{s_1,...,s_L=1}}^M \prod\limits_{\{a,b\} \in \bm{\pi}}  \underbrace{||A_{\bm{s}(a),\bm{s}(b)}||}_{= B_{\bm{s}(a),\bm{s}(b)}} \ .
\end{align}
We can then analogously split up the second sum according to column-cycles to see
\begin{align*}
	& \sum\limits_{\substack{s_1,...,s_L=1}}^M \prod\limits_{\{a,b\} \in \bm{\pi}} B_{\bm{s}(a),\bm{s}(b)} = \prod\limits_{\zeta \in C_c(\bm{\pi})} \sum\limits_{\substack{s_a=1 \\ \forall a \leq L, a \in \zeta}}^M \prod\limits_{\substack{\{a,b\} \in \bm{\pi} \\ a,b \in \zeta}} B_{\bm{s}(a),\bm{s}(b)}
\end{align*}
and likewise use the cycle structure of $\zeta$ to write the right hand sum as $\tr(B^{\frac{\#\zeta}{2}})$. We finally arrive at
\begin{align*}
	& \E\Big[\tr\Big( \big( \bm{Y} \bm{Y}^\top \big)^L \Big) \Big] \overset{\text{(\ref{Eq_TraceBoundStep3})}}{\leq} \sum\limits_{\bm{\pi} \in \Pi_2(2L)} d^{\#C_r(\bm{\pi})} \sum\limits_{\substack{s_1,...,s_L=1}}^M \prod\limits_{\{a,b\} \in \bm{\pi}} B_{\bm{s}(a),\bm{s}(b)}\\
	& = \sum\limits_{\bm{\pi} \in \Pi_2(2L)} d^{\#C_r(\bm{\pi})} \prod\limits_{\zeta \in C_c(\bm{\pi})} \tr(B^{\frac{\#\zeta}{2}}) \overset{\text{(\ref{Eq_TraceBound_Neumann})}}{\leq} \sum\limits_{\bm{\pi} \in \Pi_2(2L)} d^{\#C_r(\bm{\pi})} \prod\limits_{\zeta \in C_c(\bm{\pi})} M \kappa^{\frac{\#\zeta}{2}}\\
	& = \kappa^L \sum\limits_{\bm{\pi} \in \Pi_2(2L)} d^{\#C_r(\bm{\pi})} M^{\#C_c(\bm{\pi})} \leq \kappa^L \sum\limits_{\bm{\pi} \in \Pi_2(2L)} (d+M)^{\#C_r(\bm{\pi}) + \#C_c(\bm{\pi})}\\
	& \overset{\text{(\ref{Eq_CycleBound})}}{\leq} \kappa^L (d+M)^{L+1} \#\Pi_2(2L) \ .
\end{align*}
It is a simple combinatorial exercise that the number of pairings of the set $\{1,...,2L\}$ is given by $\#\Pi_2(2L) = (2L-1) \cdot (2L-3) \cdots 3 \cdot 1 = (2L-1)!!$, which concludes the proof of Theorem \ref{Thm_TraceMomentBound}.\qed

\section{Proof of Theorem \ref{Thm_MarchenkoPastur_Daniell} in the Gaussian case}

The proof of Theorem \ref{Thm_MarchenkoPastur_Daniell} in the Gaussian case relies strongly on the approximation of the Daniell smoothed periodogram by a random matrix of a structure similar to that of a Sample covariance matrix.
With the notation $\bm{\eta} = \big[ \eta_0,...,\eta_{n-1} \big]$ we define the approximating matrix
\begin{align}\label{Eq_Def_tSbar0}
	& \tilde{S}'\Big( \frac{2\pi r}{n} \Big) \coloneq \frac{1}{2m+1} G\Big( \frac{2\pi r}{n} \Big) \bm{\eta} V D_r V^* \bm{\eta}^\top G\Big( \frac{2\pi r}{n} \Big)^* \ .
\end{align}
The main part of this proof is the following approximation result, made possible by Theorem \ref{Thm_TraceMomentBound}.

\begin{proposition}[Approximation by a simpler model]\label{Prop_ApproxSimpler}\
	\\
	Suppose Assumptions \ref{Assumption0_LinearProcess} and \ref{Assumption_Main} without (A4) hold. Assume further that the innovations $(\eta_t)_{t \in \Z}$ are Gaussian.\\
	There for every (small) $\delta > 0$ and (large) $D>0$ exists a constant $C = C(\delta,D) > 0$, which also depends on the constants $\cK_1,\cK_2,\alpha,\gamma$ from Assumption \ref{Assumption_Main},
	such that
	\begin{align*}
		& \bP\Big( \sup\limits_{r \in \{0,...,n-1\}} d_{\operatorname{BL}}\big( \hat{\mu}_{S(\frac{2\pi r}{n})}, \hat{\mu}_{\tilde{S}'(\frac{2\pi r}{n})} \big) \geq \big( 48\cK_2 + 20\cK_2^2 \big) n^{\max(\frac{\delta+(1-\gamma)\alpha}{1+\gamma}, \delta + \alpha-1)} \Big) \leq C n^{-D}
	\end{align*}
	holds for all $n \in \N$.
\end{proposition}
\
\\
The usefulness of this result is highlighted by the following lemma. The matrix $\bm{\eta} V D_r V^* \bm{\eta}^\top$ being almost of isotropic Wishart type makes $\tilde{S}'\big( \frac{2\pi r}{n} \big)$ almost of Wishart-type with covariance matrix $F\big( \frac{2\pi r}{n} \big)$.

\begin{lemma}[Almost Wishart matrices]\thlabel{Lemma_AlmostWishart}\
\\
Suppose that the innovations $(\eta_t)_{t \in \Z}$ are Gaussian and $n>2m$.
\begin{itemize}
	\item[a)]
	For $r \in \{1,...,n\}$ such that $\rho_n(r,0) > m$ and $\rho_n(r,\frac{n}{2}) > m$ the matrix $\bm{\eta} V D_r V^* \bm{\eta}^\top$ has isotropic complex Wishart distribution $CW_d(\operatorname{Id}_d,2m+1)$.
	
	\item[b)]
	For $r \in \{0, \lfloor \frac{n}{2} \rfloor\}$ there exists a (Hermitian) matrix $E \in \C^{d \times d}$ with rank no greater than $3$ such that $\bm{\eta} V D_r V^* \bm{\eta}^\top-E$ has isotropic real Wishart distribution $W_d(\operatorname{Id}_d,2m+1)$.
\end{itemize}
\end{lemma}
\
\\
We will also require so-called local laws, which are a generalization of the Marchenko-Pastur law, where the difference between $\cs_{\nu_\infty}$ and the Stieltjes transform of the ESD of a sample covariance matrix is bounded. The following Lemma is a special case of Theorem 2.2 from \cite{GramErdos}. The main advantage for our purposes is that probabilities are bounded with sub-polynomial rates in $m$, which will prove necessary, if we wish to allow $\alpha \in (0,1)$ to be arbitrarily small.

\begin{lemma}[Application of local laws]\thlabel{Lemma_LocalLaw}\
\\
Suppose (A1) and (A3) of Assumption \ref{Assumption_Main} hold. Define $\hat{\nu}_n(\theta) \coloneq \hat{\mu}_{\frac{1}{2m+1}G(\theta)\bm{W}G(\theta)^*}$ and let $\nu_n(\theta)$ denote the probability measure defined by Lemma \ref{Lemma_MPEquation} for $H_n(\theta) = \hat{\mu}_{F^{(n)}(\theta)}$ and $c_n = \frac{d}{2m}$.
For a fixed $\tau \in (0,1)$ define the spectral domain
\begin{align*}
	& \bm{S}(\tau) \coloneq \big\{ z \in \C^+ \ \big| \ \Im(z) \geq \tau , \, |z| \leq \tau^{-1} \big\} \ .
\end{align*}
For every (small) $\delta \in (0,1)$ and (large) $D>0$ there exists a sequence $(a_n)_{n \in \N} \subset (0,1)$ converging to zero and a constant $C=C(\tau,\delta,D)>0$, which also depends on the constants $\cK_1,\cK_2$ and $\alpha$, such that:
\begin{itemize}
	\item[a)]
	For a $(d \times d)$ complex isotropic Wishart matrix $\bm{W} \sim CW_{d}(\operatorname{Id}_d,2m+1)$ it holds that
	\begin{align}\label{Eq_OuterLaw_complex}
		& \bP\Big( \exists r < n \, \exists z \in \bm{S}(\tau) : \  \big| \cs_{\hat{\nu}_n(\frac{2\pi r}{n})}(z) - \cs_{\nu_n(\frac{2\pi r}{n})}(z) \big| \geq a_n \Big) \leq \frac{C(\tau,\delta,D)}{n^D}
	\end{align}
	for every $n \in \N$.
	
	\item[b)]
	For a $(d \times d)$ real isotropic Wishart matrix $\bm{W} \sim W_{d}(\operatorname{Id}_d,2m+1)$ it holds that
	\begin{align}\label{Eq_OuterLaw_real}
		& \bP\Big( \exists z \in \bm{S}(\tau) : \ \big| \cs_{\hat{\nu}_n(\theta)}(z) - \cs_{\nu_n(\theta)}(z) \big| \geq a_n \Big) \leq \frac{C(\tau,\delta,\tilde{D})}{n^D}
	\end{align}
	for every $n \in \N$ and $\theta \in \{0,\pi\}$.
\end{itemize}
\end{lemma}

\begin{corollary}[Uniform Marchenko-Pastur law]\label{Cor_UniformMP}\
\\
Suppose Assumptions \ref{Assumption0_LinearProcess} and (A1)-(A3) of \ref{Assumption_Main} hold and that the innovations $(\eta_t)_{t \in \Z}$ are Gaussian. For each frequency $\theta \in [0,2\pi)$ let $\nu_\infty(\theta)$ be the probability measure defined as in Lemma \ref{Lemma_MPEquation} through $H_\infty(\theta)$ and $c$.
The uniform almost sure convergence
\begin{align*}
	& 1 = \bP\Big( \forall \theta \in [0,2\pi) : \ \hat{\mu}_{\tilde{S}'(\theta)} \xRightarrow{n \rightarrow \infty} \nu_\infty(\theta) \Big)
\end{align*}
holds.
\end{corollary}
\begin{proof}\
\\
For every $\theta \in (0,\pi)\cup(\pi,2\pi)$ let $r_n(\theta)$ denote $\lfloor \frac{n\theta}{2\pi} + \frac{1}{2} \rfloor \in \{0,...,n-1\}$ such that $[\theta]_n = \frac{2\pi r_n(\theta)}{n}$. For large $n$, it by $m \asymp n^{\alpha}$ must hold that $\rho_n(r_n(\theta),0) > m$ and $\rho_n(r_n(\theta),\frac{n}{2}) > m$ and (a) of Lemma \ref{Lemma_AlmostWishart} together with (a) of Lemma \ref{Lemma_LocalLaw} by Borel-Cantelli yields
\begin{align*}
	& 1 = \bP\Big( \forall \theta \in (0,\pi)\cup(\pi,2\pi) \, \forall z \in \bm{S}(\tau) : \ \cs_{\hat{\mu}_{\tilde{S}'([\theta]_n)}}(z) - \cs_{\nu_n([\theta]_n)}(z) \xrightarrow{n \rightarrow \infty} 0 \Big)
\end{align*}
for all $\tau > 0$. Continuity of measures allows us to take the limit for $\tau \searrow 0$ and the convergence $\cs_{\nu_n([\theta]_n)}(z) \xrightarrow{n \rightarrow \infty} \cs_{\nu_\infty(\theta)}(z), \, \forall z \in \C^+$ follows from (\ref{Eq_Assumption_ESD_Convergence}) and well-known analytic properties of the Stieltjes transform. It follows that
\begin{align*}
	& 1 = \bP\Big( \forall \theta \in (0,\pi)\cup(\pi,2\pi) \, \forall z \in \C^+ : \ \cs_{\hat{\mu}_{\tilde{S}'([\theta]_n)}}(z) \xrightarrow{n \rightarrow \infty} \cs_{\nu_\infty(\theta)}(z) \Big)
\end{align*}
and, as it is well known that point-wise convergence of Stieltjes transforms implies weak convergence of the corresponding probability measures (see Theorem 2.4.4 of \cite{AndersonIRM}), we have shown
\begin{align}\label{Eq_ApproxUniformMarchenkoPastur1}
	& 1 = \bP\Big( \forall \theta \in (0,\pi)\cup(\pi,2\pi) : \ \hat{\mu}_{\tilde{S}'([\theta]_n)} \xRightarrow{n \rightarrow \infty} \nu_\infty(\theta) \Big) \ .
\end{align}
For $\theta \in \{0,\pi\}$ there by (b) of Lemma \ref{Lemma_AlmostWishart} exists a Hermitian matrix $E_n(\theta)$ of rank no greater than $3$ such that
\begin{align}\label{Eq_S_tildeBar_Decomp}
	& \tilde{S}'([\theta]_n) = \frac{1}{2m+1} G([\theta]_n) \bm{W} G([\theta]_n)^* + E_n(\theta)
\end{align}
for some isotropic real Wishart matrix $\bm{W} \sim W_d(\operatorname{Id}_d,2m+1)$. Statement (b) of Lemma \ref{Lemma_LocalLaw} by the same arguments as above gives
\begin{align}\label{Eq_ApproxUniformMarchenkoPastur2}
	& 1 = \bP\Big( \forall \theta \in \{0,\pi\} : \ \hat{\mu}_{\frac{1}{2m+1} G(\theta) \bm{W} G(\theta)^*} \xRightarrow{n \rightarrow \infty} \nu_\infty(\theta) \Big) \ .
\end{align}
It remains to follow
\begin{align}\label{Eq_ApproxUniformMarchenkoPastur3}
	& 1 = \bP\Big( \forall \theta \in \{0,\pi\} : \ \hat{\mu}_{\tilde{S}'([\theta]_n)} \xRightarrow{n \rightarrow \infty} \nu_\infty(\theta) \Big) \ ,
\end{align}
which we do by bounding the bounded Lipschitz distance
\begin{align*}
	& d_{\operatorname{BL}}\big(\hat{\mu}_{\tilde{S}'([\theta]_n)},\hat{\mu}_{\frac{1}{2m+1} G(\theta) \bm{W} G(\theta)^*}\big)
\end{align*}
in high probability. The trivial continuity property
\begin{align}\label{Eq_G_Continuity}
	& ||G(\tau)-G(\tau')|| \leq \sum\limits_{k=0}^\infty \big| e^{-ik\tau} - e^{-ik\tau'} \big| \, ||\Psi_k|| \leq \sum\limits_{k=0}^\infty \big| e^{-i\tau} - e^{-i\tau'} \big| \, k \, ||\Psi_k|| \nonumber\\
	& \leq \sum\limits_{k=0}^\infty \operatorname{dist}(\tau+2\pi\Z,\tau'+2\pi\Z) \, k \, ||\Psi_k|| = \frac{2\pi}{n} \rho_n\Big(\frac{n \tau}{2\pi}, \frac{n \tau'}{2\pi}\Big) \sum\limits_{k=0}^\infty k \, ||\Psi_k||
\end{align}
by (\ref{Eq_Assumption_LongRandeDependence_cK2}) yields $||G([\theta]_n) - G(\theta)|| \leq \frac{2\pi}{n} \cK_2$.
The operator norms of isotropic (complex or real) Wishart matrices $\bm{W}$ are very well understood and for example Theorem 2 of \cite{LedouxRider} may be used to show the existence of a constant
$C''(D)>0$ only dependent on $D$ and $\cK_1$ such that
\begin{align}\label{Eq_WishartNormBound}
	& \bP\Big( \frac{\big|\big| \bm{W} \big|\big|}{2m+1} \geq C''(D) \Big) \leq \frac{C''(D)}{n^{D}}
\end{align}
holds for all $n \in \N$. This results in the bound
\begin{align*}
	& \bP\Big( \Big|\Big| \frac{1}{2m+1} G([\theta]_n) \bm{W} G([\theta]_n)^* - \frac{1}{2m+1} G(\theta) \bm{W} G(\theta)^* \Big|\Big| \geq \frac{4\pi\cK_2^2}{n} C''(D) \Big)\\
	& \overset{\text{(\ref{Eq_Assumption_LongRandeDependence_cK2})}}{\leq} \bP\Big( \frac{2\cK_2}{2m+1} \big|\big| \bm{W} \big|\big| \, \big|\big| G([\theta]_n) - G(\theta) \big|\big| \geq \frac{4\pi\cK_2^2}{n} C''(D) \Big)\\
	& \overset{\text{(\ref{Eq_G_Continuity})}}{\leq} \bP\Big( \frac{4\pi\cK_2^2}{n(2m+1)} \big|\big| \bm{W} \big|\big| \geq \frac{4\pi\cK_2^2}{n} C''(D) \Big) \leq \frac{C''(D)}{n^D} \ .
\end{align*}
By Borel-Cantelli and Lemma \ref{Lemma_LowRankPerturbation} we have finally shown
\begin{align}\label{Eq_ApproxUniformMarchenkoPastur4}
	& 1 = \bP\Big( \forall \theta \in \{0,\pi\} : \ d_{\operatorname{BL}}\big(\hat{\mu}_{\tilde{S}'([\theta]_n)},\hat{\mu}_{\frac{1}{2m+1} G(\theta) \bm{W} G(\theta)^*}\big) \xrightarrow{n \rightarrow \infty} 0 \Big)
\end{align}
and the fact that $d_{\operatorname{BL}}$ metricizes weak convergence of probability measures (see Theorem 1.12.4 of \cite{wellner2013weak}) means (\ref{Eq_ApproxUniformMarchenkoPastur1}), (\ref{Eq_ApproxUniformMarchenkoPastur2}) and (\ref{Eq_ApproxUniformMarchenkoPastur4}) together lead to the wanted result.
\end{proof}
\
\\
The property
\begin{align*}
	& 1 = \bP\Big( \forall \theta \in [0,2\pi) : \ \hat{\mu}_{S(\theta)} \xRightarrow{n \rightarrow \infty} \nu_\infty(\theta) \Big)
\end{align*}
then follows directly from Proposition \ref{Prop_ApproxSimpler} and Corollary \ref{Cor_UniformMP} by application of Borel-Cantelli and the fact that $d_{\operatorname{BL}}$ metricizes weak convergence of probability measures (see Theorem 1.12.4 of \cite{wellner2013weak}).\\
This concludes the proof of Theorem \ref{Thm_MarchenkoPastur_Daniell} in the Gaussian case. \qed

\section{Proof of Proposition \ref{Prop_ApproxSimpler}}

The bounded Lipschitz metric $d_{BL}$ is used for this result, since it is known to metricize weak convergence of probability measures and is also robust under low-rank perturbations in the sense of the following lemma.

\begin{lemma}[Low rank perturbation for Hermitian matrices]\thlabel{Lemma_LowRankPerturbation}\
	\\
	For two Hermitian matrices $A,E \in \C^{d \times d}$ the bounded-Lipschitz distance $d_{\operatorname{BL}}$ satisfies the bound $d_{\operatorname{BL}}(\hat{\mu}_{A},\hat{\mu}_{A+E}) \leq \frac{8\mathrm{rank}(E)}{d}$.
\end{lemma}
\
\\
Define the $(d \times n)$ matrix $\tilde{\bm{X}}$ by its columns as
\begin{align}\label{Eq_Def_tX}
	& \tilde{\bm{X}} \coloneq \bigg[ \sum\limits_{k=0}^{\infty} \Psi_k \eta_{(0-k) \mod n} , ... , \sum\limits_{k=0}^{\infty} \Psi_k \eta_{(n-1-k) \mod n} \bigg] \ .
\end{align}
The following lemma applies a neighboring technique from time series analysis to show that the ESD of
\begin{align}\label{Eq_Def_tS}
	& \tilde{S}\Big(\frac{2\pi r}{n}\Big) \coloneq \frac{1}{2m+1} \tilde{\bm{X}} V D_r V^* \tilde{\bm{X}}^\top
\end{align}
approximates the ESD of the Daniell smoothed periodogram. This approximation method becomes applicable in the high-dimensional setting by application of our Theorem \ref{Thm_TraceMomentBound}.

\begin{lemma}[Approximation by neighboring]\thlabel{Lemma_ApproxNeighboring}\
\\
Suppose Assumptions \ref{Assumption0_LinearProcess} and (A1) of \ref{Assumption_Main} hold. Additionally, assume the innovations $(\eta_t)_{t \in \Z}$ to be Gaussian.\\
For any (small) $\delta>0$ and (large) $D>0$ there exists a constant $C'=C'(\delta,D)>0$, which also depends on the constants $\cK_1,\alpha,\gamma$ from Assumption \ref{Assumption_Main}, such that
\begin{align}\label{Eq_ApproxNei_Result1}
	& \bP\Big( d_{\operatorname{BL}}\big( \hat{\mu}_{S(\frac{2\pi r}{n})}, \hat{\mu}_{\tilde{S}(\frac{2\pi r}{n})} \big) \geq \frac{48K}{d} + \varepsilon_{\Psi}(K) \, n^{\delta + \alpha} \Big) \leq C' n^{-D}
\end{align}
holds with
\begin{align*}
	& \varepsilon_{\Psi}(K) \coloneq 4 \Big(\sum\limits_{k=K+1}^\infty ||\Psi_k||\Big) \Big(\sum\limits_{k=0}^\infty ||\Psi_k||\Big)
\end{align*}
for all $n \in \N$ and $r,K \in \{0,...,n-1\}$.
\end{lemma}
\
\\
Assumption (A2) of \ref{Assumption_Main} may be used to bound $\varepsilon_{\Psi}(K) \leq 4\cK_2^2 K^{-\gamma}$. The property (\ref{Eq_ApproxNei_Result1}) will be most useful for the choice $K=\lfloor n^{\frac{\delta + 2\alpha}{1+\gamma}} \rfloor$. We assume here that $\delta > 0$ is small enough for $\frac{\delta+2\alpha}{1+\gamma} < 1$ to hold, such that this choice of $K \in \{0,...,n-1\}$ is valid. The bound
\begin{align*}
	& \frac{48K}{d} + \varepsilon_{\Psi}(K) \, n^{\delta + \alpha} \overset{\text{(\ref{Eq_Assumption_Asymptotics})}}{\leq} 48K \cK_2 n^{-\alpha} + \varepsilon_{\Psi}(K) \, n^{\delta + \alpha} \leq 48K \cK_2 n^{-\alpha} + 4\cK_2^2 K^{-\gamma} \, n^{\delta + \alpha}\\
	& \leq 48 n^{\frac{\delta + 2\alpha}{1+\gamma}} \cK_2 n^{-\alpha} + 4\cK_2^2 n^{-\gamma\frac{\delta + 2\alpha}{1+\gamma}} \, n^{\delta + \alpha} = \big( 48\cK_2 + 4\cK_2^2 \big) n^{\frac{\delta+(1-\gamma)\alpha}{1+\gamma}}
\end{align*}
for this choice of $K$ allows us to follow
\begin{align}\label{Eq_ApproxNei_Result2}
	& \bP\Big( d_{\operatorname{BL}}\big( \hat{\mu}_{S(\frac{2\pi r}{n})}, \hat{\mu}_{\tilde{S}(\frac{2\pi r}{n})} \big) \geq \big( 48\cK_2 + 4\cK_2^2 \big) n^{\frac{\delta+(1-\gamma)\alpha}{1+\gamma}} \Big) \leq C' n^{-D}
\end{align}
from (\ref{Eq_ApproxNei_Result1}).
\\
\\
It remains to show that $\tilde{S}(\frac{2\pi r}{n})$ is approximated sufficiently well by $\tilde{S}'(\frac{2\pi r}{n})$ as defined in (\ref{Eq_Def_tSbar0}). We do this in the following lemma, where it is once more Theorem \ref{Thm_TraceMomentBound} that allows for the application of the continuity property (\ref{Eq_G_Continuity}) in this high-dimensional setting.

\begin{lemma}[Approximation by continuity]\thlabel{Lemma_ApproxContinuity}\
\\
Suppose Assumptions \ref{Assumption0_LinearProcess} and (A1) of \ref{Assumption_Main} hold. Additionally assume the innovations $(\eta_t)_{t \in \Z}$ to be Gaussian.\\
For any (small) $\delta>0$ and (large) $D>0$ there exists a constant $C''=C''(\delta,D)>0$, which also depends on the constants $\cK_1,\alpha,\gamma$ from Assumption \ref{Assumption_Main}, such that
\begin{align}\label{Eq_ApproxContinuity_result}
	& \bP\Big( \Big|\Big| \tilde{S}\Big( \frac{2\pi r}{n} \Big) - \tilde{S}'\Big( \frac{2\pi r}{n} \Big)\Big|\Big| \geq \varepsilon_{\Psi} \, n^{\delta + \alpha-1} \Big) \leq C'' n^{-D}
\end{align}
holds with
\begin{align*}
	& \varepsilon_{\Psi} = 16 \bigg(\sum\limits_{k=0}^\infty k ||\Psi_k||\bigg) \bigg(\sum\limits_{k=0}^\infty ||\Psi_k||\bigg)
\end{align*}
for all $n \in \N$ and $r \in \{0,...,n-1\}$.
\end{lemma}
\
\\
It is from the definition of the bounded Lipschitz metric easy to see that
\begin{align*}
	& d_{BL}(\hat{\mu}_{A},\hat{\mu}_{A'}) \leq ||A-A'||
\end{align*}
holds for any Hermitian matrices $A,A'$. Additionally, (A2) of Assumption \ref{Assumption_Main} allows for $\varepsilon_{\Psi} \leq 16 \cK_2^2$. We combine the results (\ref{Eq_ApproxNei_Result2}) and (\ref{Eq_ApproxContinuity_result}) into
\begin{align*}
	& \bP\Big( \sup\limits_{r \in \{0,...,n-1\}} d_{\operatorname{BL}}\big( \hat{\mu}_{S(\frac{2\pi r}{n})}, \hat{\mu}_{\tilde{S}'(\frac{2\pi r}{n})} \big) \geq \big( 48\cK_2 + 20\cK_2^2 \big) n^{\max(\frac{\delta+(1-\gamma)\alpha}{1+\gamma}, \delta + \alpha-1)} \Big)\\
	& \leq \bP\Big( \sup\limits_{r \in \{0,...,n-1\}} d_{\operatorname{BL}}\big( \hat{\mu}_{S(\frac{2\pi r}{n})}, \hat{\mu}_{\tilde{S}'(\frac{2\pi r}{n})} \big) \geq \big( 48\cK_2 + 4\cK_2^2 \big) n^{\frac{\delta+(1-\gamma)\alpha}{1+\gamma}} + 16\cK_2^2 n^{\delta + \alpha-1} \Big)\\
	& \leq \sum\limits_{r=0}^{n-1} \bP\Big( d_{\operatorname{BL}}\big( \hat{\mu}_{S(\frac{2\pi r}{n})}, \hat{\mu}_{\tilde{S}(\frac{2\pi r}{n})} \big) \geq \big( 48\cK_2 + 4\cK_2^2 \big) n^{\frac{\delta+(1-\gamma)\alpha}{1+\gamma}} \Big)\\
	& \hspace{1cm} + \sum\limits_{r=0}^{n-1} \bP\Big( d_{\operatorname{BL}}\big( \hat{\mu}_{\tilde{S}(\frac{2\pi r}{n})}, \hat{\mu}_{\tilde{S}'(\frac{2\pi r}{n})} \big) \geq 16\cK_2^2 n^{\delta + \alpha-1} \Big)\\
	& \leq n C' n^{-D} + n C'' n^{-D} = (C'+C'') n^{-(D-1)} \ .
\end{align*}
By choosing $C(\delta,D) = C'(\delta,D+1)+C''(\delta,D+1)$ we have proved Proposition \ref{Prop_ApproxSimpler}. \qed

%\newpage
\section{Proof of Theorem \ref{Thm_MarchenkoPastur_Daniell} by universality}
This proof goes along similar lines as section 9 in \cite{Aue} or section A.3 in \cite{StrongMP}.
\\
\\
Let $(\xi_t)_{t \in \Z}$ be $d$-dimensional i.i.d. standard normal innovations and let the innovations $(\eta_t)_{t \in \N}$ be as described in Assumptions \ref{Assumption0_LinearProcess} and \ref{Assumption_Main}. For any fixed small $\delta > 0$ let
\begin{align}\label{Eq_Univ_DefK}
	& K = \Big\lfloor n^{\min(\frac{1-\alpha}{2}, \alpha-\frac{1}{2})-\delta} \Big\rfloor
\end{align}
and define the approximating processes
\begin{align}\label{Eq_FiniteApprox}
	& \mathcal{X}^{(K)}_t \coloneq \sum\limits_{k=0}^K \Psi_{k} \xi_{t-k} \ \ \text{ and } \ \ X_t^{(K)} \coloneq \sum\limits_{k=0}^K \Psi_{k} \eta_{t-k}
\end{align}
as well as their corresponding Daniell smoothed periodograms $\mathcal{S}^{(K)}$ and $S^{(K)}$ analogous to (\ref{Eq_Def_Daniell}) and (\ref{Eq_DaniellAsProduct}), i.e.
\begin{align}\label{Eq_Def_cSK}
	& \mathcal{S}^{(K)}\Big( \frac{2\pi r}{n} \Big) = \frac{1}{2m+1} \bm{\mathcal{X}}^{(K)} VD_rV^* (\bm{\mathcal{X}}^{(K)})^\top
\end{align}
and
\begin{align}\label{Eq_Def_SK}
	& S^{(K)}\Big( \frac{2\pi r}{n} \Big) = \frac{1}{2m+1} \bm{X}^{(K)} VD_rV^* (\bm{X}^{(K)})^\top
\end{align}
for $\bm{\mathcal{X}}^{(K)} = [\mathcal{X}^{(K)}_0,...,\mathcal{X}^{(K)}_{n-1}]$ and $\bm{X}^{(K)} = [X^{(K)}_0,...,X^{(K)}_{n-1}]$.
\\
\\
The first step will be to show the convergence
\begin{align}\label{Eq_Univ_Step1}
	& \E\big[ \cs_{\hat{\mu}_{\mathcal{S}^{(K)}(\theta)}}(z) - \cs_{\hat{\mu}_{S^{(K)}(\theta}}(z) \big] \xrightarrow{n \rightarrow \infty} 0
\end{align}
for all $r \in \{0,...,n-1\}$ and $z \in \C^+$ with $\Im(z) \in (0,1)$. The second step will be to show that the Stieltjes transforms will both be close to their respective means in the sense
\begin{align}\label{Eq_Univ_Step2}
	& 1 = \bP\Big( \forall\theta\in [0,2\pi) \, \forall z \in \C^+ : \ \cs_{\hat{\mu}_{\mathcal{S}^{(K)}(\theta)}}(z) - \E[\cs_{\hat{\mu}_{\mathcal{S}^{(K)}(\theta)}}(z)] \xrightarrow{n \rightarrow \infty} 0 \Big) \nonumber\\
	& 1 = \bP\Big( \forall \theta \in [0,2\pi) \, \forall z \in \C^+ : \ \cs_{\hat{\mu}_{S^{(K)}(\theta)}}(z) - \E\big[ \cs_{\hat{\mu}_{S^{(K)}(\theta)}}(z) \big] \xrightarrow{n \rightarrow \infty} 0 \Big) \ .
\end{align}
Finally, the third step will show the approximations
\begin{align}\label{Eq_Univ_Step3}
	& 1 = \bP\Big( \forall \theta \in [0,2\pi) \, \forall z \in \C^+ : \ \cs_{\hat{\mu}_{\mathcal{S}^{(K)}(\theta)}}(z) - \cs_{\hat{\mu}_{\mathcal{S}(\theta)}}(z) \xrightarrow{n \rightarrow \infty} 0 \Big) \nonumber\\
	& 1 = \bP\Big( \forall \theta \in [0,2\pi) \, \forall z \in \C^+ : \ \cs_{\hat{\mu}_{S^{(K)}(\theta)}}(z) - \cs_{\hat{\mu}_{S(\theta)}}(z) \xrightarrow{n \rightarrow \infty} 0 \Big) \ .
\end{align}
It suffices to show (\ref{Eq_Univ_Step2}) and (\ref{Eq_Univ_Step3}) for all $z \in \C^+$ with $\Im(z) \in (0,1)$, which can be seen by expressing the Stieltjes transforms for $\Im(z) > 1$ as complex curve integrals.

\subsection{Step 1: Proving convergence of means (\ref{Eq_Univ_Step1})}

The first tool for this step is a minor adaptation of the Lindeberg principle from \cite{ChatterjeeLindeberg}. Recall the definition
\begin{align*}
	& \mathrm{u}_i^{(d)} \coloneq (\underbrace{0,...,0}_{\times(i-1)},1,0,...,0)^\top \in \R^d \ .
\end{align*}

\begin{lemma}[Lindeberg principle]\label{Lemma_Lindeberg}\
	\\
	Let $f: \R^N \rightarrow \C$ be a smooth, bounded map and let $X_1,...,X_N,Y_1,...,Y_N$ be independent random variables with the same mean and variance as well as uniformly bounded sixth moments, i.e. $\E[X_i^6] \leq C_6 \geq \E[Y_i^6]$ for some constant $C_6>0$. The bound
	\begin{align*}
		& \big| \E\big[ f(X_1,...,X_N) - f(Y_1,...,Y_N) \big] \big| \leq \frac{\sqrt{C_6}}{2} \sum\limits_{i=1}^N M^X_i + M^Y_i
	\end{align*}
	holds for
	\begin{align*}
		& M^X_i = \max\limits_{\tau \in [0,1]} \E\big[ |\partial_i^3 f(\bm{Z}^0_i + \tau X_i \mathrm{u}_i^{(N)})|^2 \big]^\frac{1}{2} \ \ \text{ and } \ \ M^Y_i = \max\limits_{\tau \in [0,1]} \E\big[ |\partial_i^3 f(\bm{Z}^0_i + \tau Y_i \mathrm{u}_i^{(N)})|^2 \big]^\frac{1}{2} \ .
	\end{align*}
\end{lemma}
\
\\
In order to apply this lemma to $\eta \mapsto \cs_{\hat{\mu}_{S^{(K)}(\theta)}}(z)$, we must first calculate some derivatives.
The derivative of $X^{(K)}$ with regards to a component $(\eta_{t_0})_i$ is
\begin{align*}
	& \frac{\partial}{\partial (\eta_{t_0})_i} X^{(K)}_t = \sum\limits_{k=0}^K \Psi_{k} \frac{\partial }{\partial (\eta_{t_0})_i} \eta_{t-k} = \mathbbm{1}_{t-t_0 \in \{0,...,K\}} \, \Psi_{t-t_0} \mathrm{u}_i^{(d)} \ ,
\end{align*}
which for the matrix
\begin{align}\label{Eq_Def_B}
	& \bm{B} \coloneq \frac{1}{\sqrt{2m+1}} \bm{X}^{(K)} V D_r^{\frac{1}{2}}
\end{align}
immediately gives
\begin{align*}
	& \frac{\partial^2}{\partial (\eta_{t_1})_{i_1} \, \partial (\eta_{t_2})_{i_2}} \bm{B} = 0 \ \ \text{ and } \ \ \mathrm{rank}\Big(\frac{\partial}{\partial (\eta_{t_0})_i}\bm{B}\Big) \leq K \ .
\end{align*}
We further observe
\begin{align*}
	& \frac{\partial}{\partial (\eta_{t_0})_i}\bm{B}_{j,s} = \frac{1}{\sqrt{2m+1}} \Big(\frac{\partial}{\partial (\eta_{t_0})_i} \bm{X}^{(K)} V D_r^{\frac{1}{2}}\Big)_{j,s}\\
	& = \frac{\mathbbm{1}_{\rho_n(r,s-1)\leq m}}{\sqrt{2m+1} \sqrt{n}} \sum\limits_{l=1}^n \frac{\partial}{\partial (\eta_{t_0})_i} \bm{X}^{(K)}_{j,l} e^{-2\pi i \frac{(l-1)(s-1)}{n}}\\
	& = \frac{\mathbbm{1}_{\rho_n(r,s-1)\leq m}}{\sqrt{2m+1} \sqrt{n}} \sum\limits_{t=0}^{n-1} e^{-2\pi i \frac{t(s-1)}{n}} \mathbbm{1}_{t-t_0 \in \{0,...,K\}} \, (\mathrm{u}_j^{(d)})^\top \Psi_{t-t_0} \mathrm{u}_i^{(d)} \ ,
\end{align*}
which gives
\begin{align*}
	& \Big|\Big| \frac{\partial}{\partial (\eta_{t_0})_i}\bm{B} \Big|\Big|^2 \leq \Big|\Big| \frac{\partial}{\partial (\eta_{t_0})_i}\bm{B} \Big|\Big|_F^2\\
	& = \frac{1}{(2m+1)n} \sum\limits_{\substack{s=1 \\ \rho_n(r,s-1) \leq m}}^n \Big|\Big| \sum\limits_{t=0}^{n-1} e^{-2\pi i \frac{t(s-1)}{n}} \mathbbm{1}_{t-t_0 \in \{0,...,K\}} \, \Psi_{t-t_0} \mathrm{u}_i^{(d)} \Big|\Big|_2^2\\
	& \leq \frac{1}{(2m+1)n} \sum\limits_{\substack{s=1 \\ \rho_n(r,s-1) \leq m}}^n \bigg(\sum\limits_{t=0}^{n-1} \mathbbm{1}_{t-t_0 \in \{0,...,K\}} \, ||\Psi_{t-t_0}||\bigg)^2\\
	& \leq \frac{1}{n} \bigg(\sum\limits_{k=0}^K ||\Psi_k||\bigg)^2 \overset{\text{(\ref{Eq_Assumption_LongRandeDependence_cK2})}}{\leq} \frac{\cK_2^2}{n} \ .
\end{align*}
We can thus apply the following lemma to $\eta \mapsto \bm{B}$ with $\kappa = \frac{\cK_2^2}{n}$.

\begin{lemma}[Bounding derivatives of Stieltjes transforms]\label{Lemma_BoundingStieltjesDerivatives}\
	\\
	Let $B : \R^N \rightarrow \C^{d \times m}$ be a smooth map with the properties
	\begin{align*}
		& \Big|\Big|\frac{\partial}{\partial {x_r}} B(x)\Big|\Big| \leq \kappa \ \ ; \ \ \frac{\partial^2}{\partial {x_{r_1}} \, \partial {x_{r_2}}} B(x) = 0 \ \ ; \ \ \mathrm{rank}\Big( \frac{\partial}{\partial {x_r}} B(x) \Big) \leq K \ ,
	\end{align*}
	then for the function
	\begin{align*}
		& f_{z}(x) \mapsto \cs_{\hat{\mu}_{B(x)B^*(x)}}(z)
	\end{align*}
	and $z \in \C^+$ the bound
	%\begin{align*}
	%& \Big| \frac{d}{dx_r} f_z(x) \Big| \leq \frac{2K||B(x)||\kappa}{d \Im(z)^2} \ ,
	%\end{align*}
	%\begin{align*}
	%& \Big| \frac{d^2}{dx_{r_1} \, dx_{r_2}} f_z(x) \Big| \leq \frac{8K ||B(x)||^2 \kappa^2}{d \Im(z)^3} + \frac{2K \kappa^2}{d \Im(z)^2}
	%\end{align*}
	%and
	\begin{align*}
		& \Big| \frac{\partial^3}{\partial x_{r_1} \, \partial x_{r_2} \, \partial x_{r_3}} f_z(x) \Big| \leq \frac{48 K ||B(x)||^3 \kappa^3}{d \Im(z)^4} + \frac{12 K ||B(x)|| \kappa^3}{d \Im(z)^3}
	\end{align*}
	holds.
\end{lemma}
\
\\
Since $S^{(K)}(\frac{2\pi r}{n}) = \bm{B} \bm{B}^*$, we for $f_z(\eta) \coloneq \cs_{\hat{\mu}_{S^{(K)}(\frac{2\pi r}{n})}}(z)$ get
\begin{align}\label{Eq_Univ_fzThirdDerivativeBound}
	& \Big| \frac{\partial^3}{(\partial (\eta_{t_0})_{i_0})^3} \cs_{\hat{\mu}_{S^{(K)}(\frac{2\pi r}{n})}}(z) \Big| \leq \frac{48 \cK_2^3 K ||\bm{B}||^3}{d \Im(z)^4 n^{\frac{3}{2}}} + \frac{12 \cK_2^3 K ||\bm{B}||}{d \Im(z)^3 n^{\frac{3}{2}}}
\end{align}
for all $z \in \C^+$, $t_0 \in \{-K,...,n-1\}$ and $i_0 \leq d$. By employing more specific and less general forms of arguments from the proof of Theorem \ref{Thm_TraceMomentBound}, we can bound $\E\big[ ||\bm{B}||^{2} \big]$ and $\E\big[ ||\bm{B}||^{6} \big]$ as follows.

\begin{lemma}[Crude trace moment bound]\label{Lemma_CrudeTraceMomentBound}\
	\\
	Under the assumptions of Theorem \ref{Thm_MarchenkoPastur_Daniell}, it for $\bm{B} \coloneq\frac{1}{\sqrt{2m+1}} \bm{X}^{(K)} V D_r^{\frac{1}{2}}$ with $\bm{X}^{(K)}$ as in (\ref{Eq_Def_SK}) holds that
	\begin{align}\label{Eq_CrudeTB_Result1}
		& \E\big[ ||\bm{B}||^{2} \big] \leq \cK_2^2 d
	\end{align}
	and
	\begin{align}\label{Eq_CrudeTB_Result2}
		& \E\big[ ||\bm{B}||^{6} \big] \leq 15 \sup\limits_{t \in \Z} \E\big[ \max_{j\leq d} |(\eta_{t})_j|^6 \big] \cK_2^6 \Big( \frac{max(d, 2m+1)^4}{(2m+1)^3} + \frac{15 d^2}{(2m+1)n} \Big) \ .
	\end{align}
\end{lemma}
\
\\
So far, none of our steps have relied on the variances of $(\eta_t)_i$ being one. We can thus replace the innovations $\big( (\xi_t)_i \big)_{i \leq d, \, t \in \{-K,n-1\}}$ relevant to $\mathcal{\bm{X}}^{(K)}$ with
\begin{align}\label{Def_zeta1}
	& \big( (\zeta^{,\eta,t_0,i_0,\tau}_t)_i \big)_{i \leq d, \, t \in \{-K,n-1\}} = {\footnotesize \left( \begin{array}{ccccccc}
		(\eta_{-K})_1 & \cdots & (\eta_{t_0-1})_1 & (\eta_{t_0})_1 & (\xi_{t_0+1})_1 & \cdots & (\xi_{n-1})_1\\
		\vdots & & \vdots & \vdots & \vdots & & \vdots\\
		\vdots & & \vdots & (\eta_{t_0})_{i_0-1} & \vdots & & \vdots\\
		\vdots & & \vdots & \tau\times(\eta_{t_0})_{i_0} & \vdots & & \vdots\\
		\vdots & & \vdots & (\xi_{t_0})_{i_0+1} & \vdots & & \vdots\\
		\vdots & & \vdots & \vdots & \vdots & & \vdots\\
		(\eta_{-K})_d & \cdots & (\eta_{t_0-1})_d & (\xi_{t_0})_d & (\xi_{t_0+1})_d & \cdots & (\xi_{n-1})_d
	\end{array} \right)}
\end{align}
or
\begin{align}\label{Def_zeta2}
	& \big( (\zeta^{\xi,t_0,i_0,\tau}_t)_i \big)_{i \leq d, \, t \in \{-K,n-1\}} = {\footnotesize \left( \begin{array}{ccccccc}
		(\eta_{-K})_1 & \cdots & (\eta_{t_0-1})_1 & (\eta_{t_0})_1 & (\xi_{t_0+1})_1 & \cdots & (\xi_{n-1})_1\\
		\vdots & & \vdots & \vdots & \vdots & & \vdots\\
		\vdots & & \vdots & (\eta_{t_0})_{i_0-1} & \vdots & & \vdots\\
		\vdots & & \vdots & \tau\times(\xi_{t_0})_{i_0} & \vdots & & \vdots\\
		\vdots & & \vdots & (\xi_{t_0})_{i_0+1} & \vdots & & \vdots\\
		\vdots & & \vdots & \vdots & \vdots & & \vdots\\
		(\eta_{-K})_d & \cdots & (\eta_{t_0-1})_d & (\xi_{t_0})_d & (\xi_{t_0+1})_d & \cdots & (\xi_{n-1})_d
	\end{array} \right)}
\end{align}
for any $\tau \in [0,1]$ and get the same bounds as (\ref{Eq_Univ_fzThirdDerivativeBound}) and Lemma \ref{Lemma_CrudeTraceMomentBound} with $C_{6,\operatorname{new}} = \max(C_6, 15)$, since the sixth moment of a standard normal random variable is $15$. We can thus apply Lemma \ref{Lemma_Lindeberg} to $f_z(\eta) \coloneq \cs_{\hat{\mu}_{S^{(K)}(\frac{2\pi r}{n})}}(z)$ and get
\begin{align}\label{Eq_UnivLindebergStep}
	& \big| \E\big[ f_z(\eta) - f_z(\xi) \big] \big| \leq \frac{\sqrt{\max(C_6, 15)}}{2} \sum\limits_{t_0=-K}^{n-1} \sum\limits_{i_0=1}^d M^{\eta}_{t_0,i_0} + M^{\xi}_{t_0,i_0} \ ,
\end{align}
where
\begin{align}\label{Eq_univ_MBound1}
	& M^{\eta}_{t_0,i_0} = \max\limits_{\tau \in [0,1]} \E\big[ |\partial_i^3 f_z(\zeta^{\eta,t_0,i_0,\tau})|^2 \big]^\frac{1}{2} \nonumber\\
	& \overset{\text{(\ref{Eq_Univ_fzThirdDerivativeBound})}}{\leq} \max\limits_{\tau \in [0,1]} \E\Bigg[ \bigg| \frac{48 \cK_2^3 K ||\bm{B}^{t_0,i_0,\tau}||^3}{d \Im(z)^4 n^{\frac{3}{2}}} + \frac{12 \cK_2^3 K ||\bm{B}^{t_0,i_0,\tau}||}{d \Im(z)^3 n^{\frac{3}{2}}} \bigg|^2 \Bigg]^{\frac{1}{2}} \nonumber\\
	& \leq \max\limits_{\tau \in [0,1]} \frac{48 \cK_2^3 K \E\big[||\bm{B}^{t_0,i_0,\tau}||^6\big]^\frac{1}{2}}{d \Im(z)^4 n^{\frac{3}{2}}} + \max\limits_{\tau \in [0,1]} \frac{12 \cK_2^3 K \E\big[||\bm{B}^{t_0,i_0,\tau}||^2\big]^{\frac{1}{2}}}{d \Im(z)^3 n^{\frac{3}{2}}} \nonumber\\
	& \overset{\text{Lemma \ref{Lemma_CrudeTraceMomentBound}}}{\leq} \frac{48 \cK_2^3 K \sqrt{C\max\limits_{\tau \in [0,1]}\sup\limits_{t \in \Z}\E\big[ \max_{j\leq d} |(\zeta^{\eta,t_0,i_0,\tau}_{t})_j|^6 \big]}}{d \Im(z)^4 n^{\frac{3}{2}}} \Big( \overbrace{\frac{\max(d, 2m+1)^4}{(2m+1)^3} + \frac{15 d^2}{(2m+1)n}}^{\leq C'' d \text{, since $\frac{d}{2m+1} \rightarrow c$ and $n \gg d$}} \Big)^{\frac{1}{2}} \nonumber\\
	& \hspace{1cm} + \frac{12 \cK_2^3 K \sqrt{C} d^{\frac{1}{2}}}{d \Im(z)^3 n^{\frac{3}{2}}} \ .
\end{align}
It is clear that 
\begin{align*}
	& \E\big[ \max_{j\leq d} |(\zeta^{\eta,t_0,i_0,\tau}_{t})_j|^6 \big] \leq \E\big[ \max_{j\leq d} |(\eta_{t})_j|^6 \big] + \E\big[ \max_{j\leq d} |(\xi_{t})_j|^6 \big]
\end{align*}
and we can for any $L \in \N$ by Jensen's inequality easily bound
\begin{align}\label{Eq_Univ_MeanMax}
	& \E\big[ \max_{j\leq d} |(\eta_{t})_j|^6 \big] \leq \E\big[ \max_{j\leq d} |(\eta_{t})_j|^{6L} \big]^{\frac{1}{L}} \leq \bigg(\sum\limits_{j=1}^d \E\big[ |(\eta_{t})_j|^{6L} \big]\bigg)^{\frac{1}{L}} \overset{\text{(\ref{Eq_Assumption_InnovationMoments})}}{\leq} \big( d C_{6L} \big)^{\frac{1}{L}} \nonumber\\
	& \E\big[ \max_{j\leq d} |(\xi_{t})_j|^6 \big] \leq \E\big[ \max_{j\leq d} |(\xi_{t})_j|^{6L} \big]^{\frac{1}{L}} \leq \bigg(\sum\limits_{j=1}^d \E\big[ |(\xi_{t})_j|^{6L} \big]\bigg)^{\frac{1}{L}} \leq \big( d (6L-1)!! \big)^{\frac{1}{L}} \ ,
\end{align}
so there exists a constant $C'_L>0$ such that
\begin{align*}
	& \sup\limits_{t \in \Z}\E\big[ \max_{j\leq d} |(\zeta^{\eta,t_0,i_0,\tau}_{t})_j|^6 \big] \leq d^{\frac{1}{L}} C'_L \ .
\end{align*}
and (\ref{Eq_univ_MBound1}) for all $z \in \C^+$ with $\Im(z) \in (0,1)$ becomes
\begin{align*}
	& M^{\eta}_{t_0,i_0} \leq \frac{48 \cK_2^3 K \sqrt{CC'_LC''} d^{\frac{1}{2L}} }{d \Im(z)^4 n^{\frac{3}{2}}} d^{\frac{1}{2}} + \frac{12 \cK_2^3 K \sqrt{C} d^{\frac{1}{2}}}{d \Im(z)^3 n^{\frac{3}{2}}} \leq \frac{\cK_2^3 \sqrt{C} (48\sqrt{C'_LC''}+12)}{\Im(z)^4} \frac{K d^{\frac{1}{2L}}}{d^{\frac{1}{2}} n^{\frac{3}{2}}}
\end{align*}
and analogously $M^{\xi}_{t_0,i_0} \leq \frac{\cK_2^3 \sqrt{C} (48\sqrt{C'_LC''}+12)}{\Im(z)^4} \frac{K d^{\frac{1}{2L}}}{d^{\frac{1}{2}} n^{\frac{3}{2}}}$. Plugging these bounds back into (\ref{Eq_UnivLindebergStep}) we see
\begin{align*}
	& \big| \E\big[ \cs_{\hat{\mu}_{\mathcal{S}^{(K)}(\frac{2\pi r}{n})}}(z) - \cs_{\hat{\mu}_{S^{(K)}(\frac{2\pi r}{n})}}(z) \big] \big| \leq \frac{\sqrt{\max(C_6, 15)}}{2} nd \, 2\frac{\cK_2^3 \sqrt{C} (48\sqrt{C'_LC''}+12)}{\Im(z)^4} \frac{K d^{\frac{1}{2L}}}{d^{\frac{1}{2}} n^{\frac{3}{2}}}\\
	& = \underbrace{\sqrt{\max(C_6, 15)} \, \cK_2^3 \sqrt{C} (48\sqrt{C'_LC''}+12)}_{=: C(L)} \frac{1}{\Im(z)^4} \frac{K d^{\frac{1}{2}+\frac{1}{2L}}}{n^{\frac{1}{2}}} \overset{\text{(\ref{Eq_Assumption_Asymptotics})}}{\leq} \frac{C(L) \cK_1}{\Im(z)^4} \frac{K n^{\frac{\alpha}{2}+\frac{\alpha}{2L}}}{n^{\frac{1}{2}}}
\end{align*}
and can choose $L \in \N$ large enough that $\delta > \frac{\alpha}{2L}$, which by (\ref{Eq_Univ_DefK}) gives (\ref{Eq_Univ_Step1}).

\subsubsection{Step 2: Proving concentration of Stieltjes transforms (\ref{Eq_Univ_Step2})}
This step is similar to its counterpart in A.3 of \cite{StrongMP}. The key ingredient will be McDiarmid's inequality, for which we must bound the effect of changing a single innovation vector $\eta_{t_0}$ on the Stieltjes transform $\cs_{S(\frac{2\pi r}{n})}$.
\\
\\
Let $(\ul{\eta}_t)_{t \in \Z}$ be a copy of the innovations $(\eta_t)_{t \in \Z}$ which differs at exactly one point in time $t_0$, i.e.
\begin{align*}
	& \ul{\eta}_t = \eta_t \ , \ \forall t \in \Z \setminus \{t_0\} \ .
\end{align*}
Define
\begin{align*}
	& X^{(K)}_t \coloneq \sum\limits_{k=0}^K \Psi_{k} \Sigma^{\frac{1}{2}} \eta_{t-k} \ \ ; \ \ \ul{X}^{(K)}_t \coloneq \sum\limits_{k=0}^K \Psi_{k} \Sigma^{\frac{1}{2}} \ul{\eta}_{t-k}
\end{align*}
and the corresponding Daniell smoothed periodograms $S^{(K)},\ul{S}^{(K)}$ in analogy to (\ref{Eq_Def_SK}). 
Next, split the matrices $\bm{X}^{(K)} = \big[X^{(K)}_0,...,X^{(K)}_{n-1}\big]$ and $\ul{\bm{X}}^{(K)} = \big[\ul{X}^{(K)}_0,...,\ul{X}^{(K)}_{n-1}\big]$ into
\begin{align*}
	& \bm{X}^{(K)}_1 = \big[ \mathbbm{1}_{t-t_0 \in \{0,...,K-1\}} X^{(K)}_t \big]_{t \in \{0,...,n-1\}} \ \ ; \ \ \bm{X}^{(K)}_2 = \big[ \mathbbm{1}_{t-t_0 \notin \{0,...,K-1\}} X^{(K)}_t \big]_{t \in \{0,...,n-1\}}\\
	& \ul{\bm{X}}^{(K)}_1 = \big[ \mathbbm{1}_{t-t_0 \in \{0,...,K-1\}} \ul{X}^{(K)}_t \big]_{t \in \{0,...,n-1\}} \ \ ; \ \ \ul{\bm{X}}^{(K)}_2 = \big[ \mathbbm{1}_{t-t_0 \notin \{0,...,K-1\}} \ul{X}^{(K)}_t \big]_{t \in \{0,...,n-1\}}
\end{align*}
and observe $\bm{X}^{(K)}_2=\ul{\bm{X}}^{(K)}_2$.
Decompose $S^{(K)}(\frac{2\pi r}{n})$ and $\ul{S}^{(K)}(\frac{2\pi r}{n})$ each into
\begin{align*}
	& S^{(K)}\Big( \frac{2\pi r}{n} \Big) = \frac{1}{2m+1} \bm{X}^{(K)}_2 V D_r V^* (\bm{X}^{(K)}_2)^\top + E\\
	& \ul{S}^{(K)}\Big( \frac{2\pi r}{n} \Big) = \frac{1}{2m+1} \ul{\bm{X}}^{(K)}_2 V D_r V^* (\ul{\bm{X}}^{(K)}_2)^\top + \ul{E} \ ,
\end{align*}
where $E$ and $\ul{E}$ are random Hermitian matrices with rank no greater than $3K$. Lemma \ref{Lemma_LowRankPerturbation} can be combined with the following lemma to for all $z \in \C$ with $\Im(z) \in (0,1)$ see
\begin{align}\label{Eq_McDiarmid_Preparation1}
	& \big| \cs_{\hat{\mu}_{S^{(K)}(\frac{2\pi r}{n})}}(z) - \cs_{\hat{\mu}_{\frac{1}{2m+1} \bm{X}^{(K)}_2 V D_r V^* (\bm{X}^{(K)}_2)^\top}}(z) \big| \leq \frac{48K}{d \Im(z)^2} \nonumber\\
	& \big| \cs_{\hat{\mu}_{\ul{S}^{(K)}(\frac{2\pi r}{n})}}(z) - \cs_{\hat{\mu}_{\frac{1}{2m+1} \ul{\bm{X}}^{(K)}_2 V D_r V^* (\ul{\bm{X}}^{(K)}_2)^\top}}(z) \big| \leq \frac{48K}{d \Im(z)^2} \ .
\end{align}

\begin{lemma}[Stieltjes transforms and the bounded Lipschitz metric]\label{Lemma_StieltjesPerturbation}\
	\\
	For any probability measures $\mu_1,\mu_2$ on $\R$ and all $z \in \C$ with $\Im(z) \in (0,1)$ holds that 
	\begin{align}\label{Eq_StieltjesPerturbation}
		& \frac{\Im(z)^2}{2} |\cs_{\mu_1}(z) - \cs_{\mu_2}(z)| \leq \, d_{\operatorname{BL}}(\mu_1,\mu_2) \ .
	\end{align}
\end{lemma}
\begin{proof}\
	\\
	For any $z \in \C^+$ the function $f_z : \R \rightarrow \C \ ; \ x \mapsto \frac{1}{x-z}$ is easily seen to satisfy
	\begin{align*}
		& |f_z(x)| \leq \frac{1}{\Im(z)} \ \ \text{ and } \ \ |f_z(x) - f_z(y)| \leq \frac{|x-y|}{\Im(z)^2}
	\end{align*}
	for all $x,y \in \R$. Consequently, for all $z \in \C^+$ with $\Im(z) \in (0,1)$, we have
	\begin{align*}
		& \Im(z)^2 (\Re \circ f_z) , \, \Im(z)^2 (\Im \circ f_z) \in \operatorname{Lip}_1 \ .
	\end{align*}
	The definitions of $\operatorname{d}_{BL}$ and $\cs_{\mu}$ yield the wanted bound.
\end{proof}
\
\\
Since $\bm{X}^{(K)}_2=\ul{\bm{X}}^{(K)}_2$, (\ref{Eq_McDiarmid_Preparation1}) shows
\begin{align}\label{Eq_McDiarmid_Preparation2}
	& \big| \cs_{\hat{\mu}_{S^{(K)}(\frac{2\pi r}{n})}}(z) - \cs_{\hat{\mu}_{\ul{S}^{(K)}(\frac{2\pi r}{n})}}(z) \big| \leq \frac{96K}{d \Im(z)^2}
\end{align}
and we are ready to apply McDiarmid's inequality. Define the map
\begin{align*}
	& f_{r,z} : (\xi_{-K},...,\xi_{n-1}) \mapsto \cs_{\hat{\mu}_{\mathcal{S}^{(K)}(\frac{2\pi r}{n})}}(z)
\end{align*}
and use McDiarmind's inequality (Theorem 6.2 in \cite{BoucheronConcentrationInequalities}) with (\ref{Eq_McDiarmid_Preparation2}) for
\begin{align*}
	& \bP\Big( \Big| f_{r,z}(\xi_{-K},...,\xi_{n-1}) - \E[f_{r,z}(\xi_{-K},...,\xi_{n-1})] \Big| > \varepsilon \Big) \leq 2 \exp\bigg( -\frac{4 \varepsilon^2}{(\underbrace{n+K}_{\leq 2n}) \big(\frac{96 K}{d \Im(z)^2}\big)^2} \bigg) \ ,
\end{align*}
which by the choice of $f_{r,z}$ can be written as
\begin{align}\label{Eq_McDiarmidApplication}
	& \bP\Big( \Big| \cs_{\hat{\mu}_{S^{(K)}(\frac{2\pi r}{n})}}(z) - \E[\cs_{\hat{\mu}_{S^{(K)}(\frac{2\pi r}{n})}}(z)] \Big| > \varepsilon \Big) \leq 2 \exp\bigg( -\frac{2\varepsilon^2 d^2 \Im(z)^4}{n 96^2 K^2} \bigg)
\end{align}
holds for all $\varepsilon > 0$, $r \in \{0,...,n-1\}$ and $z \in \C$ with $\Im(z) \in (0,1)$. By (\ref{Eq_Assumption_Asymptotics}) and (\ref{Eq_Univ_DefK}) the right hand side is less than $2 \exp\big( -C\varepsilon^2 \Im(z)^4 n^{2\delta} \big)$ for some constant $C>0$. We proceed with an $\varepsilon$-net argument to follow (\ref{Eq_Univ_Step2}).
\\
\\
The Stieltjes transform satisfies the Lipschitz property
\begin{align*}
	& \big|\cs_{\mu}(z_1) - \cs_{\mu}(z_2)\big| \leq \int_\R \Big| \frac{1}{\lambda-z_1} - \frac{1}{\lambda-z_2} \Big| \, d\mu(\lambda)\\
	& = \int_\R \frac{|z_2-z_1|}{|(\lambda-z_1)(\lambda-z_2)|} \, d\mu(\lambda) \leq \frac{|z_1-z_2|}{\Im(z)^2}
\end{align*}
for any probability measure $\mu$ on $\R$.
Let $J_n$ be a $\frac{1}{n}$-net of the set $Q_n = [-n,n] \times i[n^{-\delta/4},n]$, which can be chosen with $\#J_n \leq 4n^4$. The above Lipschitz property implies
\begin{align*}
	& \sup\limits_{z \in Q_n} \Big| \cs_{\hat{\mu}_{\mathcal{S}^{(K)}(\frac{2\pi r}{n})}}(z) - \E[\cs_{\hat{\mu}_{\mathcal{S}^{(K)}(\frac{2\pi r}{n})}}(z)] \Big| \leq \sup\limits_{z \in J_n} \Big| \cs_{\hat{\mu}_{\mathcal{S}^{(K)}(\frac{2\pi r}{n})}}(z) - \E[\cs_{\hat{\mu}_{\mathcal{S}^{(K)}(\frac{2\pi r}{n})}}(z)] \Big| + 2\frac{1/n}{n^{-\delta/2}}
\end{align*}
and we see
\begin{align*}
	& \bP\Big( \sup\limits_{r \in \{0,...,n-1\}} \sup\limits_{z \in Q_n} \Big| \cs_{\hat{\mu}_{\mathcal{S}^{(K)}(\frac{2\pi r}{n})}}(z) - \E[\cs_{\hat{\mu}_{\mathcal{S}^{(K)}(\frac{2\pi r}{n})}}(z)] \Big| > \varepsilon \Big)\\
	& \leq \bP\Big( \sup\limits_{r \in \{0,...,n-1\}} \sup\limits_{z \in J_n} \Big| \cs_{\hat{\mu}_{\mathcal{S}^{(K)}(\frac{2\pi r}{n})}}(z) - \E[\cs_{\hat{\mu}_{\mathcal{S}^{(K)}(\frac{2\pi r}{n})}}(z)] \Big| + 2n^{\delta/2-1} > \varepsilon \Big)\\
	& \leq \sum\limits_{r=0}^{n-1} \sum\limits_{z \in J_n} \bP\Big( \Big| \cs_{\hat{\mu}_{\mathcal{S}^{(K)}(\frac{2\pi r}{n})}}(z) - \E[\cs_{\hat{\mu}_{\mathcal{S}^{(K)}(\frac{2\pi r}{n})}}(z)] \Big| > \varepsilon - 2n^{\delta/2-1} \Big)\\
	& \hspace{-0.25cm} \overset{\text{(\ref{Eq_McDiarmidApplication})}}{\leq} \sum\limits_{r=0}^{n-1} \sum\limits_{z \in J_n} 2 \exp\big( -C(\varepsilon-2n^{\delta/2-1})^2 \Im(z)^4 n^{2\delta} \big)\\
	& \leq n \, 4n^4 \, 2 \exp\big( -C(\varepsilon-2n^{\delta/2-1})^2 n^{-\delta} n^{2\delta} \big) = 8n^5 \exp\big( -C(\varepsilon-2n^{\delta/2-1})^2 n^{\delta} \big) \ ,
\end{align*}
where the right hand side is clearly summable over $n$ and Borel-Cantelli yields
\begin{align*}
	& 1 = \bP\Big( \forall\theta\in [0,2\pi) \, \forall z \in \C^+ : \ \cs_{\hat{\mu}_{S^{(K)}(\theta)}}(z) - \E[\cs_{\hat{\mu}_{S^{(K)}(\theta)}}(z)] \xrightarrow{n \rightarrow \infty} 0 \Big) \ .
\end{align*}
This has worked completely without examining the distributions of the innovations $(\eta_t)_{t \in \Z}$, so the above property also holds for $\mathcal{S}^{(K)}$ and we have shown (\ref{Eq_Univ_Step2}).

\subsubsection{Step 3: Proving approximation by finite linear processes (\ref{Eq_Univ_Step3})}

With standard analysis we may make use of the bounds (\ref{Eq_Univ_MeanMax}) and (\ref{Eq_Assumption_LongRandeDependence_gamma}) to get the following result.

\begin{lemma}[Approximation by a finite linear process]\label{Lemma_Approximation_FiniteLinear}\
	\\
	Under the conditions of Theorem \ref{Thm_MarchenkoPastur_Daniell} for fixed $K \in \N$ and $S^{(K)}$ as in (\ref{Eq_Def_SK}) the bound
	\begin{align*}
		& \E\Big[ \big| \cs_{\hat{\mu}_{S(\frac{2\pi r}{n})}}(z) - \cs_{\hat{\mu}_{S^{(K)}(\frac{2\pi r}{n})}}(z) \big|^6 \Big] \leq C(L) \frac{d^{\frac{2}{L}} K^{-6\gamma}}{\Im(z)^4}
	\end{align*}
	holds for any $L \in \N$ and some constant $C(L) > 0$.
\end{lemma}
\
\\
We proceed with an $\varepsilon$-net argument to show (\ref{Eq_Univ_Step3}). It is here that Assumption (\ref{Eq_Assumption_UniversalityRequirement}) is needed.
\\
\\
Plugging the definition (\ref{Eq_Univ_DefK}) of $K$ into the above Lemma we get
\begin{align*}
	& \frac{C(L) \cK_2^{\frac{2}{L}}}{\Im(z)^4} n^{\frac{2\alpha}{L}} n^{-6\gamma (\min(\frac{1-\alpha}{2}, \alpha-\frac{1}{2})-\delta)} \ .
\end{align*}
For any fixed small $\iota > 0$ let $J_n$ be an $n^{-2\iota}$-net of the set $Q_n = [-n^{\iota},n^{\iota}] \times i[n^{-\iota},n^{\iota}]$, which can be chosen with $\#J_n \leq 4n^{4\iota}$ and analogously to Step 2 we have
\begin{align*}
	& \sup\limits_{z \in Q_n} \Big| \cs_{\hat{\mu}_{S^{(K)}(\frac{2\pi r}{n})}}(z) - \cs_{\hat{\mu}_{S^{(K)}(\frac{2\pi r}{n})}}(z) \Big| \leq \sup\limits_{z \in J_n} \Big| \cs_{\hat{\mu}_{S^{(K)}(\frac{2\pi r}{n})}}(z) - \cs_{\hat{\mu}_{S^{(K)}(\frac{2\pi r}{n})}}(z) \Big| + 2n^{-\iota}
\end{align*}
We calculate
\begin{align*}
	& \bP\Big( \forall r \in \{0,...,n-1\} , \, \forall t \in Q_n : \ \Big| \cs_{\hat{\mu}_{S^{(K)}(\frac{2\pi r}{n})}}(z) - \cs_{\hat{\mu}_{S(\frac{2\pi r}{n})}}(z) \Big| \geq \varepsilon \Big)\\
	& \leq \bP\Big( \forall r \in \{0,...,n-1\} , \, \forall t \in J_n : \ \Big| \cs_{\hat{\mu}_{S^{(K)}(\frac{2\pi r}{n})}}(z) - \cs_{\hat{\mu}_{S(\frac{2\pi r}{n})}}(z) \Big| \geq \varepsilon - 2n^{-\iota} \Big)\\
	& \leq \sum\limits_{r=0}^{n-1} \sum\limits_{z \in J_n} \bP\Big( \Big| \cs_{\hat{\mu}_{S^{(K)}(\frac{2\pi r}{n})}}(z) - \cs_{\hat{\mu}_{S(\frac{2\pi r}{n})}}(z) \Big| \geq \varepsilon - 2n^{-\iota} \Big)\\
	& \leq \sum\limits_{r=0}^{n-1} \sum\limits_{z \in J_n} (\varepsilon - 2n^{-\iota})^{-6} \frac{C(L) \cK_2^{\frac{2}{L}}}{\Im(z)^4} n^{\frac{2\alpha}{L}} n^{-6\gamma (\min(\frac{1-\alpha}{2}, \alpha-\frac{1}{2})-\delta)}\\
	& \leq n \, 4n^{4\iota} \, (\varepsilon - 2n^{-\iota})^{-6} \frac{C(L) \cK_2^{\frac{2}{L}}}{n^{-4\iota}} n^{\frac{2\alpha}{L}} n^{-6\gamma (\min(\frac{1-\alpha}{2}, \alpha-\frac{1}{2})-\delta)}\\
	& = 4C(L)\cK_2^{\frac{2}{L}} \, (\varepsilon - 2n^{-\iota})^{-6} \, n^{8\iota+\frac{2\alpha}{L}} \, n^{1-6\gamma (\min(\frac{1-\alpha}{2}, \alpha-\frac{1}{2})-\delta)}
\end{align*}
The assumption
\begin{align*}
	& \gamma > \frac{1}{3\min(\frac{1-\alpha}{2}, \alpha-\frac{1}{2})}
\end{align*}
for small enough $\delta, \iota > 0$ and large enough $L \in \N$ guarantees the summability of the right hand side, giving
\begin{align*}
	& 1 = \bP\Big( \forall \theta \in [0,2\pi) \, \forall z \in \C^+ : \ \cs_{\hat{\mu}_{S^{(K)}(\frac{2\pi r}{n})}}(z) - \cs_{\hat{\mu}_{S(\frac{2\pi r}{n})}}(z) \xrightarrow{n \rightarrow \infty} 0 \Big) \ .
\end{align*}
The same arguments work for $\mathcal{S}^{(K)}$ and we have proven (\ref{Eq_Univ_Step3}).
\\
\\
We finally combine the three steps (\ref{Eq_Univ_Step1}), (\ref{Eq_Univ_Step2}) and (\ref{Eq_Univ_Step3}) into
\begin{align*}
	& 1 = \bP\Big( \forall \theta \in [0,2\pi) \, \forall z \in \C^+ : \ \cs_{\hat{\mu}_{\mathcal{S}(\theta)}}(z) - \cs_{\hat{\mu}_{S(\theta)}}(z) \xrightarrow{n \rightarrow \infty} 0 \Big) \ .
\end{align*}
As we have already shown Theorem \ref{Thm_MarchenkoPastur_Daniell} in the Gaussian case, we by Theorem 2.4.4 of \cite{AndersonIRM} know
\begin{align*}
	& 1 = \bP\Big( \forall \theta \in [0,2\pi) \, \forall z \in \C^+ : \ \cs_{\hat{\mu}_{\mathcal{S}(\theta)}}(z) \xrightarrow{n \rightarrow \infty} \cs_{\nu_\infty(\theta)}(z) \Big) \ ,
\end{align*}
which with the above result directly yields
\begin{align*}
	& 1 = \bP\Big( \forall \theta \in [0,2\pi) \, \forall z \in \C^+ : \ \cs_{\hat{\mu}_{S(\theta)}}(z) \xrightarrow{n \rightarrow \infty} \cs_{\nu_\infty(\theta)}(z) \Big) \ .
\end{align*}
Again by Theorem 2.4.4 of \cite{AndersonIRM} this is equivalent to
\begin{align*}
	& 1 = \bP\Big( \forall \theta \in [0,2\pi) : \ \hat{\mu}_{S(\theta)} \xRightarrow{n \rightarrow \infty} \nu_\infty(\theta) \Big) \ .
\end{align*}
This concludes the proof of Theorem \ref{Thm_MarchenkoPastur_Daniell}. \qed

%%%%%%%%%%%%%%%%%%%%%%%%%%%
\newpage
\appendix
\section{Proofs of lemmas}

\subsection{Proof of Lemma \ref{Lemma_AlmostWishart}}\label{Proof_Lemma_AlmostWishart}
The entries of $\bm{\eta} V$ are by construction complex Gaussian with covariance structure
\begin{align*}
	& \E\big[ (\bm{\eta}V)_{j_1,s_1} (\bm{\eta}V)_{j_2,s_2} \big] = \mathbbm{1}_{j_1=j_2} \, \mathbbm{1}_{s_1 = n+2-s_2} \ \ ; \ \ \E\big[ (\bm{\eta}V)_{j_1,s_1} \ol{(\bm{\eta}V)_{j_2,s_2}} \big] = \mathbbm{1}_{j_1=j_2} \, \mathbbm{1}_{s_1 = s_2} \ ,
\end{align*}
as we see by the calculations
\begin{align*}
	& \E\big[ (\bm{\eta}V)_{j_1,s_1} (\bm{\eta}V)_{j_2,s_2} \big] = \frac{1}{n} \sum\limits_{l_1,l_2=1}^n \E\big[ \bm{\eta}_{j_1,l_1} \bm{\eta}_{j_2,l_2} \big] e^{-2\pi i \frac{(l_1-1)(s_1-1)}{n}} e^{-2\pi i \frac{(l_2-1)(s_2-1)}{n}}\\
	& = \mathbbm{1}_{j_1=j_2} \, \frac{1}{n} \sum\limits_{l=1}^n e^{-2\pi i \frac{(l-1)(s_1+s_2-2)}{n}} = \mathbbm{1}_{j_1=j_2} \, \mathbbm{1}_{s_1+s_2-2 = 0 \mod n} = \mathbbm{1}_{j_1=j_2} \, \mathbbm{1}_{s_1 = n+2-s_2}
\end{align*}
and
\begin{align*}
	& \E\big[ (\bm{\eta}V)_{j_1,s_1} \ol{(\bm{\eta}V)_{j_2,s_2}} \big] = \frac{1}{n} \sum\limits_{l_1,l_2=1}^n \E\big[ \bm{\eta}_{j_1,l_1} \bm{\eta}_{j_2,l_2} \big] e^{-2\pi i \frac{(l_1-1)(s_1-1)}{n}} e^{2\pi i \frac{(l_2-1)(s_2-1)}{n}}\\
	& = \mathbbm{1}_{j_1=j_2} \, \frac{1}{n} \sum\limits_{l=1}^n e^{-2\pi i \frac{(l-1)(s_1-s_2)}{n}} = \mathbbm{1}_{j_1=j_2} \, \mathbbm{1}_{s_1-s_2 = 0 \mod n} = \mathbbm{1}_{j_1=j_2} \, \mathbbm{1}_{s_1 = s_2} \ .
\end{align*}
It follows that:
\begin{itemize}
	%\item[a)] The first half of the columns $(\bm{\eta}V)_{\bullet,1},...,(\bm{\eta}V)_{\bullet,\lceil \frac{n}{2} \rceil}$ are independent.
	
	\item[i)] The entries of $\big[ (\bm{\eta}V)_{\bullet,1},...,(\bm{\eta}V)_{\bullet,\lfloor \frac{n}{2} \rfloor} \big]$ are i.i.d. standard complex Gaussian. Additionally, for odd $n$, the entries of $(\bm{\eta}V)_{\bullet,\lceil \frac{n}{2} \rceil}$ are i.i.d. standard normal and also independent of the other columns.
	
	\item[ii)] The second half of the columns is the complex conjugate of the first half in the sense that $(\bm{\eta}V)_{\bullet,n+2-s} = \ol{(\bm{\eta}V)_{\bullet,s}}$ for all $s \in \{2,...,n\}$.
\end{itemize}
Property (i) follows immediately from the described covariance/relation structure. Property (ii) follows from the calculation
\begin{align*}
	& \E\big[ \big| (\bm{\eta}V)_{j,n+2-s} - \ol{(\bm{\eta}V)_{j,s}} \big|^2 \big] = \E\big[ \big( (\bm{\eta}V)_{j,n+2-s} - \ol{(\bm{\eta}V)_{j,s}} \big) \big( \ol{(\bm{\eta}V)_{j,n+2-s}} - (\bm{\eta}V)_{j,s} \big) \big]\\
	& = \E\big[ (\bm{\eta}V)_{j,n+2-s} \ol{(\bm{\eta}V)_{j,n+2-s}} \big] - \E\big[ (\bm{\eta}V)_{j,n+2-s} (\bm{\eta}V)_{j,s} \big]\\
	& \hspace{0.5cm} - \E\big[ \ol{(\bm{\eta}V)_{j,s}} \ol{(\bm{\eta}V)_{j,n+2-s}} \big] + \E\big[ \ol{(\bm{\eta}V)_{j,s}} (\bm{\eta}V)_{j,s} \big]\\
	& = 1 - 1 - 1 + 1 = 0 \ .
\end{align*}
With these preliminaries, we now start the proof of (a) and (b).
\begin{itemize}
	\item[a)]
	By definition of $D_r$ and $\rho_n(r,0) > m$ we see
	\begin{align*}
		& \bm{\eta} V D_r^{\frac{1}{2}} = \big[ (\bm{\eta}V)_{\bullet,r+1-m},...,(\bm{\eta}V)_{\bullet,r+1+m} \big] \ .
	\end{align*}
	The assumptions $\rho_n(r,0) > m$ and $\rho_n(r,\frac{n}{2}) > m$ ensure that the range $r+1-m,...,r+1+m$ is either entirely in the first or second half of $\{1,...,n\}$. Property (i) thus guarantees that the entries of $\bm{\eta} V D_r^{\frac{1}{2}}$ are i.i.d. standard complex Gaussian, which proves (a).
	
	\item[b)]
	For $r=0$ we simply with the definition of $D_0$ and (ii) calculate
	\begin{align*}
		& \bm{\eta} V D_0 V^* \bm{\eta}^\top = \big[ (\bm{\eta}V)_{\bullet,1},...,(\bm{\eta}V)_{\bullet,m+1},(\bm{\eta}V)_{\bullet,n-m+1},...,(\bm{\eta}V)_{\bullet,n} \big] \nonumber\\
		& \hspace{3cm} \times \big[ (\bm{\eta}V)_{\bullet,1},...,(\bm{\eta}V)_{\bullet,m+1},(\bm{\eta}V)_{\bullet,n-m+1},...,(\bm{\eta}V)_{\bullet,n} \big]^* \nonumber\\
		& = \big[ (\bm{\eta}V)_{\bullet,1},\sqrt{2}\Re\big((\bm{\eta}V)_{\bullet,2}\big),...,\sqrt{2}\Re\big((\bm{\eta}V)_{\bullet,m+1}\big),\\
		& \hspace{3cm} \sqrt{2}\Im\big((\bm{\eta}V)_{\bullet,2}\big),...,\sqrt{2}\Im\big((\bm{\eta}V)_{\bullet,m+1}\big) \big] \nonumber\\
		& \hspace{0.5cm} \times \big[ (\bm{\eta}V)_{\bullet,1},\sqrt{2}\Re\big((\bm{\eta}V)_{\bullet,2}\big),...,\sqrt{2}\Re\big((\bm{\eta}V)_{\bullet,m+1}\big),\\
		& \hspace{4cm}\sqrt{2}\Im\big((\bm{\eta}V)_{\bullet,2}\big),...,\sqrt{2}\Im\big((\bm{\eta}V)_{\bullet,m+1}\big) \big]^* \ .
	\end{align*}
	Let $\xi \sim \mathcal{N}(0,\frac{1}{2}\operatorname{Id}_d)$ be independent of $\bm{\eta}$.
	With the notation
	\begin{align*}
		& \bm{Z}_0 \coloneq \sqrt{2}\big[\xi, \Re\big((\bm{\eta}V)_{\bullet,2}\big),...,\Re\big((\bm{\eta}V)_{\bullet,m+1}\big), \Im\big((\bm{\eta}V)_{\bullet,2}\big),...,\Im\big((\bm{\eta}V)_{\bullet,m+1}\big) \big]
	\end{align*}
	we have shown
	\begin{align*}
		& \bm{\eta} V D_0 V^* \bm{\eta}^\top - \bm{Z}_0\bm{Z}_0^\top\\
		& = \underbrace{([(\bm{\eta}V)_{\bullet,1}]-\xi) \bm{Z}_0^\top + \bm{Z}_0([(\bm{\eta}V)_{\bullet,1}]-\xi)^\top + ([(\bm{\eta}V)_{\bullet,1}]-\xi)([(\bm{\eta}V)_{\bullet,1}]-\xi)^\top}_{= E_0} \ .
	\end{align*}
	It is clear that $\mathrm{rank}(E_0) \leq 3$ and $\bm{Z}_0$ by (i) and (ii) has real standard normal entries.\\
	For $r=\lfloor \frac{n}{2} \rfloor$ the calculations differ slightly depending on the parity of $n$. If $n$ is even, we have $r=\frac{n}{2}$ and analogously to the ($r=0$)-case have
	\begin{align*}
		& \bm{\eta} V D_{\frac{n}{2}} V^* \bm{\eta}^\top = \big[ (\bm{\eta}V)_{\bullet,\frac{n}{2}-m+1},...,(\bm{\eta}V)_{\bullet,\frac{n}{2}+m+1} \big] \times \big[ (\bm{\eta}V)_{\bullet,\frac{n}{2}-m+1},...,(\bm{\eta}V)_{\bullet,\frac{n}{2}+m+1} \big]^* \nonumber\\
		& = \big[ (\bm{\eta}V)_{\bullet,\frac{n}{2}},\sqrt{2}\Re(\bm{\eta}V)_{\bullet,\frac{n}{2}-m+1},...,\sqrt{2}\Re(\bm{\eta}V)_{\bullet,\frac{n}{2}-1},\\
		& \hspace{3cm} \sqrt{2}\Im(\bm{\eta}V)_{\bullet,\frac{n}{2}-m+1},...,\sqrt{2}\Im(\bm{\eta}V)_{\bullet,\frac{n}{2}-1} \big] \nonumber\\
		& \hspace{0.3cm} \times \big[ (\bm{\eta}V)_{\bullet,\frac{n}{2}},\sqrt{2}\Re(\bm{\eta}V)_{\bullet,\frac{n}{2}-m+1},...,\sqrt{2}\Re(\bm{\eta}V)_{\bullet,\frac{n}{2}-1},\\
		& \hspace{4cm} \sqrt{2}\Im(\bm{\eta}V)_{\bullet,\frac{n}{2}-m+1},...,\sqrt{2}\Im(\bm{\eta}V)_{\bullet,\frac{n}{2}-1} \big]^* \ .
	\end{align*}
	With the notation
	\begin{align*}
		& \bm{Z}_1 = \sqrt{2}\big[\xi, \Re(\bm{\eta}V)_{\bullet,\frac{n}{2}-m+1},...,\Re(\bm{\eta}V)_{\bullet,\frac{n}{2}-1},\Im(\bm{\eta}V)_{\bullet,\frac{n}{2}-m+1},...,\Im(\bm{\eta}V)_{\bullet,\frac{n}{2}-1} \big]
	\end{align*}
	it follows that
	\begin{align*}
		& \bm{\eta} V D_{\frac{n}{2}} V^* \bm{\eta}^\top - \bm{Z}_1\bm{Z}_1^\top\\
		& = \underbrace{\bm{Z}_1([(\bm{\eta}V)_{\bullet,\frac{n}{2}}]-\xi)^\top + (\bm{\eta}V)_{\bullet,\frac{n}{2}}]-\xi) \bm{Z}_1 + (\bm{\eta}V)_{\bullet,\frac{n}{2}}]-\xi)(\bm{\eta}V)_{\bullet,\frac{n}{2}}]-\xi)^\top}_{= E_1}
	\end{align*}
	and again $\mathrm{rank}(E_1) \leq 3$ and $\bm{Z}_1$ has real standard normal entries by (i) and (ii).\\
	Lastly, when $n$ is odd and $r=\lfloor \frac{n}{2} \rfloor = \frac{n-1}{2}$, we have
	\begin{align*}
		& \bm{\eta} V D_{\frac{n-1}{2}} V^* \bm{\eta}^\top = \big[ (\bm{\eta}V)_{\bullet,\frac{n-1}{2}-m+1},...,(\bm{\eta}V)_{\bullet,\frac{n-1}{2}+m+1} \big]\\
		& \hspace{5cm} \times \big[ (\bm{\eta}V)_{\bullet,\frac{n-1}{2}-m+1},...,(\bm{\eta}V)_{\bullet,\frac{n-1}{2}+m+1} \big]^* \nonumber\\
		& = \big[ (\bm{\eta}V)_{\bullet,\frac{n-1}{2}+m+1},\sqrt{2}\Re(\bm{\eta}V)_{\bullet,\frac{n-1}{2}-m+1},...,\sqrt{2}\Re(\bm{\eta}V)_{\bullet,\frac{n-1}{2}},\\
		& \hspace{3.5cm} \sqrt{2}\Im(\bm{\eta}V)_{\bullet,\frac{n-1}{2}-m+1},...,\sqrt{2}\Im(\bm{\eta}V)_{\bullet,\frac{n-1}{2}} \big] \nonumber\\
		& \hspace{0.3cm} \times \big[ (\bm{\eta}V)_{\bullet,\frac{n-1}{2}+m+1},\sqrt{2}\Re(\bm{\eta}V)_{\bullet,\frac{n-1}{2}-m+1},...,\sqrt{2}\Re(\bm{\eta}V)_{\bullet,\frac{n-1}{2}},\\
		& \hspace{4cm}\sqrt{2}\Im(\bm{\eta}V)_{\bullet,\frac{n-1}{2}-m+1},...,\sqrt{2}\Im(\bm{\eta}V)_{\bullet,\frac{n-1}{2}} \big]^*
	\end{align*}
	and the notation
	\begin{align*}
		& \bm{Z}_2 = \sqrt{2}\big[\xi, \Re(\bm{\eta}V)_{\bullet,\frac{n-1}{2}-m+1},...,\Re(\bm{\eta}V)_{\bullet,\frac{n-1}{2}}, \Im(\bm{\eta}V)_{\bullet,\frac{n-1}{2}-m+1},...,\Im(\bm{\eta}V)_{\bullet,\frac{n-1}{2}} \big]
	\end{align*}
	again allows
	\begin{align*}
		& \bm{\eta} V D_{\frac{n-1}{2}} V^* \bm{\eta}^\top = \bm{Z}_2\bm{Z}_2^\top + E_2
	\end{align*}
	for $\mathrm{rank}(E_2)\leq 3$ and (i) and (ii) also shows that $\bm{Z}_2$ has i.i.d. real standard normal entries. \qed
\end{itemize}

\subsection{Proof of Lemma \ref{Lemma_LowRankPerturbation}}\label{Proof_Lemma_LowRankPerturbation}
Let $K$ denote the maximum of $||A||$ and $||A+E||$ such that all eigenvalues under consideration are in the interval $[-K,K]$. As the rank of $E$ is no larger than $\rho$, the eigenvalue interlacing theorem (see for example Theorem 2.12 of \cite{Garoni}) states
\begin{align}\label{Eq_Interlacing0}
	& \lambda_{j-\rho}(A) \geq \lambda_j(A+E) \geq \lambda_{j+\rho}(A)
\end{align}
for all $j \in \{1,...,d\}$, where we can by construction of $K$ set $\lambda_{-k}(A) \coloneq -K$ and $\lambda_{d+1+k}(A) \coloneq K$ for all $k \geq 0$. %Making use of this interlacing property requires a little setup.
\\
\\
For any fixed $F \in \operatorname{Lip}_1(\R)$ we are only interested in its behavior on the interval $[-K,K]$. Define its total variation $V: [-K,K] \rightarrow [0,\infty)$ as
\begin{align*}
	& V(x) \coloneq \sup\limits_{N \in \N} \, \sup\limits_{-K \leq t_0 \leq ... \leq t_N \leq x} \, \sum\limits_{j=1}^N |F(t_{j}) - F(t_{j-1})| \ ,
\end{align*}
then the functions $F_1,F_2 : [-K,K] \rightarrow \R$ defined by
\begin{align*}
	& F_1 = V \ \ \text{ and } \ \ F_2 = V-F+F(-K)
\end{align*}
are by construction non-decreasing with $F_1(-K) = 0 = F_2(-K)$ and satisfy $F - F(-K) = F_1-F_2$. One easily checks $F_1,F_2 \in \operatorname{Lip}_2(\R)$. The fact that $F_1$ is non-decreasing and the interlacing property (\ref{Eq_Interlacing0}) together yield
\begin{align*}
	& \bigg| \frac{1}{d} \sum\limits_{j=1}^d \big( F_1\big( \lambda_j(A) \big) - F_1\big( \lambda_j(A+E) \big) \big) \bigg|\\
	& \leq \frac{1}{d} \sum\limits_{j=1}^d \big| F_1\big( \lambda_j(A) \big) - F_1\big( \lambda_j(A+E) \big) \big|\\
	& \leq \frac{1}{d} \sum\limits_{j=1}^d \big( F_1( \lambda_{j-\rho}(A)) - F_1( \lambda_{j+\rho}(A)) \big)\\
	%& = \frac{1}{d} \sum\limits_{j=1}^d F_1( \lambda_{j-\rho}(A)) - \frac{1}{d} \sum\limits_{j=1}^d F_1( \lambda_{j+\rho}(A))\\
	& = \frac{1}{d} \sum\limits_{j=1-\rho}^{\rho} F_1( \lambda_{j}(A)) - \frac{1}{d} \sum\limits_{j=d+1-\rho}^{d+\rho} \underbrace{F_1( \lambda_{j}(A))}_{\geq 0}\\
	& \leq \frac{1}{d} \sum\limits_{j=1-\rho}^{\rho} F_1( \lambda_{j}(A)) \leq \frac{2\rho}{d} ||F_1||_{[-K,K]} \leq \frac{4\rho}{d} \ .
\end{align*}
The very same steps can be applied for $F_2$ and we thus see
\begin{align*}
	& d_{\operatorname{BL}}(\hat{\mu}_{A},\hat{\mu}_{A+E}) \nonumber\\
	& = \sup\limits_{F \in \operatorname{Lip}_1(\R)} \bigg| \frac{1}{d} \sum\limits_{j=1}^d \big( F\big( \lambda_j(A) \big) - F\big( \lambda_j(A+E) \big) \big) \bigg| \nonumber\\
	& \leq \sup\limits_{F \in \operatorname{Lip}_1(\R)} \bigg( \bigg| \frac{1}{d} \sum\limits_{j=1}^d \big( F_1\big( \lambda_j(A) \big) - F_1\big( \lambda_j(A+E) \big) \big) \bigg| \nonumber\\
	& \hspace{2.5cm} + \bigg| \frac{1}{d} \sum\limits_{j=1}^d \big( F_2\big( \lambda_j(A) \big) - F_2\big( \lambda_j(A+E) \big) \big) \bigg| \bigg) \leq \frac{8\rho}{d} \ . \qed
\end{align*}

\subsection{Proof of Lemma \ref{Lemma_ApproxNeighboring}}
By (\ref{Eq_DaniellAsProduct}), the Daniell smoothed periodogram $S\big(\frac{2\pi r}{n}\big)$ has the form $Z Z^*$ for
\begin{align*}
	& Z = \frac{1}{\sqrt{2m+1}} \bm{X} V D_r^{\frac{1}{2}}
\end{align*}
and by definition (\ref{Eq_Def_tS}), $\tilde{S}$ has the form $\tilde{S}\big(\frac{2\pi r}{n}\big) = \tilde{Z} \tilde{Z}^*$ for
\begin{align*}
	& \tilde{Z} = \frac{1}{\sqrt{2m+1}} \tilde{\bm{X}} V D_r^{\frac{1}{2}} \ .
\end{align*}
Define
\begin{align*}
	& \mathbb{D}_1 \coloneq \operatorname{diag}(\underbrace{1,...,1}_{\times K},\underbrace{0,...,0}_{\times (n-K)}) \ \ ; \ \ \mathbb{D}_2 \coloneq \operatorname{diag}(\underbrace{0,...,0}_{\times K},\underbrace{1,...,1}_{\times (n-K)})
\end{align*}
and
\begin{align*}
	& Z_{1/2} \coloneq \frac{1}{\sqrt{2m+1}} \bm{X} \mathbb{D}_{1/2} V D_r^{\frac{1}{2}} \ \ ; \ \ \tilde{Z}_{1/2} \coloneq \frac{1}{\sqrt{2m+1}} \tilde{\bm{X}} \mathbb{D}_{1/2} V D_r^{\frac{1}{2}} \ ,
\end{align*}
then the matrix
\begin{align}\label{Eq_NeighboringE}
	& E \coloneq S\Big( \frac{2\pi r}{n} \Big) - Z_2 Z_2^* = (Z_1+Z_2)(Z_1+Z_2)^* - Z_2 Z_2^* = Z_1 Z_1^* + Z_1 Z_2^* + Z_2 Z_1^*
\end{align}
has at most rank $3K$ and Lemma \ref{Lemma_LowRankPerturbation} gives
\begin{align}\label{Eq_dBL_bound1}
	& d_{\operatorname{BL}}\big( \hat{\mu}_{S( \frac{2\pi r}{n})}, \hat{\mu}_{Z_2 Z_2^*} \big) \leq \frac{24K}{d}
\end{align}
and analogously
\begin{align}\label{Eq_dBL_bound2}
	& d_{\operatorname{BL}}\big( \hat{\mu}_{\tilde{S}( \frac{2\pi r}{n})}, \hat{\mu}_{\tilde{Z}_2 \tilde{Z}_2^*} \big) \leq \frac{24K}{d} \ .
\end{align}
We continue to bound the norm $||Z_2 Z_2^* - \tilde{Z}_2 \tilde{Z}_2^*||$. A basic identity is
\begin{align}\label{Eq_PolarizationIdentity}
	& Z_2 Z_2^\top - \tilde{Z}_2 \tilde{Z}_2^\top = \frac{1}{2} \Big( \big(Z_2 + \tilde{Z}_2\big)\big(Z_2 - \tilde{Z}_2\big)^\top + \big(Z_2 - \tilde{Z}_2\big)\big(Z_2 + \tilde{Z}_2\big)^\top \Big)
\end{align}
and we can trivially follow
\begin{align*}
	& ||Z_2 Z_2^* - \tilde{Z}_2 \tilde{Z}_2^*|| \leq ||Z_2 - \tilde{Z}_2|| \, ||Z_2 + \tilde{Z}_2|| \ .
\end{align*}
As Theorem \ref{Thm_TraceMomentBound} can only deal with real Gaussian matrices, we further bound
\begin{align}\label{Eq_ApproxNei1}
	& ||Z_2 Z_2^* - \tilde{Z}_2 \tilde{Z}_2^*|| \leq \big(||\overbrace{\Re(Z_2 - \tilde{Z}_2)}^{=: \bm{Y}_{\Re,-}}|| + ||\overbrace{\Im(Z_2 - \tilde{Z}_2)}^{=: \bm{Y}_{\Im,-}}||\big) \nonumber\\
	& \hspace{4cm} \times \big(||\underbrace{\Re(Z_2 + \tilde{Z}_2)}_{=: \bm{Y}_{\Re,+}}|| + ||\underbrace{\Im(Z_2 + \tilde{Z}_2)}_{=: \bm{Y}_{\Im,+}}||\big) \ .
\end{align}
The columns of $Z_2-\tilde{Z}_2$ are
\begin{align}\label{Eq_ZColumns_Decomp}
	& (Z_2)_{\bullet,s}-(\tilde{Z}_2)_{\bullet,s} = \frac{\mathbbm{1}_{\rho_n(r,s-1)\leq m}}{\sqrt{2m+1}} \sum\limits_{l=K+1}^n \big(\bm{X}_{\bullet,l} - \tilde{\bm{X}}_{\bullet,l}\big) V_{l,s} \nonumber\\
	& = \frac{1}{\sqrt{n}} \frac{\mathbbm{1}_{\rho_n(r,s-1)\leq m}}{\sqrt{2m+1}} \sum\limits_{l=K+1}^n \sum\limits_{k=0}^\infty \Psi_k \underbrace{(\eta_{l-1-k} - \eta_{(l-1-k) \mod n})}_{=0 \text{ for }k < l} \, e^{-2\pi i \frac{(l-1)(s-1)}{n}} \nonumber\\
	& = \frac{1}{\sqrt{n}} \frac{\mathbbm{1}_{\rho_n(r,s-1)\leq m}}{\sqrt{2m+1}} \sum\limits_{k=K+1}^\infty \Psi_k \sum\limits_{l=K+1}^{n \land k} e^{-2\pi i \frac{(l-1)(s-1)}{n}} (\eta_{l-1-k} - \eta_{(l-1-k) \mod n}) \ .
\end{align}
Define the index set $J_r \coloneq \{s \in \{1,...,n\} \mid \rho_n(r,s-1) \leq m\}$, then for the covariance of columns $(Z_2)_{\bullet,s}-(\tilde{Z}_2)_{\bullet,s}$ we see
\begin{align*}
	& \E\big[ \big( (Z_2)_{\bullet,s_1}-(\tilde{Z}_2)_{\bullet,s_1}\big) \big( (Z_2)_{\bullet,s_2}-(\tilde{Z}_2)_{\bullet,s_2}\big)^* \big]\\
	& = \frac{\mathbbm{1}_{s_1,s_2 \in J_r}}{2m+1} \frac{1}{n} \sum\limits_{k_1,k_2=K+1}^\infty \Psi_{k_1} \Psi_{k_2}^\top\\
	& \hspace{2cm} \times \sum\limits_{l_{1/2} = K+1}^{n \land k_{1/2}} e^{-2\pi i \frac{(l_1-1)(s_1-1) - (l_2-1)(s_2-1)}{n}} (\mathbbm{1}_{l_1-k_1=l_2-k_2} + \mathbbm{1}_{l_1-k_1=l_2-k_2 \mod n})\\
	& = \frac{\mathbbm{1}_{s_1,s_2 \in J_r}}{2m+1} \frac{1}{n} \sum\limits_{k_1,k_2=K+1}^\infty \Psi_{k_1} \Psi_{k_2}^\top \sum\limits_{l_{1/2} = K+1}^{n \land k_{1/2}} e^{-2\pi i \frac{(l_1-1)(s_1-1) - (l_2-1)(s_2-1)}{n}} \mathbbm{1}_{l_2=l_1+k_2-k_1}\\
	& \hspace{0.5cm} + \frac{\mathbbm{1}_{s_1,s_2 \in J_r}}{2m+1} \frac{1}{n} \sum\limits_{k_1,k_2=K+1}^\infty \Psi_{k_1} \Psi_{k_2}^\top \sum\limits_{l_{1/2} = K+1}^{n \land k_{1/2}} e^{-2\pi i \frac{(l_1-1)(s_1-1) - (l_2-1)(s_2-1)}{n}} \mathbbm{1}_{l_2=l_1+k_2-k_1 \mod n}
\end{align*}
and elementary bounds yield
\begin{align}\label{Eq_ApproxNei2}
	& \big|\big|\E\big[ (\bm{Y}_{\Re/\Im,-})_{\bullet,s} (\bm{Y}_{\Re/\Im,-})_{\bullet,s'}^\top \big]\big|\big| \leq \frac{\mathbbm{1}_{s_1,s_2 \in J_r}}{2m+1} \, 2 \bigg(\sum\limits_{k=K+1}^\infty ||\Psi_k||\bigg)^2
\end{align}
and analogously one gets
\begin{align}\label{Eq_ApproxNei3}
	& \big|\big|\E\big[ (\bm{Y}_{\Re/\Im,+})_{\bullet,s} (\bm{Y}_{\Re/\Im,+})_{\bullet,s'}^\top \big]\big|\big| \leq \frac{\mathbbm{1}_{s_1,s_2 \in J_r}}{2m+1} \, 4 \bigg(\sum\limits_{k=0}^\infty ||\Psi_k||\bigg)^2 \ .
\end{align}
The matrices $\bm{Y}_{\Re/\Im,\pm}$ are technically $(d \times n)$ matrices, but for the sake of Theorem \ref{Thm_TraceMomentBound} are effectively $(d \times (2m+1))$, since only $2m+1$ columns differ from zero. We apply Theorem \ref{Thm_TraceMomentBound} to matrix $\bm{Y}_{\Re/\Im,-}$ with $M=2m+1$ and $\kappa_- = 2 \big(\sum\limits_{k=K+1}^\infty ||\Psi_k||\big)^2$
to get
\begin{align*}
	& \bP\Big( \big|\big| \bm{Y}_{\Re/\Im,-} \big|\big| \geq \varepsilon_- \Big) = \bP\Big( \big|\big| \bm{Y}_{\Re/\Im,-}\bm{Y}_{\Re/\Im,-}^\top \big|\big| \geq \varepsilon_-^2 \Big) \leq \frac{1}{\varepsilon_-^{2L}} \E\Big[ \big|\big| \bm{Y}_{\Re/\Im,-}\bm{Y}_{\Re/\Im,-}^\top \big|\big|^L \Big]\\
	& \leq \frac{1}{\varepsilon_-^{2L}} \E\Big[ \tr\Big( \big( \bm{Y}_{\Re/\Im,-}\bm{Y}_{\Re/\Im,-}^\top \big)^L \Big) \Big] \overset{\text{Thm. \ref{Thm_TraceMomentBound}}}{\leq} \frac{1}{\varepsilon_-^{2L}} \kappa_-^L (2L-1)!! \, (d+2m+1)^{L+1}\\
	& = 2^L (2L-1)!! \, \bigg(\frac{d+2m+1}{\varepsilon_-^2} \bigg(\sum\limits_{k=K+1}^\infty ||\Psi_k||\bigg)^2\bigg)^L (d+2m+1)\\
	& \overset{\text{(\ref{Eq_Assumption_Asymptotics})}}{\leq} 8^L \cK_1^{L} (2L-1)!! \, \bigg(\frac{n^{\alpha}}{\varepsilon_-^2} \bigg(\sum\limits_{k=K+1}^\infty ||\Psi_k||\bigg)^2\bigg)^L 4\cK_1 n^{\alpha} \ .
\end{align*}
Analogously, one may apply Theorem \ref{Thm_TraceMomentBound} to $\bm{Y}_{\Re/\Im,+}$ with $M=2m+1$ and $\kappa_+ = 4 \bigg(\sum\limits_{k=0}^\infty ||\Psi_k||\bigg)^2$ to get
\begin{align*}
	& \bP\Big( \big|\big| \bm{Y}_{\Re/\Im,-} \big|\big| \geq \varepsilon_+ \Big) \leq 4^L (2L-1)!! \, \bigg(\frac{d+2m+1}{\varepsilon_+^2} \bigg(\sum\limits_{k=0}^\infty ||\Psi_k||\bigg)^2\bigg)^L (d+2m+1)\\
	& \overset{\text{(\ref{Eq_Assumption_Asymptotics})}}{\leq} 16^L \cK_1^{L} (2L-1)!! \, \bigg(\frac{n^{\alpha}}{\varepsilon_+^2} \bigg(\sum\limits_{k=0}^\infty ||\Psi_k||\bigg)^2\bigg)^L 4\cK_1 n^{\alpha}
\end{align*}
As we are still free to choose $\varepsilon_-,\varepsilon_+>0$, we set
\begin{align*}
	& \varepsilon_-^2 = n^{\delta + \alpha} \Big(\sum\limits_{k=K+1}^\infty ||\Psi_k||\Big)^2 \ \ \text{ and } \ \ \varepsilon_+^2 = n^{\delta + \alpha} \Big(\sum\limits_{k=0}^\infty ||\Psi_k||\Big)^2 \ ,
\end{align*}
which yields
\begin{align*}
	& \bP\big( ||Z_2 Z_2^* - \tilde{Z}_2 \tilde{Z}_2^*|| \geq \varepsilon_{\Psi}(K) \, n^{\delta + \alpha} \big)\\
	& = \bP\big( ||Z_2 Z_2^* - \tilde{Z}_2 \tilde{Z}_2^*|| \geq 4\varepsilon_- \, \varepsilon_+ \big) \overset{\text{(\ref{Eq_ApproxNei1})}}{\leq} \bP\Big( \big( \bm{Y}_{\Re,-} + \bm{Y}_{\Im,-} \big) \big( \bm{Y}_{\Re,+} + \bm{Y}_{\Im,+} \big) \geq 2\varepsilon_- \, 2\varepsilon_+ \Big)\\
	& \leq \bP\big( \bm{Y}_{\Re,-} \geq \varepsilon_- \big) + \bP\big( \bm{Y}_{\Im,-} \geq \varepsilon_- \big) + \bP\big( \bm{Y}_{\Re,+} \geq \varepsilon_+ \big) + \bP\big( \bm{Y}_{\Im,+} \geq \varepsilon_+ \big)\\
	& \leq 16\cK_1 16^L \cK_1^{L} (2L-1)!! \, \big(n^{-\delta}\big)^L n^{\alpha} \ .
\end{align*}
Choosing $L = L(\alpha,\delta,D) \in \N$ large enough that $-\delta L + \alpha < n^{-D}$ forces the right hand side of the previous bound to be smaller than
\begin{align*}
	& \underbrace{16\cK_1 16^L \cK_1^{L} (2L-1)!!}_{=: C'} \, n^{-D} \ .
\end{align*}
By combining this with (\ref{Eq_dBL_bound1}) and (\ref{Eq_dBL_bound2}) we have shown
\begin{align*}
	& \bP\Big( d_{\operatorname{BL}}\big( \hat{\mu}_{S(\frac{2\pi r}{n})}, \hat{\mu}_{\tilde{S}(\frac{2\pi r}{n})} \big) > \frac{48K}{d} + \varepsilon_{\Psi}(K) \, n^{\delta + \alpha} \Big)\\
	& \leq \overbrace{\bP\Big( d_{\operatorname{BL}}\big( \hat{\mu}_{S( \frac{2\pi r}{n})}, \hat{\mu}_{Z_2 Z_2^*} \big) > \frac{24K}{d} \Big)}^{=0} + \overbrace{\bP\Big( d_{\operatorname{BL}}\big( \hat{\mu}_{\tilde{S}( \frac{2\pi r}{n})}, \hat{\mu}_{\tilde{Z}_2 \tilde{Z}_2^*} \big) > \frac{24K}{d} \Big)}^{=0}\\
	& \hspace{1cm} + \bP\big( ||Z_2 Z_2^* - \tilde{Z}_2 \tilde{Z}_2^*|| \geq \varepsilon_{\Psi}(K) \, n^{\delta + 2\alpha} \big)\\
	& \leq C' \, n^{-D} \ ,
\end{align*}
which proves the lemma. \qed

\subsection{Proof of Lemma \ref{Lemma_ApproxContinuity}}\label{Proof_Lemma_ApproxContinuity}
By definition, $\tilde{S}$ and $\tilde{S}'$ have the forms $\tilde{S}\big( \frac{2\pi r}{n} \big) = \tilde{Z} \tilde{Z}^*$ and $\tilde{S}'\big( \frac{2\pi r}{n} \big) = \tilde{Z}' (\tilde{Z}')^*$ for
\begin{align*}
	& \tilde{Z} = \frac{1}{\sqrt{2m+1}} \tilde{\bm{X}} V D_r^{\frac{1}{2}} \ \ \text{ and } \ \ \tilde{Z}' = \frac{1}{\sqrt{2m+1}} \tilde{\bm{X}}' V D_r^{\frac{1}{2}} \ .
\end{align*}
In analogy to (\ref{Eq_ApproxNei1}) we bound
\begin{align}\label{Eq_ApproxCont1}
	& ||\tilde{Z} \tilde{Z}^* - \tilde{Z}' (\tilde{Z}')^*|| \nonumber\\
	& \leq \big(||\underbrace{\Re(\tilde{Z} - \tilde{Z}')}_{=: \bm{Y}'_{\Re,-}}|| + ||\underbrace{\Im(\tilde{Z} - \tilde{Z}')}_{=: \bm{Y}'_{\Im,-}}||\big) \, \big(||\underbrace{\Re(\tilde{Z} + \tilde{Z}')}_{=: \bm{Y}'_{\Re,+}}|| + ||\underbrace{\Im(\tilde{Z} + \tilde{Z}')}_{=: \bm{Y}'_{\Im,+}}||\big) \ .
\end{align}
The calculation
\begin{align}\label{Eq_NeigboringProperty}
	& (\tilde{\bm{X}} V)_{\bullet,s} = \sum\limits_{l=1}^n \tilde{\bm{X}}_{\bullet,l} V_{l,s} = \frac{1}{\sqrt{n}} \sum\limits_{l=1}^n \bigg( \sum\limits_{k=0}^\infty \Psi_{k} \eta_{(l-1-k) \mod n} \bigg) e^{-2\pi i \frac{(l-1)(s-1)}{n}} \nonumber\\
	& = \frac{1}{\sqrt{n}} \sum\limits_{k=0}^\infty \Psi_{k} \sum\limits_{l=1}^n \eta_{(l-1-k) \mod n} e^{-2\pi i \frac{(l-1)(l-1)}{n}} = \frac{1}{\sqrt{n}} \sum\limits_{k=0}^\infty \Psi_{k} \sum\limits_{l=0}^{n-1} \eta_{l} e^{-2\pi i \frac{(l+k)(s-1)}{n}} \nonumber\\
	& = \frac{1}{\sqrt{n}} \sum\limits_{k=0}^\infty e^{-2\pi i \frac{k(s-1)}{n}} \Psi_{k} \sum\limits_{l=0}^{n-1} \eta_{l} e^{-2\pi i \frac{l(s-1)}{n}} = G\Big( \frac{2\pi(s-1)}{n} \Big) \big(\bm{\eta} V\big)_{\bullet,s}
\end{align}
shows that the columns of $\tilde{Z} - \tilde{Z}'$ are
\begin{align*}
	& \tilde{Z}_{\bullet,s} - \tilde{Z}'_{\bullet,s} = \frac{\mathbbm{1}_{\rho_n(r,s-1) \leq m}}{\sqrt{2m+1}} \Big( G\Big( \frac{2\pi(s-1)}{n} \Big) - G\Big( \frac{2\pi r}{n} \Big) \Big) \big(\bm{\eta} V\big)_{\bullet,s} \ .
\end{align*}
Define the index set $J_r \coloneq \{s \in \{1,...,n\} \mid \rho_n(r,s-1) \leq m\}$. In the proof of Lemma \ref{Lemma_AlmostWishart} we see
\begin{align*}
	& \E\big[(\bm{\eta} V)_{i_1,s_1} \ol{(\bm{\eta} V)_{i_2,s_2}}\big] \mathbbm{1}_{i_1=i_2, s_1=s_2} \ \ \text{ and } \ \ \E\big[(\bm{\eta} V)_{i_1,l_1} (\bm{\eta} V)_{i_2,l_2}\big] = \mathbbm{1}_{i_1=i_2, s_1-1=n-(s_2-1)} \ ,
\end{align*}
which may be used to calculate the covariance of columns $(\tilde{Z})_{\bullet,s}-(\tilde{Z}')_{\bullet,s}$ as
\begin{align*}
	& \E\big[ \big( (Z_2)_{\bullet,s_1}-(\tilde{Z}_2)_{\bullet,s_1}\big) \big( (Z_2)_{\bullet,s_2}-(\tilde{Z}_2)_{\bullet,s_2}\big)^* \big] \nonumber\\
	& = \frac{\mathbbm{1}_{s_1,s_2 \in J_r}}{2m+1} \Big( G\Big( \frac{2\pi(s_1-1)}{n} \Big) - G\Big( \frac{2\pi r}{n} \Big) \Big) \nonumber\\
	& \hspace{3cm} \times \underbrace{\E\big[ \big(\bm{\eta} V\big)_{\bullet,s_1} \big(\bm{\eta} V\big)_{\bullet,s_2}^* \big]}_{= \mathbbm{1}_{s_1=s_2} \, \operatorname{Id}_d} \Big( G\Big( \frac{2\pi(s_2-1)}{n} \Big) - G\Big( \frac{2\pi r}{n} \Big) \Big)^*
\end{align*}
and elementary bounds yield
\begin{align}\label{Eq_ApproxCont2}
	& \big|\big| \E\big[ (\bm{Y}'_{\Re/\Im,-})_{\bullet,s_1} (\bm{Y}'_{\Re/\Im,-})_{\bullet,s_2}^\top \big] \big|\big| \nonumber\\
	& \leq \frac{\mathbbm{1}_{s_1,s_2 \in J_r}}{2m+1} \Big|\Big| G\Big( \frac{2\pi(s_1-1)}{n} \Big) - G\Big( \frac{2\pi r}{n} \Big) \Big|\Big| \nonumber\\
	& \hspace{2cm} \times \mathbbm{1}_{s_1=s_2 \, \text{ or } \, s_1-1=n-(s_2-1)} \, \Big|\Big| G\Big( \frac{2\pi(s_2-1)}{n} \Big) - G\Big( \frac{2\pi r}{n} \Big) \Big|\Big| \ .
\end{align}
The observation (\ref{Eq_G_Continuity}) and the definition of $J_r$ together turn (\ref{Eq_ApproxCont2}) into
\begin{align*}
	& \big|\big| \E\big[ (\bm{Y}'_{\Re/\Im,-})_{\bullet,s_1} (\bm{Y}'_{\Re/\Im,-})_{\bullet,s_2}^\top \big] \big|\big|\\
	& \leq \frac{\mathbbm{1}_{s_1,s_2 \in J_r}}{2m+1} \Big( \frac{2\pi}{n} \Big)^2 m^2 \bigg( \sum\limits_{k=0}^\infty k ||\Psi_k|| \bigg)^2 \, \mathbbm{1}_{s_1=s_2 \, \text{ or } \, s_1-1=n-(s_2-1)}\\
	& \leq \mathbbm{1}_{s_1,s_2 \in J_r} \, \mathbbm{1}_{s_1=s_2 \, \text{ or } \, s_1-1=n-(s_2-1)} \, 2\pi^2 \frac{m}{n^2} \bigg( \sum\limits_{k=0}^\infty k ||\Psi_k|| \bigg)^2 \ .
\end{align*}
The matrices $\bm{Y}_{\Re/\Im,\pm}$ are technically $(d \times n)$ matrices, but for the sake of Theorem \ref{Thm_TraceMomentBound} are effectively $(d \times (2m+1))$, since only $2m+1$ columns differ from zero.
Application of Theorem \ref{Thm_TraceMomentBound} to the matrices $\bm{Y}'_{\Re/\Im,-}$ with $M=2m+1$ and $\kappa_-' = 4\pi^2 \frac{m}{n^2} \big( \sum\limits_{k=0}^\infty k ||\Psi_k|| \big)^2$ yields
\begin{align}\label{Eq_ApproxCont3}
	& \bP\Big( \big|\big| \bm{Y}'_{\Re/\Im,-} \big|\big| \geq \varepsilon_- \Big) = \bP\Big( \big|\big| \bm{Y}'_{\Re/\Im,-}(\bm{Y}'_{\Re/\Im,-})^\top \big|\big| \geq \varepsilon_-^2 \Big) \nonumber\\
	& \leq \frac{1}{\varepsilon_-^{2L}} \E\Big[ \big|\big| \bm{Y}'_{\Re/\Im,-}(\bm{Y}'_{\Re/\Im,-})^\top \big|\big|^L \Big] \leq \frac{1}{\varepsilon_-^{2L}} \E\Big[ \tr\Big( \big( \bm{Y}'_{\Re/\Im,-}(\bm{Y}'_{\Re/\Im,-})^\top \big)^L \Big) \Big] \nonumber\\
	& \hspace{-0.45cm} \overset{\text{Thm. \ref{Thm_TraceMomentBound}}}{\leq} \frac{1}{\varepsilon_-^{2L}} (\kappa_-')^L (2L-1)!! \, (d+2m+1)^{L+1} \nonumber\\
	& = (4\pi)^L (2L-1)!! \, \bigg(\frac{d+2m+1}{\varepsilon_-^2} \frac{m}{n^2} \bigg(2\sum\limits_{k=0}^\infty k ||\Psi_k||\bigg)^2\bigg)^L (d+2m+1) \nonumber\\
	& \hspace{-0.15cm} \overset{\text{(\ref{Eq_Assumption_Asymptotics})}}{\leq} (16\pi \cK_1^2)^L (2L-1)!! \, \bigg(\frac{n^{2\alpha}}{\varepsilon_-^2 \, n^2} \bigg(2\sum\limits_{k=0}^\infty k ||\Psi_k||\bigg)^2\bigg)^L 4\cK_1 n^{\alpha} \ .
\end{align}
The same steps for $\bm{Y}'_{\Re/\Im,+}$ are simpler, since $||G(\tau)|| \leq \sum\limits_{k=0}^\infty ||\Psi_k||$ analogously to (\ref{Eq_ApproxCont2}) gives
\begin{align}\label{Eq_ApproxCont4}
	& \big|\big| \E\big[ (\bm{Y}'_{\Re/\Im,+})_{\bullet,s_1} (\bm{Y}'_{\Re/\Im,+})_{\bullet,s_2}^\top \big] \big|\big| \nonumber\\
	& \leq \frac{\mathbbm{1}_{s_1,s_2 \in J_r}}{2m+1} \bigg( 2\sum\limits_{k=0}^\infty ||\Psi_k|| \bigg)^2 \, \mathbbm{1}_{s_1=s_2 \, \text{ or } \, s_1-1=n-(s_2-1)}
\end{align}
and we can apply Theorem \ref{Thm_TraceMomentBound} with $M=2m+1$ and $\kappa_+' = \frac{1}{m} \bigg( 2\sum\limits_{k=0}^\infty ||\Psi_k|| \bigg)^2$ to get
\begin{align}\label{Eq_ApproxCont5}
	& \bP\Big( \big|\big| \bm{Y}'_{\Re/\Im,+} \big|\big| \geq \varepsilon_+ \Big) \leq (2L-1)!! \, \bigg(\frac{d+2m+1}{\varepsilon_+^2} \frac{1}{m} \bigg(2\sum\limits_{k=0}^\infty ||\Psi_k||\bigg)^2\bigg)^L (d+2m+1) \nonumber\\
	& \overset{\text{(\ref{Eq_Assumption_Asymptotics})}}{\leq} (4\cK_1^2)^L (2L-1)!! \, \bigg(\frac{1}{\varepsilon_+^2} \bigg(2\sum\limits_{k=0}^\infty ||\Psi_k||\bigg)^2\bigg)^L 4\cK_1 n^{\alpha} \ .
\end{align}
Choosing
\begin{align*}
	& \varepsilon_-^2 = \frac{n^{\delta + 2\alpha}}{n^2} \bigg(2\sum\limits_{k=0}^\infty k ||\Psi_k||\bigg)^2 \ \ \text{ and } \ \ \varepsilon_+^2 = n^{\delta} \bigg(2\sum\limits_{k=0}^\infty ||\Psi_k||\bigg)^2
\end{align*}
we see
\begin{align*}
	& \bP\Big( ||\tilde{Z} \tilde{Z}^* - \tilde{Z}' (\tilde{Z}')^*|| \geq \varepsilon_{\Psi} \, n^{\delta + \alpha-1} \Big)\\
	& = \bP\Big( ||\tilde{Z} \tilde{Z}^* - \tilde{Z}' (\tilde{Z}')^*|| \geq 4\varepsilon_- \varepsilon_+ \Big) \overset{\text{(\ref{Eq_ApproxCont1})}}{\leq} \bP\Big( \big( \bm{Y}'_{\Re,-} + \bm{Y}'_{\Im,-} \big) \big( \bm{Y}'_{\Re,+} + \bm{Y}'_{\Im,+} \big) \geq 2\varepsilon_- \, 2\varepsilon_+ \Big)\\
	& \leq \bP\big( \bm{Y}'_{\Re,-} \geq \varepsilon_- \big) + \bP\big( \bm{Y}'_{\Im,-} \geq \varepsilon_- \big) + \bP\big( \bm{Y}'_{\Re,+} \geq \varepsilon_+ \big) + \bP\big( \bm{Y}'_{\Im,+} \geq \varepsilon_+ \big)\\
	& \leq 16\cK_1(4\cK_1^2)^L (2L-1)!! \, \big(n^{-\delta}\big)^L n^{\alpha} \ .
\end{align*}
For $L=L(\alpha,\delta,D) \in \N$ large enough that $-\delta L + \alpha < n^{-D}$ the right hand side of the previous bound is smaller than
\begin{align*}
	& \underbrace{16\cK_1(4\cK_1^2)^L (2L-1)!!}_{=: C''} \, n^{-D} \ ,
\end{align*}
which proves the lemma. \qed

\subsection{Proof of Lemma \ref{Lemma_LocalLaw}}\label{Proof_Lemma_LocalLaw}
For now assume that there exists a constant $\varepsilon > 0$ such that the singular value bound 
\begin{align}\label{Eq_TempAssumption_LowerBound}
	& \varepsilon \leq \min\limits_{\theta \in [0,2\pi)} \sigma_{\min}(G(\theta))
\end{align}
holds uniformly in $n \in \N$. After this proof, said assumption may be removed by following Section 11 of \cite{Knowles} verbatim.
\begin{itemize}
	\item[a)]
	Let $U_1\Sigma U_2$ be the singular value decomposition of $G(\frac{2\pi r}{n})$. We may without loss of generality replace the definition of $\hat{\nu}_n(\frac{2\pi r}{n})$ with
	\begin{align*}
		& \hat{\nu}_n\Big(\frac{2\pi r}{n}\Big) \coloneq \hat{\mu}_{\Sigma \bm{Z} \bm{Z}^* \Sigma^\top} \ ,
	\end{align*}
	where $\bm{Z}$ is a $(d \times 2m+1)$ random matrix with i.i.d. complex standard normal entries. This is due to the fact that $U_1$ does not affect the ESD of $U_1 \Sigma \bm{Z} \bm{Z}^* \Sigma^\top U_1^*$ and the fact that $\bm{Z} \bm{Z}^*$ has the same distribution as $U_2 \bm{W} U_2^*$. The entries of the matrix
	\begin{align*}
		& \mathbb{X} \coloneq \frac{1}{\sqrt{2m+1}} \Sigma \bm{Z}
	\end{align*}
	are independent, centered and have covariance structure $\mathbb{S}_{i,k} \coloneq \E[|\mathbb{X}_{i,k}|^2] = \frac{\Sigma_{i,i}^2}{2m+1}$. We briefly check assumptions (A)-(D) of \cite{GramErdos} for $\mathbb{X}$.
	\begin{itemize}
		\item[A)]
		For any $s_* \geq 2\cK_2^2\max(\cK_1^2, 1)$ the assumption $\bS_{i,k} \leq \frac{s_*}{d+n}$ holds, since
		\begin{align*}
			& \mathbb{S}_{i,k} = \frac{\Sigma_{i,i}^2}{2m+1} \leq \frac{||\Sigma^2||}{2m+1} =  \frac{||G( \frac{2\pi r}{n} )||^2}{2m+1} \overset{\text{(\ref{Eq_Def_G})}}{\leq} \frac{\big(\sum\limits_{k=0}^\infty ||\Psi_k||\big)^2}{2m+1}\\
			& \overset{\text{(\ref{Eq_Assumption_LongRandeDependence_cK2})}}{\leq} \frac{\cK_2^2}{\frac{1}{2}(3m+1)} \overset{\text{(\ref{Eq_Assumption_Asymptotics})}}{\leq} \frac{2\cK_2^2}{\frac{d}{\cK_1^2} + 2m+1} \leq \frac{2\cK_2^2\max(\cK_1^2, 1)}{d + 2m+1} \ .
		\end{align*}
		
		\item[B)]
		By construction of $\Sigma$ and temporary assumption (\ref{Eq_TempAssumption_LowerBound}), it holds that $\Sigma_{i,i} \geq \sqrt{\varepsilon}$ for all $i \leq d$. We calculate
		\begin{align*}
			& (\bS \bS^\top)_{i,j} = \sum\limits_{k=1}^{2m+1} \frac{\Sigma_{i,i}^2}{2m+1} \frac{\Sigma_{j,j}^2}{2m+1} \overset{\text{(\ref{Eq_TempAssumption_LowerBound})}}{\geq} \frac{\varepsilon^2}{2m+1} \geq \frac{\varepsilon^2}{d+2m+1}
		\end{align*}
		and
		\begin{align*}
			& (\bS^\top \bS)_{k,l} = \sum\limits_{i=1}^d \frac{\Sigma_{i,i}^4}{(2m+1)^2} \overset{\text{(\ref{Eq_TempAssumption_LowerBound})}}{\geq} \frac{\varepsilon^2 \, d}{(2m+1)^2}\\
			& \geq \frac{\varepsilon^2 \, d/4m}{2m+1} \overset{\text{(\ref{Eq_Assumption_Asymptotics})}}{\geq} \frac{\varepsilon^2/4\cK_1^2}{d + 2m+1} \geq \frac{\frac{\varepsilon^2}{4\cK_1^2}}{d + 2m+1} \ ,
		\end{align*}
		which shows that assumption (B) from \cite{GramErdos} holds with $L_1=1=L_2$ and $\psi_1 = \varepsilon^2$ as well as $\psi_2 = \frac{\varepsilon^2}{4\cK_1^2}$.
		
		\item[C)]
		Since complex standard normal random variables $Z$ satisfy $\E[|Z|^m] \leq \sqrt{m!}$, the trivial calculation
		\begin{align*}
			& \E\big[ |\mathbb{X}_{i,k}|^m \big] = \frac{\Sigma_{i,i}^m}{\sqrt{2m+1}} \E\big[ |\bm{Z}_{i,k}|^m \big] = \bS_{i,k}^{\frac{m}{2}} \E\big[ |\bm{Z}_{i,k}|^m \big]
		\end{align*}
		shows assumption (D) from \cite{GramErdos} with $\mu_m=\sqrt{m!}$.
		
		\item[D)]
		Assumption (\ref{Eq_Assumption_Asymptotics}) directly yields
		\begin{align*}
			& \frac{1}{3\cK_1^2} \leq \frac{d}{2m+1} \leq \frac{\cK_1^2}{2} \ ,
		\end{align*}
		so assumption (D) from \cite{GramErdos} also holds.
	\end{itemize}
	By Theorem 2.2 of \cite{GramErdos} (see (2.6b) with $w=(1,...,1)^\top$), there for any $\tau,\delta,D>0$ exists a constant $C_{\delta,D}>0$, which in addition to $\delta,D$ only depends on $\cK_1,\cK_2,\tau$ and $\varepsilon$, such that
	\begin{align}\label{Eq_GramErdösResult1}
		& \bP\Big( \exists z \in \C^+ , \, \tau \leq |z| \leq \tau^{-1} , \, \dist(z,\operatorname{supp}(\nu_n) \geq \tau) : \nonumber\\
		& \hspace{3cm} \big| \cs_{\hat{\nu}_n(\frac{2\pi r}{n})}(z) - \cs_{\nu_n(\frac{2\pi r}{n})}(z) \big| \geq \frac{d^\delta}{d} \Big) \leq \frac{C_{\delta,D}}{d^D}
	\end{align}
	holds for all $n \in \N$. Note that all $z \in \bm{S}(\tau)$ satisfy the prerequisites from the above bound and that $C_{\delta,D}$ does not depend on $r \in \{0,...,n-1\}$, which allows us to follow
	\begin{align*}
		& \bP\Big( \exists r < n \, \exists z \in \bm{S}(\tau) : \  \big| \cs_{\hat{\nu}_n(\frac{2\pi r}{n})}(z) - \cs_{\nu_n(\frac{2\pi r}{n})}(z) \big| \geq \frac{d^\delta}{d} \Big) \leq \frac{C_{\delta,D}}{d^D}n \ .
	\end{align*}
	Lastly, with the simple bound
	\begin{align*}
		& \frac{n}{d^D} \overset{\text{(\ref{Eq_Assumption_Asymptotics})}}{\leq} \cK_1^D n^{1-\alpha D}
	\end{align*}
	we by choosing $C(\tau,\delta,\tilde{D}) \geq \cK_1^{\frac{\tilde{D}+1}{\alpha}} C_{\delta,\frac{\tilde{D}+1}{\alpha}}$ get
	\begin{align*}
		& \bP\Big( \exists r < n \, \exists z \in \bm{S}(\tau) : \  \big| \cs_{\hat{\nu}_n(\frac{2\pi r}{n})}(z) - \cs_{\nu_n(\frac{2\pi r}{n})}(z) \big| \geq \underbrace{d^{\delta-1}}_{= a_n} \Big) \leq \frac{C(\tau,\delta,D)}{n^D} \ .
	\end{align*}
	
	\item[b)]
	For $\theta \in \{0,\pi\}$ let $U_1\Sigma U_2$ be the singular value decomposition of $G(\theta)$. We may without loss of generality replace the definition of $\hat{\nu}_n(\frac{2\pi r}{n})$ with
	\begin{align*}
		& \hat{\nu}_n(\theta) \coloneq \hat{\mu}_{\Sigma \bm{Z} \bm{Z}^\top \Sigma^\top} \ ,
	\end{align*}
	where $\bm{Z}$ is a $(d \times 2m)$ random matrix with i.i.d. real standard normal entries. This is due to the fact that $U$ does not affect the ESD of $U_1 \Sigma \bm{Z} \bm{Z}^\top \Sigma^\top U_1^\top$ and the fact that $\bm{Z} \bm{Z}^\top$ has the same distribution as $U_2 \bm{W} U_2^\top$. We have used the fact that $G(\theta)$ has only real entries for $\theta \in \{0,\pi\}$ and thus, $U_1,U_2$ must be real orthogonal matrices instead of only unitary.\\
	\\
	We can now follow the proof of (a) up to (\ref{Eq_GramErdösResult1}) verbatim. The new choice of $\mu_m$ is $\sqrt{(2m-1)!!}$ and we analogously get a constant $C'_{\delta,D} > 0$ also depending on $\cK_1,\cK_2,\tau$ and $\varepsilon$ such that
	\begin{align*}
		& \bP\Big( \exists z \in \bm{S}(\tau) : \ \big| \cs_{\hat{\nu}_n(\theta)}(z) - \cs_{\nu_n(\theta)}(z) \big| \geq \frac{d^\delta}{d} \Big) \leq \frac{C'_{\delta,D}}{d^D}
	\end{align*}
	holds for both $\theta \in \{0,\pi\}$. By choosing $C(\tau,\delta,\tilde{D}) \geq \cK_1^{\frac{\tilde{D}}{\alpha}} C'_{\delta,\frac{\tilde{D}}{\alpha}}$, it analogously follows that
	\begin{align*}
		& \bP\Big( \exists z \in \bm{S}(\tau) : \ \big| \cs_{\hat{\nu}_n(\theta)}(z) - \cs_{\nu_n(\theta)}(z) \big| \geq \underbrace{d^{\delta-1}}_{= a_n} \Big) \leq \frac{C(\tau,\delta,\tilde{D})}{n^D} \ .
	\end{align*}
\end{itemize}
It remains to lose the temporary assumption \ref{Eq_TempAssumption_LowerBound}. For each $\varepsilon > 0$, define $G_\varepsilon(\theta) \coloneq U(\Sigma + \varepsilon\operatorname{Id}_d)V$, where $U\Sigma V$ is the singular value decomposition of $G(\theta)$. Let $C_\varepsilon(\tau,\delta,D)$ be the constant for which (\ref{Eq_OuterLaw_complex}) and (\ref{Eq_OuterLaw_real}) hold with $G_\varepsilon(\theta)$ instead of $G(\theta)$ and $G_\varepsilon(\theta)G_\varepsilon(\theta)^*$ instead of $F(\theta)$, i.e. we have shown
\begin{align}\label{Eq_OuterLaw_complex_epsilon}
	& \bP\Big( \exists r < n \, \exists z \in \bm{S}(\tau) : \  \big| \cs_{\hat{\nu}_{n,\varepsilon}(\frac{2\pi r}{n})}(z) - \cs_{\nu_{n,\varepsilon}(\frac{2\pi r}{n})}(z) \big| \geq \frac{d^\delta}{d} \Big) \leq \frac{C_\varepsilon(\tau,\delta,D)}{n^D}
\end{align}
and
\begin{align}\label{Eq_OuterLaw_real_epsilon}
	& \forall \theta \in \{0,\pi\} : \ \bP\Big( \exists z \in \bm{S}(\tau) : \ \big| \cs_{\hat{\nu}_{n,\varepsilon}(\theta)}(z) - \cs_{\nu_{n,\varepsilon}(\theta)}(z) \big| \geq \frac{d^\delta}{d} \Big) \leq \frac{C_\varepsilon(\tau,\delta,\tilde{D})}{n^D} \ ,
\end{align}
where $\hat{\nu}_{n,\varepsilon}(\theta) = \hat{\mu}(\frac{1}{2m+1}G_\varepsilon(\theta)\bm{W}G_\varepsilon(\theta)^*)$ and $\nu_{n,\varepsilon}(\theta)$ denotes the probability measure defined by Lemma \ref{Lemma_MPEquation} from $H_{n,\varepsilon}(\theta) = \hat{\mu}_{G_\varepsilon(\theta)G_\varepsilon(\theta)^*}$ and $c_n$. We can bound the operator norm perturbation
\begin{align*}
	& \Big|\Big| \frac{1}{2m+1}G(\theta)\bm{W}G(\theta)^* - \frac{1}{2m+1}G_\varepsilon(\theta)\bm{W}G_\varepsilon(\theta)^* \Big|\Big|\\
	& = \frac{1}{2m+1} \big|\big| \Sigma V\bm{W}V^* \Sigma - (\Sigma+\varepsilon\operatorname{Id}_d) V\bm{W}V^* (\Sigma+\varepsilon\operatorname{Id}_d) \big|\big|\\
	& \leq \frac{2\varepsilon||\Sigma||}{2m+1} \big|\big| V\bm{W}V^* \big|\big| + \frac{\varepsilon^2}{2m+1} \big|\big| V\bm{W}V^* \big|\big| \overset{\text{(\ref{Eq_Assumption_LongRandeDependence_cK2})}}{\leq} \varepsilon\frac{2\cK_2 + \varepsilon}{2m+1} \big|\big| \bm{W} \big|\big| \ .
\end{align*}
The operator norms of isotropic (complex or real) Wishart matrices $\bm{W}$ are very well understood and for example Theorem 2 of \cite{LedouxRider} may be used to show the existence of a constant
$C''(D)>0$ only dependent on $D$ and $\cK_1$ such that
\begin{align}\label{Eq_WishartNormBound_old}
	& \bP\Big( \frac{\big|\big| \bm{W} \big|\big|}{2m+1} \geq C''(D) \Big) \leq \frac{C''(D)}{n^{D}}
\end{align}
holds for all $n \in \N$.
By the simple bound
\begin{align*}
	& \big| \cs_{\hat{\mu}(A)}(z) - \cs_{\hat{\mu}(B)}(z) \big| \leq \frac{1}{d} \sum\limits_{j=1}^d \Big| \frac{1}{\lambda_j(A)-z} - \frac{1}{\lambda_j(B)-z} \Big|\\
	& = \frac{1}{d} \sum\limits_{j=1}^d \frac{|\lambda_j(A) - \lambda_j(B)|}{|\lambda_j(A)-z| \, |\lambda_j(B)-z|} \leq \frac{||A-B||}{\Im(z)^2}
\end{align*}
we follow
\begin{align}\label{Eq_TempAssumption_Resolution1}
	& \bP\Big( \forall z \in \bm{S}(\tau) : \ \big| \cs_{\hat{\mu}(\frac{1}{2m+1}G(\theta)\bm{W}G(\theta)^*)}(z) - \cs_{\hat{\mu}(\frac{1}{2m+1}G_\varepsilon(\theta)\bm{W}G_\varepsilon(\theta)^*)}(z) \big|\\
	& \hspace{6cm} \geq \frac{\varepsilon}{\tau^2} (2\cK_2+\varepsilon) C''(D) \Big) \leq \frac{C''(D)}{n^{D}}
\end{align}
for all $n \in \N$. Since $\nu_{n(,\varepsilon)}(\theta)$ is also the limiting spectral distribution of a meta-model with dimension quotient $\frac{\tilde{d}}{\tilde{n}}$ converging to $c_n$ and population spectral distribution converging to $H_{n(,\varepsilon)}(\theta)$, one can from the above bound also follow
\begin{align}\label{Eq_TempAssumption_Resolution2}
	& \big| \cs_{\nu_n(\theta)}(z) - \cs_{\nu_{n,\varepsilon}(\theta)}(z) \big| \leq \frac{\varepsilon}{\tau^2} (2\cK_2+\varepsilon) C''(D)
\end{align}
for all $z \in \bm{S}(\tau)$, $\theta \in [0,2\pi)$ and $n \in \N$. Applying (\ref{Eq_TempAssumption_Resolution1}) and (\ref{Eq_TempAssumption_Resolution2}) to (\ref{Eq_OuterLaw_complex_epsilon}) yields
\begin{align*}
	& \bP\Big( \exists r < n \, \exists z \in \bm{S}(\tau) : \  \big| \cs_{\hat{\nu}_{n}(\frac{2\pi r}{n})}(z) - \cs_{\nu_{n}(\frac{2\pi r}{n})}(z) \big|\\
	& \hspace{2cm} \geq \frac{d^\delta}{d} + \frac{2\varepsilon}{\tau^2} (2\cK_2+\varepsilon) C''(D) \Big) \leq \frac{C_\varepsilon(\tau,\delta,D)+C''(D)}{n^D} \ .
\end{align*}
Let $(\varepsilon_n)_{n \in \N} \subset (0,1)$ be a sequence converging slowly enough that $C_{\varepsilon_n}(\tau,\delta,D) \leq n C_{\varepsilon_1}(\tau,\delta,D)$, then we have shown (\ref{Eq_OuterLaw_complex}) for $a_n = \frac{d^\delta}{d} + \frac{2\varepsilon_n}{\tau^2} (2\cK_2+\varepsilon_n) C''(D)$ and $C(\tau,\delta,D) = C_{\varepsilon_1}(\tau,\delta,D+1) + C''(D+1)$. The proof of (\ref{Eq_OuterLaw_real}) is analogous. \qed

\subsection{Proof of Lemma \ref{Lemma_Lindeberg}}\label{Proof_Lemma_Lindeberg}
By simple telescope sum one sees
\begin{align*}
	& \E\big[ f(X_1,...,X_n) - f(Y_1,...,Y_n) \big]\\
	& = \sum\limits_{j=1}^n \E\big[ f(\underbrace{X_1,...,X_j,Y_{j+1},...,Y_n}_{=: \bm{Z}_{j}}) - f(\underbrace{X_1,...,X_{j-1},Y_{j},...,Y_n}_{=\bm{Z}_{j-1}}) \big] \ .
\end{align*}
For
\begin{align*}
	& \bm{Z}^0_j \coloneq (X_1,...,X_{j-1},0,Y_{j+1},...,Y_n)
\end{align*}
a third order Taylor series with integral remainder gives
\begin{align*}
	& f(\bm{Z}_j) = f(\bm{Z}^0_j) + X_j \partial_j f(\bm{Z}^0_j) + \frac{X_j^2}{2} \partial_j^2 f(\bm{Z}^0_j) + \overbrace{\frac{1}{2}\int_0^{X_j} \tau^2 \partial_j^3 f(\bm{Z}^0_j + \tau \mathrm{u}_j^{(N)}) \, d\tau}^{= \frac{X_j^3}{2}\int_0^{1} \tau^2 \partial_j^3 f(\bm{Z}^0_j + \tau X_j \mathrm{u}_j^{(N)}) \, d\tau =: R^X_j}\\
	& f(\bm{Z}_{j-1}) =  f(\bm{Z}^0_j) + Y_j \partial_j f(\bm{Z}^0_j) + \frac{Y_j^2}{2} \partial_j^2 f(\bm{Z}^0_j) + \underbrace{\frac{1}{2}\int_0^{Y_j} \tau^2 \partial_j^3 f(\bm{Z}^0_j + \tau \mathrm{u}_j^{(N)}) \, d\tau}_{= \frac{Y_j^3}{2}\int_0^{1} \tau^2 \partial_j^3 f(\bm{Z}^0_j + \tau Y_j \mathrm{u}_j^{(N)}) \, d\tau =: R^Y_{j}} \ .
\end{align*}
The independence of $X_1,...,X_n,Y_1,...,Y_n$ may then be applied for
\begin{align*}
	& \E\big[ f(\bm{Z}_j) - f(\bm{Z}_{j-1}) \big]\\
	& = \E\Big[ (X_j-Y_j) \partial_j f(\bm{Z}^0_j) \Big] + \frac{1}{2}\E\big[ (X_j^2-Y_j^2) \partial_j^2 f(\bm{Z}^0_j) \big] + \E\big[ R^X_j-R^Y_j \big]\\
	& = \underbrace{\E[X_j-Y_j]}_{=0} \E\Big[ \partial_j f(\bm{Z}^0_j) \Big] + \frac{1}{2} \underbrace{\E[X_j^2-Y_j^2]}_{=0} \E\big[ \partial_j^2 f(\bm{Z}^0_j) \big] + \E\big[ R^X_j-R^Y_j \big]
\end{align*}
%and we by Hölder's inequality with $(p,q) = (\frac{4}{3},4)$ get
%\begin{align*}
%& \Big| \E\big[ f(\bm{Z}_i) - f(\bm{Z}_{i-1}) \big] \Big| \leq \frac{1}{6} \Big|\E\big[ \partial_i^3 f(\tau_1 X_i) X_i^3\big]\Big| + \frac{1}{6} \Big|\E\big[ \partial_i^3f(\tau_2 Y_i) Y_i^3 \big]\Big|\\
%& \leq \frac{1}{6} \E\big[ |\partial_i^3 f(\tau_1 X_i)|^4\big]^{\frac{1}{4}} \E[X_i^4]^{\frac{3}{4}} + \frac{1}{6} \E\big[ |\partial_i^3f(\tau_2 Y_i)|^4 \big]^{\frac{1}{4}} \E[Y_i^4]^{\frac{3}{4}} \ . \qedhere
%\end{align*}
and Fubini and Cauchy-Schwarz yield
\begin{align*}
	& \Big| \E\big[ f(\bm{Z}_j) - f(\bm{Z}_{j-1}) \big] \Big| \leq \big| \E[R^X_j] \big| + \big| \E[R^Y_j] \big|\\
	& = \bigg| \frac{1}{2}\int_0^{1} \tau^2 \E\big[ X_j^3 \partial_j^3 f(\bm{Z}^0_j + \tau X_j \mathrm{u}_j^{(N)}) \big] \, d\tau \bigg| + \bigg| \frac{1}{2}\int_0^{1} \tau^2 \E\big[ Y_j^3 \partial_j^3 f(\bm{Z}^0_j + \tau Y_j \mathrm{u}_j^{(N)}) \big] \, d\tau \bigg|\\
	& \leq \frac{1}{2} \max\limits_{\tau \in [0,1]} \E[X_j^6]^{\frac{1}{2}} \E\big[ |\partial_j^3 f(\bm{Z}^0_j + \tau X_j \mathrm{u}_j^{(N)})|^2 \big]^\frac{1}{2} + \frac{1}{2} \max\limits_{\tau \in [0,1]} \E[X_j^6]^{\frac{1}{2}} \E\big[ |\partial_j^3 f(\bm{Z}^0_j + \tau Y_j \mathrm{u}_j^{(N)})|^2 \big]^\frac{1}{2}\\
	& \leq \frac{\sqrt{\cK}}{2} \big(M^X_j + M^Y_j\big) \ . \qed
\end{align*}

\subsection{Proof of Lemma \ref{Lemma_BoundingStieltjesDerivatives}}\label{Proof_Lemma_BoundingStieltjesDerivatives}
With the notation $A(x) = B(x) B^*(x)$ it is clear that
\begin{align*}
	& \frac{\partial }{\partial x_r} A(x) = \frac{\partial B(x)}{\partial x_r} B^*(x) + B(x) \Big( \frac{\partial B(x)}{\partial x_r} \Big)^*
\end{align*}
and
\begin{align*}
	& \frac{\partial^2}{\partial x_{r_1} \, \partial x_{r_2}} A(x) = \frac{\partial B(x)}{\partial x_{r_1}} \Big( \frac{\partial B(x)}{\partial x_{r_2}} \Big)^* + \frac{\partial B(x)}{\partial x_{r_2}} \Big( \frac{\partial B(x)}{\partial x_{r_1}} \Big)^* \ .
\end{align*}
The bounds
\begin{align}
	& \Big|\Big| \frac{\partial}{\partial x_r} A(x) \Big|\Big| \leq 2 \Big|\Big| \frac{\partial B(x)}{\partial x_r} \Big|\Big| \, \big|\big| B(x) \big|\big| \leq 2||B(x)|| \kappa \label{Eq_SpectralDerivativeBound1}\\
	& \Big|\Big| \frac{\partial^2}{\partial x_{r_1} \, dx_{r_2}} A(x) \Big|\Big| \leq 2 \Big|\Big| \frac{\partial B(x)}{\partial x_{r_1}} \Big|\Big| \, \Big|\Big| \frac{\partial B(x)}{\partial x_{r_2}} \Big|\Big| \leq 2\kappa^2 \label{Eq_SpectralDerivativeBound2}
\end{align}
immediately follow and the fact that $A(x)$ is Hermitian gives
\begin{align}\label{Eq_SpectralResolventBound}
	& ||(A(x)-z\operatorname{Id})^{-1}|| \leq \frac{1}{\Im(z)}
\end{align}
for all $z \in \C$ with $\Im(z) > 0$.
\\
\\
The following Lemma may be applied to $\cs_{B(x)B^*(x)}(z) = \frac{1}{d} \tr\big( (A(x) - z \operatorname{Id})^{-1} \big)$.

\begin{lemma}[Resolvent derivatives]\thlabel{Lemma_ResolventDerivatives}\
	\\
	For any $(d \times d)$ matrix $A$ and $z \in \C$ such that $(A-z\operatorname{Id})^{-1}$ exists we have
	\begin{align}\label{Eq_A_derivative0}
		& \frac{d}{dA_{i,j}} (A-z\operatorname{Id})^{-1} = - (A-z\operatorname{Id})^{-1} \mathrm{u}_i^{(d)} (\mathrm{u}_j^{(d)})^\top (A-z\operatorname{Id})^{-1} \ .
	\end{align}
	%\begin{align*}
	%& \frac{d^L}{dA_{i_1,j_1} \cdots d_{A_{i_L,j_L}}} (A-z\operatorname{Id})^{-1}\\
	%& = (-1)^L \sum\limits_{\sigma \in S_L} (A-z\operatorname{Id})^{-1} e_{i_{\sigma(1)}} e_{j_{\sigma(1)}}^\top (A-z\operatorname{Id})^{-1} \cdots e_{i_{\sigma(L)}} e_{j_{\sigma(L)}}^\top (A-z\operatorname{Id})^{-1} \ ,
	%\end{align*} 
	Further, for a smooth map $A : \R^n \rightarrow \C^{d \times d}$, we can by chain rule calculate the partial derivatives of the first three orders to be
	\begin{align}\label{Eq_A_derivative1}
		& \frac{\partial}{\partial x_{k_1}} (A(x)-z\operatorname{Id})^{-1} = - (A-z\operatorname{Id})^{-1} \frac{\partial A(x)}{\partial x_{k_1}} (A-z\operatorname{Id})^{-1} \ ,
	\end{align}
	\begin{align}\label{Eq_A_derivative2}
		& \frac{\partial^2}{\partial x_{k_1} \, \partial x_{k_2}} (A(x)-z\operatorname{Id})^{-1} = \sum\limits_{\sigma \in S_2} (A-z\operatorname{Id})^{-1} \frac{\partial A(x)}{\partial x_{k_{\sigma(1)}}} (A-z\operatorname{Id})^{-1} \frac{\partial A(x)}{\partial x_{k_{\sigma(2)}}} (A-z\operatorname{Id})^{-1} \nonumber\\
		& \hspace{4cm} - (A-z\operatorname{Id})^{-1} \frac{\partial^2A(x)}{\partial x_{k_1} \, \partial x_{k_2}} (A-z\operatorname{Id})^{-1}
	\end{align}
	and
	\begin{align}\label{Eq_A_derivative3}
		& \frac{\partial^3}{\partial x_{k_1} \, \partial x_{k_2} \, \partial x_{k_3}} (A(x)-z\operatorname{Id})^{-1} \nonumber\\
		& = - \sum\limits_{\sigma \in S_3} (A-z\operatorname{Id})^{-1} \frac{\partial A(x)}{\partial x_{k_{\sigma(1)}}} (A-z\operatorname{Id})^{-1} \frac{\partial A(x)}{\partial x_{k_{\sigma(2)}}} (A-z\operatorname{Id})^{-1} \frac{\partial A(x)}{\partial x_{k_{\sigma(3)}}} (A-z\operatorname{Id})^{-1} \nonumber\\
		& \hspace{1cm} + \sum\limits_{\substack{\alpha \sqcup \beta = \{1,2,3\} \\ \alpha,\beta \neq \emptyset}} (A-z\operatorname{Id})^{-1} \frac{\partial^{\#\alpha}A(x)}{\partial x_{k_\alpha}} (A-z\operatorname{Id})^{-1} \frac{\partial^{\#\beta}A(x)}{\partial x_{k_\beta}} (A-z\operatorname{Id})^{-1} \nonumber\\
		& \hspace{1cm} - (A-z\operatorname{Id})^{-1} \frac{\partial ^3A(x)}{\partial x_{k_1} \, \partial x_{k_2} \, \partial x_{k_3}} (A-z\operatorname{Id})^{-1} \ ,
	\end{align}
	where $S_L \coloneq \{\sigma : \{1,...,L\} \hookrightarrow \{1,...,L\}\}$ is the symmetric group on $L$ letters.
\end{lemma}
\
\\
By (\ref{Eq_A_derivative3}) we have
%\begin{align*}
%& \frac{d}{dx_r} \cs_{B(x)B^*(x)}(z) = \frac{1}{d} \tr\Big( (A(x)-z\operatorname{Id})^{-1} \frac{d A(x)}{dx_r} (A(x)-z\operatorname{Id})^{-1}\Big) \ ,
%\end{align*}
%\begin{align*}
%& \frac{d^2}{dx_{r_1} \, dx_{r_2}} \cs_{B(x)B^*(x)}(z)\\
%& = \frac{1}{d}\sum\limits_{\sigma \in S_2} \tr\Big( (A(x)-z\operatorname{Id})^{-1} \frac{d A(x)}{dx_{r_{\sigma(1)}}} (A(x)-z\operatorname{Id})^{-1}  \frac{d A(x)}{dx_{r_{\sigma(2)}}} (A(x)-z\operatorname{Id})^{-1} \Big) \\
%& \hspace{1cm} - \frac{1}{d} \tr\Big( (A-z\operatorname{Id})^{-1} \frac{d^2A(x)}{dx_{r_1} \, dx_{r_2}} (A-z\operatorname{Id})^{-1} \Big)
%\end{align*}
%and
\begin{align*}
	& \frac{\partial^3}{\partial x_{r_1} \, \partial x_{r_2} \, \partial x_{r_3}} \cs_{B(x)B^*(x)}(z)\\
	& = - \frac{1}{d} \sum\limits_{\sigma \in S_3} \tr\Big( (A-z\operatorname{Id})^{-1} \frac{\partial A(x)}{\partial x_{r_{\sigma(1)}}} (A-z\operatorname{Id})^{-1} \frac{\partial A(x)}{\partial x_{r_{\sigma(2)}}} (A-z\operatorname{Id})^{-1} \frac{\partial A(x)}{\partial x_{r_{\sigma(3)}}} (A-z\operatorname{Id})^{-1} \Big)\\
	& \hspace{1cm} + \frac{1}{d} \sum\limits_{\substack{\alpha \sqcup \beta = \{1,2,3\} \\ \alpha,\beta \neq \emptyset}} \tr\Big( (A-z\operatorname{Id})^{-1} \frac{\partial^{\#\alpha}A(x)}{\partial x_{r_{\alpha}}} (A-z\operatorname{Id})^{-1} \frac{\partial^{\#\beta}A(x)}{\partial x_{r_{\beta}}} (A-z\operatorname{Id})^{-1} \Big)\\
	& \hspace{1cm} - \frac{1}{d} \tr\Big( (A-z\operatorname{Id})^{-1} \underbrace{\frac{\partial^3A(x)}{\partial x_{r_1} \, \partial x_{r_2} \, \partial x_{r_3}}}_{=0} (A-z\operatorname{Id})^{-1} \Big) \ .
\end{align*}
The trace bound (\ref{Eq_TraceBound_Neumann}) with (\ref{Eq_SpectralDerivativeBound1})-(\ref{Eq_SpectralResolventBound}) then yields
%\begin{align*}
%& \Big| \frac{d}{dx_r} \cs_{B(x)B^*(x)}(z) \Big| \leq \frac{K}{d} \frac{1}{\Im(z)} 2||B(x)||\kappa \frac{1}{\Im(z)} = \frac{2K||B(x)||\kappa}{d \Im(z)^2} \ ,
%\end{align*}
%\begin{align*}
%\Big| \frac{d^2}{dx_{r_1} \, dx_{r_2}} \cs_{B(x)B^*(x)}(z) \Big| & \leq \frac{2K}{d} \frac{(2||B(x)||\kappa)^2}{\Im(z)^3} + \frac{K}{d} \frac{2\kappa^2}{\Im(z)^2} = \frac{8K ||B(x)||^2 \kappa^2}{d \Im(z)^3} + \frac{2K \kappa^2}{d \Im(z)^2}
%\end{align*}
%and
\begin{align*}
	& \Big| \frac{\partial^3}{\partial x_{r_1} \, \partial x_{r_2} \, \partial x_{r_3}} \cs_{B(x)B^*(x)}(z) \Big|\\
	& \leq \frac{6K}{d} \frac{(2||B(x)||\kappa)^3}{\Im(z)^4} + \frac{3K}{d} \frac{2\kappa^2 \, 2||B(x)||\kappa}{\Im(z)^3} = \frac{48 K ||B(x)||^3 \kappa^3}{d \Im(z)^4} + \frac{12 K ||B(x)|| \kappa^3}{d \Im(z)^3} \ . \qed
\end{align*}

\subsection{Proof of Lemma \ref{Lemma_ResolventDerivatives}}\label{Proof_Lemma_ResolventDerivatives}
The set $M \coloneq \{z \in \C \mid \det(A-z\operatorname{Id}) \neq 0\}$ is open, connected and both sides of the equality are analytic on the set, which means it suffices to show the equality on an open subset. For all $z \in M$ large enough that $||z^{-1}A||<1$ we have
\begin{align*}
	& (A-z\operatorname{Id})^{-1} = -z^{-1} (\operatorname{Id}-z^{-1}A)^{-1} = -z^{-1} \sum\limits_{k=0}^\infty z^{-k} A^{k} \ ,
\end{align*}
which yields
\begin{align*}
	& \frac{\partial }{\partial A_{i,j}} (A-z\operatorname{Id})^{-1} = -z^{-1} \sum\limits_{k=0}^\infty z^{-k} \frac{\partial}{\partial A_{i,j}} A^{k}\\
	& = -z^{-1} \sum\limits_{k=0}^\infty z^{-k} \sum\limits_{\substack{r=1}}^k A \cdots A \cdot \underbrace{\mathrm{u}_i^{(d)} (\mathrm{u}_j^{(d)})^\top}_{r \text{-th pos.}} \cdot A \cdots A\\
	& = -z^{-2} \sum\limits_{k=0}^\infty \sum\limits_{\substack{r=1}}^k (z^{-1}A) \cdots (z^{-1}A) \cdot \underbrace{\mathrm{u}_i^{(d)} (\mathrm{u}_j^{(d)})^\top}_{r \text{-th pos.}} \cdot (z^{-1}A) \cdots (z^{-1}A)\\
	& = -z^{-2} (\operatorname{Id}-z^{-1}A)^{-1} \mathrm{u}_i^{(d)} (\mathrm{u}_j^{(d)})^\top (\operatorname{Id}-z^{-1}A)^{-1}\\
	& = - (A-z\operatorname{Id})^{-1} \mathrm{u}_i^{(d)} (\mathrm{u}_j^{(d)})^\top (A-z\operatorname{Id})^{-1} \ .
\end{align*}
%\begin{align*}
%& \frac{d^L}{dA_{i_1,j_1} \cdots d_{A_{i_L,j_L}}} (A-z\operatorname{Id})^{-1} = -z^{-1} \sum\limits_{k=0}^\infty z^{-k} \frac{d^L}{dA_{i_1,j_1} \cdots d_{A_{i_L,j_L}}} A^{k}\\
%& = -z^{-1} \sum\limits_{k=0}^\infty z^{-k} \sum\limits_{\substack{r_1,...,r_L=1 \\ r_q \neq r_p, \forall q \neq p}}^k A \cdots A \cdot \underbrace{e_{i_1} e_{j_1}^\top}_{r_1 \text{-th pos.}} \cdot A \cdots A \underbrace{e_{i_2} e_{j_2}^\top}_{r_2 \text{-th pos.}} \cdot A \cdots \cdots A\\
%& = -z^{-1-L} \sum\limits_{k=0}^\infty \sum\limits_{\substack{r_1,...,r_L=1 \\ r_q \neq r_p, \forall q \neq p}}^k (z^{-1}A) \cdots (z^{-1}A) \cdot \underbrace{e_{i_1} e_{j_1}^\top}_{r_1 \text{-th pos.}} \cdot (z^{-1}A) \cdots (z^{-1}A) \underbrace{e_{i_2} e_{j_2}^\top}_{r_2 \text{-th pos.}} \cdot (z^{-1}A) \cdots \cdots (z^{-1}A)\\
%& = -z^{-1-L} \hspace{-0.5cm} \sum\limits_{\sigma : \{1,...,L\} \hookrightarrow \{1,...,L\}} \hspace{-0.5cm} (\operatorname{Id}-z^{-1}A)^{-1} e_{i_{\sigma(1)}} e_{j_{\sigma(1)}}^\top (\operatorname{Id}-z^{-1}A)^{-1} e_{i_{\sigma(2)}} e_{j_{\sigma(2)}}^\top (\operatorname{Id}-z^{-1}A)^{-1} \cdots\\
%& \hspace{5cm} \cdots e_{i_{\sigma(L)}} e_{j_{\sigma(L)}}^\top (\operatorname{Id}-z^{-1}A)^{-1}\\
%& = (-1)^L \sum\limits_{\sigma \in S_L} (A-z\operatorname{Id})^{-1} e_{i_{\sigma(1)}} e_{j_{\sigma(1)}}^\top (A-z\operatorname{Id})^{-1} \cdots e_{i_{\sigma(L)}} e_{j_{\sigma(L)}}^\top (A-z\operatorname{Id})^{-1} \ .
%\end{align*}
The previous observation for $L=1$ with chain rule yields
\begin{align*}
	& \frac{\partial}{\partial x_{k_1}} (A(x)-z\operatorname{Id})^{-1} = \sum\limits_{i,j=1}^d \frac{\partial A(x)}{\partial x_{k_1}} \frac{\partial}{\partial A_{i,j}} (A(x)-z\operatorname{Id})^{-1} = - (A-z\operatorname{Id})^{-1} \frac{\partial A(x)}{\partial x_{k_1}} (A-z\operatorname{Id})^{-1} \ ,
\end{align*}
which proves (\ref{Eq_A_derivative1}). To prove (\ref{Eq_A_derivative2}) we then see
\begin{align*}
	& \frac{\partial^2}{\partial x_{k_1} \, \partial x_{k_2}} (A(x)-z\operatorname{Id})^{-1} = - \frac{\partial}{\partial x_{k_2}} (A-z\operatorname{Id})^{-1} \frac{\partial A(x)}{\partial x_{k_1}} (A-z\operatorname{Id})^{-1}\\
	& = - \Big[ \frac{\partial}{\partial x_{k_2}} (A-z\operatorname{Id})^{-1} \Big] \frac{\partial A(x)}{\partial x_{k_1}} (A-z\operatorname{Id})^{-1} - (A-z\operatorname{Id})^{-1} \frac{\partial A(x)}{\partial x_{k_1}} \Big[ \frac{\partial}{\partial x_{k_2}} (A-z\operatorname{Id})^{-1} \Big]\\
	& \hspace{1cm} - (A-z\operatorname{Id})^{-1} \frac{\partial^2A(x)}{\partial x_{k_1} \, \partial x_{k_2}} (A-z\operatorname{Id})^{-1}\\
	& = (A-z\operatorname{Id})^{-1} \frac{\partial A(x)}{\partial x_{k_2}} (A-z\operatorname{Id})^{-1} \frac{\partial A(x)}{\partial x_{k_1}} (A-z\operatorname{Id})^{-1}\\
	& \hspace{1cm} + (A-z\operatorname{Id})^{-1} \frac{\partial A(x)}{\partial x_{k_1}} (A-z\operatorname{Id})^{-1} \frac{\partial A(x)}{\partial x_{k_2}} (A-z\operatorname{Id})^{-1}\\
	& \hspace{1cm} - (A-z\operatorname{Id})^{-1} \frac{\partial^2A(x)}{\partial x_{k_1} \, \partial x_{k_2}} (A-z\operatorname{Id})^{-1} \ .
	%& = \sum\limits_{\sigma \in S_2} (A-z\operatorname{Id})^{-1} \frac{dA(x)}{dx_{k_{\sigma(1)}}} (A-z\operatorname{Id})^{-1} \frac{dA(x)}{dx_{k_{\sigma(2)}}} (A-z\operatorname{Id})^{-1}\\
	%& \hspace{1cm} - (A-z\operatorname{Id})^{-1} \frac{d^2A(x)}{dx_{k_1} \, dx_{k_2}} (A-z\operatorname{Id})^{-1} \ .
\end{align*}
Analogously, we for (\ref{Eq_A_derivative3}) see
\begin{align*}
	& \frac{\partial^3}{\partial x_{k_1} \, \partial x_{k_2} \, \partial x_{k_3}} (A(x)-z\operatorname{Id})^{-1}\\
	& = \frac{\partial}{\partial x_{k_3}} (A-z\operatorname{Id})^{-1} \frac{\partial A(x)}{\partial x_{k_2}} (A-z\operatorname{Id})^{-1} \frac{\partial A(x)}{\partial x_{k_1}} (A-z\operatorname{Id})^{-1}\\
	& \hspace{1cm} + \frac{\partial }{\partial x_{k_3}} (A-z\operatorname{Id})^{-1} \frac{\partial A(x)}{\partial x_{k_1}} (A-z\operatorname{Id})^{-1} \frac{\partial A(x)}{\partial x_{k_2}} (A-z\operatorname{Id})^{-1}\\
	& \hspace{1cm} - \frac{\partial}{\partial x_{k_3}} (A-z\operatorname{Id})^{-1} \frac{\partial^2A(x)}{d\partial _{k_1} \, d\partial _{k_2}} (A-z\operatorname{Id})^{-1}\\
	%%%
	& = - \sum\limits_{\sigma \in S_3} (A-z\operatorname{Id})^{-1} \frac{\partial A(x)}{\partial x_{k_{\sigma(1)}}} (A-z\operatorname{Id})^{-1} \frac{\partial A(x)}{\partial x_{k_{\sigma(2)}}} (A-z\operatorname{Id})^{-1} \frac{\partial A(x)}{\partial x_{k_{\sigma(3)}}} (A-z\operatorname{Id})^{-1}\\
	& \hspace{1cm} + \sum\limits_{\substack{\alpha \sqcup \beta = \{1,2,3\} \\ \alpha,\beta \neq \emptyset}} (A-z\operatorname{Id})^{-1} \frac{\partial ^{\#\alpha}A(x)}{\partial x_{k_\alpha}} (A-z\operatorname{Id})^{-1} \frac{\partial ^{\#\beta}A(x)}{\partial x_{k_\beta}} (A-z\operatorname{Id})^{-1}\\
	& \hspace{1cm} - (A-z\operatorname{Id})^{-1} \frac{\partial^3A(x)}{\partial x_{k_1} \, \partial x_{k_2} \, \partial x_{k_3}} (A-z\operatorname{Id})^{-1} \ . \qed
	%%%
\end{align*}

\subsection{Proof of Lemma \ref{Lemma_CrudeTraceMomentBound}}\label{Proof_Lemma_CrudeTraceMomentBound}
The first bound is easily calculated with
\begin{align*}
	& \E\big[ ||\bm{B}||^{2} \big] \leq \E\big[ \tr\big(\bm{B}\bm{B}^*\big) \big] = \sum\limits_{j=1}^d \sum\limits_{s=1}^n \E\big[ |\bm{B}_{j,s}|^2 \big]\\
	& = \frac{1}{2m+1} \sum\limits_{j=1}^d \sum\limits_{s=1}^n \E\big[ |(\bm{X}^{(K)} V D_r^{\frac{1}{2}})_{j,s}|^2 \big]\\
	& = \frac{1}{2m+1} \sum\limits_{j=1}^d \sum\limits_{\substack{s=1 \\ \rho_n(r,s-1)\leq m}}^n \E\big[ \bigg|\sum\limits_{t=0}^{n-1} \bm{X}^{(K)}_{j,t+1} V_{t+1,s}\bigg|^2 \big]\\
	& = \frac{1}{2m+1} \sum\limits_{j=1}^d \sum\limits_{\substack{s=1 \\ \rho_n(r,s-1)\leq m}}^n \E\bigg[ \bigg|\sum\limits_{t=0}^{n-1} \sum\limits_{k=0}^K (\Psi_{k} \eta_{t-k})_j V_{t+1,s}\bigg|^2 \bigg]\\
	& = \frac{1}{2m+1} \sum\limits_{j=1}^d \sum\limits_{\substack{s=1 \\ \rho_n(r,s-1)\leq m}}^n \sum\limits_{t_1,t_2=0}^{n-1} \sum\limits_{k_1,k_2=0}^K V_{t_1+1,s} \ol{V}_{t_2+1,s} \E\big[ (\mathrm{u}_j^{(d)})^\top \Psi_{k_1} \eta_{t_1-k_1} \eta_{t_2-k_2}^\top \Psi_{k_1}^\top \mathrm{u}_j^{(d)} \big]\\
	& = \frac{1}{2m+1} \sum\limits_{j=1}^d \sum\limits_{\substack{s=1 \\ \rho_n(r,s-1)\leq m}}^n \sum\limits_{k_1,k_2=0}^K \underbrace{\bigg(\sum\limits_{t_1=0}^{n-1} V_{t_1-k_1+k_2+1,s} \ol{V}_{t_1-k_1+k_2+1,s}\bigg)}_{= 1} (\mathrm{u}_j^{(d)})^\top \Psi_{k_1} \Psi_{k_1}^\top \mathrm{u}_j^{(d)}\\
	& = \frac{1}{2m+1} \sum\limits_{\substack{s=1 \\ \rho_n(r,s-1)\leq m}}^n \tr\bigg( \Big( \sum\limits_{k=0}^K \Psi_k \Big) \Big( \sum\limits_{k=0}^K \Psi_k \Big)^\top \bigg) \overset{\text{(\ref{Eq_TraceBound_Neumann})}}{\leq} d \bigg( \sum\limits_{k=0}^K ||\Psi_k|| \bigg) \overset{\text{(\ref{Eq_Assumption_LongRandeDependence_cK2})}}{\leq} \cK_2^2 d \ .
\end{align*}
For the second bound we also make use of $\E\big[ ||\bm{B}||^{6} \big] \leq \E\big[ \tr\big((\bm{B}\bm{B}^\top)^3\big) \big]$ and expand the trace to
\begin{align}\label{Eq_Univ_TraceExpansion}
	& \E\big[ \tr\big((\bm{B}\bm{B}^*)^3\big) \big] \nonumber\\
	& = \sum\limits_{j_1,j_2,j_3=1}^d \sum\limits_{s_1,s_2,s_3=1}^n \E\big[ \bm{B}_{j_1,s_1} \ol{\bm{B}}_{j_2,s_1} \cdot \bm{B}_{j_2,s_2} \ol{\bm{B}}_{j_3,s_2} \cdot \bm{B}_{j_3,s_3} \ol{\bm{B}}_{j_1,s_3} \big] \nonumber\\
	& = \frac{1}{(2m+1)^3} \sum\limits_{j_1,j_2,j_3=1}^d \sum\limits_{\substack{s_1,s_2,s_3=1 \\ \rho_n(r,s_\bullet) \leq m}}^n \E\big[ (\bm{X}^{(K)} V)_{j_1,s_1} (\bm{X}^{(K)} \ol{V})_{j_2,s_1} \cdot (\bm{X}^{(K)} V)_{j_2,s_2} (\bm{X}^{(K)} \ol{V})_{j_3,s_2} \nonumber\\
	& \hspace{8cm} \cdot (\bm{X}^{(K)} V)_{j_3,s_3} (\bm{X}^{(K)} \ol{V})_{j_1,s_3} \big] \nonumber\\
	& = \frac{1}{(2m+1)^3} \sum\limits_{\substack{s_1,s_2,s_3=1 \\ \rho_n(r,s_\bullet) \leq m}}^n \sum\limits_{t_1,...,t_6=0}^{n-1} V_{t_1+1,s_1} \ol{V}_{t_4+1,s_1} \cdot V_{t_2+1,s_2} \ol{V}_{t_5+1,s_2} \cdot V_{t_3+1,s_3} \ol{V}_{t_6+1,s_3} \nonumber\\
	& \hspace{2cm} \times \sum\limits_{j_1,j_2,j_3=1}^d \E\big[ \bm{X}^{(K)}_{j_1,t_1+1} \bm{X}^{(K)}_{j_2,t_4+1} \cdot \bm{X}^{(K)}_{j_2,t_2+1} \bm{X}^{(K)}_{j_3,t_5+1} \cdot \bm{X}^{(K)}_{j_3,t_3+1} \bm{X}^{(K)}_{j_1,t_6+1} \big] \ .
\end{align}
and further expand
\begin{align}\label{Eq_CrudeTB_MeanExpansion1}
	& \sum\limits_{j_1,j_2,j_3=1}^d \E\big[ \bm{X}^{(K)}_{j_1,t_1+1} \bm{X}^{(K)}_{j_2,t_4+1} \cdot \bm{X}^{(K)}_{j_2,t_2+1} \bm{X}^{(K)}_{j_3,t_5+1} \cdot \bm{X}^{(K)}_{j_3,t_3+1} \bm{X}^{(K)}_{j_1,t_6+1} \big]\\
	& = \sum\limits_{j_1,j_2,j_3=1}^d \sum\limits_{k_1,...,k_6=0}^K \E\big[ (\Psi_{k_1} \eta_{t_1-k_1})_{j_1} (\Psi_{k_4} \eta_{t_4-k_4})_{j_2} \cdot (\Psi_{k_2} \eta_{t_2-k_2})_{j_2} (\Psi_{k_5} \eta_{t_5-k_5})_{j_3} \nonumber\\
	& \hspace{6cm} \cdot (\Psi_{k_3} \eta_{t_3-k_3})_{j_3} (\Psi_{k_6} \eta_{t_6-k_6})_{j_1} \big]\\
	%%%
	& = \sum\limits_{j_1,j_2,j_3=1}^d \sum\limits_{k_1,...,k_6=0}^K \sum\limits_{i_1,...,i_6=1}^d \E\bigg[ \prod\limits_{q=1}^6 (\eta_{t_q-k_q})_{i_q} \bigg] \nonumber\\
	& \hspace{1cm} \times (\Psi_{k_1})_{j_1,i_1} (\Psi_{k_4})_{j_2,i_4} \cdot (\Psi_{k_2})_{j_2,i_2} (\Psi_{k_5})_{j_3,i_5} \cdot (\Psi_{k_3})_{j_3,i_3} (\Psi_{k_6})_{j_1,i_6}
\end{align}
Note that the mean $\E\big[ \prod\limits_{q=1}^6 (\eta_{t_q-k_q})_{i_q} \big]$ is zero unless the indexes $(\tilde{t}_1,i_1),...,(\tilde{t}_6,i_6)$ come in pairs, where
\begin{align}
	& \tilde{t}_q \coloneq t_q-k_q \label{Eq_Def_tilde_t} \ .
\end{align}
As in Lemma \ref{Lemma_WicksFormula}, let $\Pi_2(6)$ be the set of parings of $\{1,...,6\}$.
With the notation
\begin{align}\label{Eq_Def_Xi}
	& \Xi_{t} \coloneq \diag\big( (\eta_t)_1,...,(\eta_t)_d \big)
\end{align}
and by reverting the summation over $j_1,j_2,j_3$ back into matrix products, we turn (\ref{Eq_CrudeTB_MeanExpansion1}) into
\begin{align}\label{Eq_CrudeTB_MeanExpansion2}
	& \sum\limits_{j_1,j_2,j_3=1}^d \E\big[ \bm{X}^{(K)}_{j_1,t_1+1} \bm{X}^{(K)}_{j_2,t_4+1} \cdot \bm{X}^{(K)}_{j_2,t_2+1} \bm{X}^{(K)}_{j_3,t_5+1} \cdot \bm{X}^{(K)}_{j_3,t_3+1} \bm{X}^{(K)}_{j_1,t_6+1} \big] \nonumber\\
	& = \sum\limits_{k_1,...,k_6=0}^K \sum\limits_{i_1,...,i_6=1}^d \mathbbm{1}_{\exists \bm{\pi} \in \Pi_2(6) , \, \forall \{q,q'\} \in \bm{\pi} : \, (\tilde{t}_q,i_q)=(\tilde{t}_{q'},i_{q'})} \, \E\Big[ (\Xi_{t_1-k_1} \Psi_{k_1}^\top \Psi_{k_6} \Xi_{t_6-k_6})_{i_1,i_6} \nonumber\\
	& \hspace{1cm} \times (\Xi_{t_2-k_2} \Psi_{k_2}^\top \Psi_{k_4} \Xi_{t_4-k_4})_{i_2,i_4} (\Xi_{t_3-k_3} \Psi_{k_3}^\top \Psi_{k_5} \Xi_{t_5-k_5})_{i_3,i_5} \Big] \ .
\end{align}
Plugging (\ref{Eq_CrudeTB_MeanExpansion2}) back into (\ref{Eq_Univ_TraceExpansion}) and changing the order of summation gives
\begin{align}\label{Eq_Univ_TraceExpansion2}
	& \E\big[ \tr\big((\bm{B}\bm{B}^*)^3\big) \big] \nonumber\\
	& = \frac{1}{(2m+1)^3} \sum\limits_{k_1,...,k_6=0}^K \sum\limits_{i_1,...,i_6=1}^d \sum\limits_{\substack{s_1,...,s_6=1\\ s_\bullet=s_{\bullet \pm 3} \\ \rho_n(r,s_\bullet) \leq m}}^n \sum\limits_{t_1,...,t_6=0}^{n-1} \mathbbm{1}_{\exists \bm{\pi} \in \Pi_2(6) , \, \forall \{q,q'\} \in \bm{\pi} : \, (\tilde{t}_q,i_q)=(\tilde{t}_{q'},i_{q'})}\nonumber\\
	& \hspace{0.5cm} \times V_{t_1+1,s_1} \ol{V}_{t_4+1,s_4} \cdot V_{t_2+1,s_2} \ol{V}_{t_5+1,s_5} \cdot V_{t_3+1,s_3} \ol{V}_{t_6+1,s_6} \nonumber\\
	& \hspace{1cm} \times \E\Big[ (\Xi_{t_1-k_1} \Psi_{k_1}^\top \Psi_{k_6} \Xi_{t_6-k_6})_{i_1,i_6}  (\Xi_{t_2-k_2} \Psi_{k_2}^\top \Psi_{k_4} \Xi_{t_4-k_4})_{i_2,i_4} \nonumber\\
	& \hspace{6cm} \times (\Xi_{t_3-k_3} \Psi_{k_3}^\top \Psi_{k_5} \Xi_{t_5-k_5})_{i_3,i_5} \Big] \ .
\end{align}
The error made when instead summing over all pairings $\bm{\pi} \in \Pi_2(6)$ is bounded by
\begin{align}\label{Eq_Crude_PairingErrorBound1}
	& \bigg| \E\big[ \tr\big((\bm{B}\bm{B}^*)^3\big) \big] - \sum\limits_{\bm{\pi} \in \Pi_2(6)} \frac{1}{(2m+1)^3} \sum\limits_{k_1,...,k_6=0}^K \sum\limits_{i_1,...,i_6=1}^d \sum\limits_{\substack{s_1,...,s_6=1\\ s_\bullet=s_{\bullet \pm 3} \\ \rho_n(r,s_\bullet) \leq m}}^n \nonumber\\
	& \hspace{3.5cm} \sum\limits_{t_1,...,t_6=0}^{n-1} \mathbbm{1}_{\forall \{q,q'\} \in \bm{\pi} : \, (\tilde{t}_q,i_q)=(\tilde{t}_{q'},i_{q'})} \, V_{t_1+1,s_1} \cdots \ol{V}_{t_6+1,s_6} \E\big[ ... \big] \bigg| \nonumber\\
	& \leq \sum\limits_{\substack{\bm{\pi}_1, \bm{\pi_2} \in \Pi_2(6) \\ \bm{\pi}_1 \neq \bm{\pi_2}}} \frac{1}{(2m+1)^3} \sum\limits_{k_1,...,k_6=0}^K \sum\limits_{i_1,...,i_6=1}^d \sum\limits_{\substack{s_1,...,s_6=1\\ s_\bullet=s_{\bullet \pm 3} \\ \rho_n(r,s_\bullet) \leq m}}^n \nonumber\\
	& \hspace{1.5cm} \sum\limits_{t_1,...,t_6=0}^{n-1} \mathbbm{1}_{\forall \{q,q'\} \in \bm{\pi}_1 \cup \bm{\pi}_2 : \, (\tilde{t}_q,i_q)=(\tilde{t}_{q'},i_{q'})} \, \Big|V_{t_1+1,s_1} \cdots \ol{V}_{t_6+1,s_6}\Big| \, \Big|\E\big[ ... \big]\Big| \ .
	%& \leq \sum\limits_{\substack{\bm{\pi}_1, \bm{\pi_2} \in \Pi_2(6) \\ \bm{\pi}_1 \neq \bm{\pi_2}}} \frac{1}{(2m+1)^3n^3} \sum\limits_{k_1,...,k_6=0}^K \sum\limits_{i_1,...,i_6=1}^d \sum\limits_{\substack{s_1,...,s_6=1\\ s_\bullet=s_{\bullet \pm 3} \\ \rho_n(r,s_\bullet) \leq m}}^n \nonumber\\
	%& \hspace{1.5cm} \sum\limits_{t_1,...,t_6=0}^{n-1} \mathbbm{1}_{\forall \{q,q'\} \in \bm{\pi}_1 \cup \bm{\pi}_2 : \, (\tilde{t}_q,i_q)=(\tilde{t}_{q'},i_{q'})} \, d^2 ||\Psi_{k_1}||\cdots||\Psi_{k_6}||\\
\end{align}
Taking advantage of the fact that the two pairings $\bm{\pi}_1$ and $\bm{\pi}_2$ together split $\{1,...,6\}$ into at most two cycles, we bound the degrees of freedom in the sums over $i_\bullet$ and $t_\bullet$ to see
\begin{align*}
	& \text{(\ref{Eq_Crude_PairingErrorBound1})} \leq \sum\limits_{\substack{\bm{\pi}_1, \bm{\pi_2} \in \Pi_2(6) \\ \bm{\pi}_1 \neq \bm{\pi_2}}} \frac{1}{(2m+1)^3n^3} \sum\limits_{k_1,...,k_6=0}^K \sum\limits_{\substack{s_1,...,s_6=1\\ s_\bullet=s_{\bullet \pm 3} \\ \rho_n(r,s_\bullet) \leq m}}^n \nonumber\\
	& \hspace{1.5cm} \sum\limits_{t_1,...,t_6=0}^{n-1} \mathbbm{1}_{\forall \{q,q'\} \in \bm{\pi}_1 \cup \bm{\pi}_2 : \, (\tilde{t}_q,i_q)=(\tilde{t}_{q'},i_{q'})} \, d^2 ||\Psi_{k_1}||\cdots||\Psi_{k_6}|| \sup\limits_{t \in \Z} \E\big[ \max_{j\leq d} |(\eta_{t})_j|^6 \big]\\
	& \leq \sum\limits_{\substack{\bm{\pi}_1, \bm{\pi_2} \in \Pi_2(6) \\ \bm{\pi}_1 \neq \bm{\pi_2}}} \frac{1}{(2m+1)^3n^3} \bigg( \underbrace{\sum\limits_{k=0}^\infty ||\Psi_k||}_{\leq \cK_2} \bigg)^6 (2m+1)^2 \, n^2 \, d^2 \sup\limits_{t \in \Z} \E\big[ \max_{j\leq d} |(\eta_{t})_j|^6 \big]\\
	& \leq (\#\Pi_2(6))^2 \sup\limits_{t \in \Z} \E\big[ \max_{j\leq d} |(\eta_{t})_j|^6 \big] \cK_2^6 \frac{d^2}{(2m+1)n} = 15^2 \sup\limits_{t \in \Z} \E\big[ \max_{j\leq d} |(\eta_{t})_j|^6 \big] \cK_2^6 \frac{d^2}{(2m+1)n} \ ,
\end{align*}
where we have used (\ref{Eq_TraceBound_Neumann}) to bound the mean and the sum over $i_\bullet$ by
\begin{align*}
	& d^2 ||\Psi_{k_1}||\cdots||\Psi_{k_6}|| \, \sup\limits_{t \in \Z} \E\big[ \max_{j\leq d} |(\eta_{t})_j|^6 \big] \ .
\end{align*}
We have thus shown
\begin{align}\label{Eq_Univ_TraceExpansion3}
	& \E\big[ \tr\big((\bm{B}\bm{B}^*)^3\big) \big] - R \nonumber\\
	& = \frac{1}{(2m+1)^3} \sum\limits_{\bm{\pi} \in \Pi_2(6)} \sum\limits_{k_1,...,k_6=0}^K \sum\limits_{\substack{s_1,...,s_6=1\\ s_\bullet=s_{\bullet \pm 3} \\ \rho_n(r,s_\bullet) \leq m}}^n \nonumber\\
	& \hspace{0.5cm} \sum\limits_{\substack{t_1,...,t_6=1 \\ t_q-k_q=t_{q'}-k_{q'}, \, \forall \{q,q'\} \in \bm{\pi}}}^{n-1} V_{t_1+1,s_1} \ol{V}_{t_4+1,s_4} \cdot V_{t_2+1,s_2} \ol{V}_{t_5+1,s_5} \cdot V_{t_3+1,s_3} \ol{V}_{t_6+1,s_6} \nonumber\\
	& \hspace{1cm} \times \sum\limits_{\substack{i_1,...,i_6=1 \\ i_q=i_{q'}, \, \forall \{q,q'\} \in \bm{\pi}}}^d \E\Big[ (\Xi_{t_1-k_1} \Psi_{k_1}^\top \Psi_{k_6} \Xi_{t_6-k_6})_{i_1,i_6}  (\Xi_{t_2-k_2} \Psi_{k_2}^\top \Psi_{k_4} \Xi_{t_4-k_4})_{i_2,i_4} \nonumber\\
	& \hspace{6cm} \times (\Xi_{t_3-k_3} \Psi_{k_3}^\top \Psi_{k_5} \Xi_{t_5-k_5})_{i_3,i_5} \Big]
\end{align}
for an error term $R$ satisfying $|R| \leq 15^2 \sup\limits_{t \in \Z} \E\big[ \max_{j\leq d} |(\eta_{t})_j|^6 \big] \cK_2^6 \frac{d^2}{(2m+1)n}$. We are now in position to use the construction of the matrix $V$. Note that final sum on the right hand side of (\ref{Eq_Univ_TraceExpansion3}) does not depend on the exact values of $t_1,...,t_6$ but only on the equality structure between $t_1-k_1,...,t_6-k_6$. Also note that the final sum on the right hand side of (\ref{Eq_Univ_TraceExpansion3}) can again by (\ref{Eq_TraceBound_Neumann}) be bounded by $\E\big[ \max_{j\leq d} |(\eta_{t})_j|^6 \big] d^{I(\bm{\pi})} ||\Psi_{k_1}|| \cdots ||\Psi_{k_6}||$, where $I(\bm{\pi})$ denotes the number of cycles that $\{1,...,6\}$ is split into by the two pairings $\bm{\pi}$ and $\{\{1,6\},\{2,4\},\{3,5\}\}$. The observation
\begin{align*}
	& \bigg|\sum\limits_{t=0}^{n-1} e^{-2\pi i \frac{ts+(t+k)s'}{n}}\bigg| = \bigg|e^{-2\pi i \frac{ks'}{n}} \sum\limits_{t=0}^{n-1} e^{-2\pi i \frac{t(s+s')}{n}}\bigg| = \mathbbm{1}_{s=s' \mod n} \, n
\end{align*}
then yields
\begin{align*}
	\text{(\ref{Eq_Univ_TraceExpansion3})} & \leq \frac{1}{(2m+1)^3} \sum\limits_{\bm{\pi} \in \Pi_2(6)} \sum\limits_{k_1,...,k_6=0}^K (2m+1)^{J(\bm{\pi})} \frac{n^3}{n^3} d^{I(\bm{\pi})} ||\Psi_{k_1}|| \cdots ||\Psi_{k_6}|| \sup\limits_{t \in \Z} \E\big[ \max_{j\leq d} |(\eta_{t})_j|^6 \big]\\
	& \leq \sup\limits_{t \in \Z} \E\big[ \max_{j\leq d} |(\eta_{t})_j|^6 \big] \cK_2^6 \frac{1}{(2m+1)^3} \sum\limits_{\bm{\pi} \in \Pi_2(6)} (2m+1)^{J(\bm{\pi})} d^{I(\bm{\pi})} \ ,
\end{align*}
where $J(\bm{\pi})$ denotes the number of cycles that $\{1,...,6\}$ is split into by the two pairings $\bm{\pi}$ and $\{\{1,4\},\{2,5\},\{3,6\}\}$. Since the event $I(\bm{\pi}) = 3$ already implies $J(\bm{\pi}) = 1$ and vice-versa, we thus have
\begin{align*}
	\text{(\ref{Eq_Univ_TraceExpansion3})} & \leq \sup\limits_{t \in \Z} \E\big[ \max_{j\leq d} |(\eta_{t})_j|^6 \big] \cK_2^6 \#\Pi_2(6) \frac{\max(d, 2m+1)^4}{(2m+1)^3}\\
	& = 15 \sup\limits_{t \in \Z} \E\big[ \max_{j\leq d} |(\eta_{t})_j|^6 \big] \cK_2^6 \frac{\max(d, 2m+1)^4}{(2m+1)^3} \ ,
\end{align*}
which by the bound on $R$ gives
\begin{align*}
	& \E\big[ ||\bm{B}||^{6} \big] \leq \E\big[ \tr\big((\bm{B}\bm{B}^*)^3\big) \big]\\
	& \leq 15 \sup\limits_{t \in \Z} \E\big[ \max_{j\leq d} |(\eta_{t})_j|^6 \big] \cK_2^6 \frac{\max(d, 2m+1)^4}{(2m+1)^3} + 15^2 \sup\limits_{t \in \Z} \E\big[ \max_{j\leq d} |(\eta_{t})_j|^6 \big] \cK_2^6 \frac{d^2}{(2m+1)n} \ . \qed
\end{align*}

\subsection{Proof of Lemma \ref{Lemma_Approximation_FiniteLinear}}\label{Proof_Lemma_Approximation_FiniteLinear}
For any $\tau \in [0,1]$ define
\begin{align*}
	& X^{(K,\tau)}_t \coloneq \tau X_t + (1-\tau) X^{(K)}_t = \sum\limits_{k=0}^K \Psi_{k} \Sigma^{\frac{1}{2}} \xi_{t-k} + \tau \sum\limits_{k=K+1}^\infty \Psi_{k} \Sigma^{\frac{1}{2}} \xi_{t-k} \ .
\end{align*}
With $\bm{X}^{(K,\tau)} = [X^{(K,\tau)}_0,...,X^{(K,\tau)}_{n-1}] = \tau \bm{X} + (1-\tau) \bm{X}^{(K)}$ we analogously define
\begin{align*}
	& S^{(K,\tau)}\Big( \frac{2\pi r}{n} \Big) = \frac{1}{2m+1} \bm{X}^{(K,\tau)} V D_r V^* (\bm{X}^{(K,\tau)})^\top \ .
\end{align*}
By the observation
\begin{align*}
	& \big| \cs_{\hat{\mu}_{S(\frac{2\pi r}{n})}}(z) - \cs_{\hat{\mu}_{S^{(K)}(\frac{2\pi r}{n})}}(z) \big| = \big| \cs_{\hat{\mu}_{S^{(K,1)}(\frac{2\pi r}{n})}}(z) - \cs_{\hat{\mu}_{S^{(K,0)}(\frac{2\pi r}{n})}}(z) \big|\\
	& = \bigg| \int_0^1 \frac{\partial}{\partial \tau} \cs_{\hat{\mu}_{S^{(K,\tau)}(\frac{2\pi r}{n})}}(z) \, d\tau \bigg|
\end{align*}
we can use the bound
\begin{align*}
	& \Big| \frac{\partial}{\partial \tau} \cs_{\hat{\mu}_{S^{(K,\tau)}(\frac{2\pi r}{n})}}(z) \Big| = \frac{1}{d} \Big| \tr\Big( \frac{\partial}{\partial \tau} \Big( S^{(K,\tau)}\Big(\frac{2\pi r}{n}\Big) - z\operatorname{Id}_d \Big)^{-1} \Big) \Big|\\
	& \overset{\text{(\ref{Eq_TraceBound_Neumann})}}{\leq} \Big|\Big| \frac{\partial}{\partial \tau} \Big( S^{(K,\tau)}\Big(\frac{2\pi r}{n}\Big) - z\operatorname{Id}_d \Big)^{-1} \Big|\Big|\\
	& \overset{\text{(\ref{Eq_A_derivative1})}}{=} \Big|\Big| \Big( S^{(K,\tau)}\Big(\frac{2\pi r}{n}\Big) - z\operatorname{Id}_d \Big)^{-1} \Big(\frac{\partial}{\partial \tau} S^{(K,\tau)}\Big(\frac{2\pi r}{n}\Big) \Big) \Big( S^{(K,\tau)}\Big(\frac{2\pi r}{n}\Big) - z\operatorname{Id}_d \Big)^{-1} \Big|\Big|
\end{align*}
and the fact that $||(A-z\operatorname{Id})^{-1}|| \leq \frac{1}{\Im(z)}$ for all Hermitian $A$ and $z \in \C^+$ to see
\begin{align*}
	& \big| \cs_{\hat{\mu}_{S(\frac{2\pi r}{n})}}(z) - \cs_{\hat{\mu}_{S^{(K)}(\frac{2\pi r}{n})}}(z) \big| \leq \frac{1}{\Im(z)^2} \max\limits_{\tau \in [0,1]} \Big|\Big| \frac{\partial}{\partial \tau} S^{(K,\tau)}\Big(\frac{2\pi r}{n}\Big) \Big|\Big|\\
	& = \frac{1}{\Im(z)^2 (2m+1)} \max\limits_{\tau \in [0,1]} \Big|\Big| \underbrace{\Big(\frac{\partial}{\partial \tau} \bm{X}^{(K,\tau)}\Big)}_{= \bm{X} - \bm{X}^{(K)}} V D_r V^* (\bm{X}^{(K,\tau)})^\top + \bm{X}^{(K,\tau)} V D_r V^* \Big(\frac{\partial}{\partial \tau} \bm{X}^{(K,\tau)}\Big)^\top \Big|\Big|\\
	& \leq \frac{2 ||\bm{X}^{(K,\tau)} V D_r^{\frac{1}{2}}||}{\Im(z)^2 (2m+1)} \big|\big| (\bm{X} - \bm{X}^{(K)}) V D_r^{\frac{1}{2}} \big|\big| \ .
\end{align*}
In complete analogy to (\ref{Eq_CrudeTB_Result2}) we have
\begin{align*}
	& \E\big[ ||(\bm{X} - \bm{X}^{(K)})V D_r^{\frac{1}{2}}||^6 \big]\\
	& \leq 15 \underbrace{\sup\limits_{t \in \Z} \E\big[ \max_{j\leq d} |(\xi_{t})_j|^6 \big]}_{\overset{\text{(\ref{Eq_Univ_MeanMax})}}{\leq} d^{\frac{1}{L}} C'_L} \bigg( \sum\limits_{k=K+1}^\infty ||\Psi_k|| \bigg)^6 \Big( \underbrace{\frac{\max(d, 2m+1)^4}{(2m+1)^3} + \frac{15 d^2}{(2m+1)n}}_{\leq C'' d \text{, since $\frac{d}{2m+1} \rightarrow c$ and $n \gg d$}} \Big) \ ,
\end{align*}
which by (\ref{Eq_Assumption_LongRandeDependence_gamma}) means
\begin{align*}
	& \E\big[ ||(\bm{X} - \bm{X}^{(K)})V D_r^{\frac{1}{2}}||^6 \big] \leq 15 C'_LC'' \cK_2^6 d^{1+\frac{1}{L}} K^{-6\gamma} \ .
\end{align*}
Analogously, we also see
\begin{align*}
	& \E\big[ ||\bm{X}^{(K,\tau)}V D_r^{\frac{1}{2}}||^6 \big] \leq 15 C'_LC'' \cK_2^6 d^{1+\frac{1}{L}} \ ,
\end{align*}
which gives
\begin{align*}
	& \E\Big[ \big| \cs_{\hat{\mu}_{S(\frac{2\pi r}{n})}}(z) - \cs_{\hat{\mu}_{S^{(K)}(\frac{2\pi r}{n})}}(z) \big|^6 \Big] \leq \frac{4 \cdot 15 C'_LC'' \cK_2^6 d^{1+\frac{1}{L}}}{\Im(z)^4 (2m+1)^2} 15 C'_LC'' \cK_2^6 d^{1+\frac{1}{L}} K^{-6\gamma}\\
	& \overset{\text{(\ref{Eq_Assumption_Asymptotics})}}{\leq} C(L) \frac{d^{\frac{2}{L}} K^{-6\gamma}}{\Im(z)^4} \ . \qed
\end{align*}

\section*{Funding}\label{Funding}
The author acknowledges the support of the Research Unit 5381 (DFG) RO 3766/8-1.

\bibliographystyle{imsart-nameyear} % Style BST file (imsart-number.bst or imsart-nameyear.bst)
\bibliography{bibliographyWithoutDOI}  
	
\end{document}